\documentclass[opre,nonblindrev]{informs3} 

\usepackage{subfiles}
\usepackage{xr}
\usepackage{float}
\def\biblio{\bibliographystyle{plainnat}\bibliography{../DataDrivenRobustBibliography}}

\usepackage{accents}

\usepackage{natbib}
 \bibpunct[, ]{(}{)}{,}{a}{}{,}%
 \def\bibfont{\small}%
\TheoremsNumberedThrough     
\ECRepeatTheorems

\EquationsNumberedThrough    

\RequirePackage[shortlabels]{enumitem}
\usepackage{mathtools} 
\usepackage{enumitem} 
\usepackage{dsfont} 
\usepackage{mathrsfs}  

\newtheorem{innercustomgeneric}{\customgenericname}
\providecommand{\customgenericname}{}
\newcommand{\newcustomtheorem}[2]{%
	\newenvironment{#1}[1]
	{%
		\renewcommand\customgenericname{#2}%
		\renewcommand\theinnercustomgeneric{##1}%
		\innercustomgeneric
	}
	{\endinnercustomgeneric}
}

\newcustomtheorem{customthm}{Theorem}
\newcustomtheorem{customlemma}{Lemma}

\usepackage[colorlinks=true,linkcolor=black,citecolor=black]{hyperref}

\newcommand{\bc}{{\bf c}}

\newcommand{\bg}{{\bf g}}
\newcommand{\bh}{{\bf h}}
\newcommand{\bp}{{\bf p}}
\newcommand{\bq}{{\bf q}}
\newcommand{\br}{{\bf r}}

\newcommand{\bw}{{\bf w}}
\newcommand{\bx}{{\bf x}}
\newcommand{\by}{{\bf y}}
\newcommand{\bz}{{\bf z}}

\newcommand{\bbg}{{\bf G}}
\newcommand{\bbh}{{\bf H}}

\newcommand{\bbt}{{\bf T}}

\newcommand{\bby}{{\bf Y}}
\newcommand{\bbw}{{\bf W}}

\newcommand{\bdelta}{{\boldsymbol{\delta}}}

\newcommand{\blambda}{{\boldsymbol{\lambda}}}
\newcommand{\bsigma}{{\boldsymbol{\sigma}}}

\newcommand{\brho}{{\boldsymbol{\rho}}}
\newcommand{\bzeta}{{\boldsymbol{\zeta}}}

\newcommand{\bxi}{{\boldsymbol{\xi}}}

\newcommand{\Prb}{\mathbb{P}}  
\newcommand{\Exp}{\mathbb{E}}

\newcommand{\bei}{{\bf e}}
\def\Real{\mathbb{R}}

\let\N\Natural

\let\R\Real

\def\dim{\text{dim}}
\def\dimlower{\underbar{\text{dim}}}

\newcommand{\norm}[1]{\left \lVert #1 \right\|}

\def\conv{\operatorname{conv}}

\newcommand{\bzero}{{\boldsymbol{0}}}

\def\Rfun{\mathcal{R}}

\newcommand{\dist}{\textnormal{dist}}
\newcommand{\Face}{\textnormal{Faces}}
\newcommand{\Facet}{\textnormal{Facets}}
\def\ext{\operatorname{ext}}

\usepackage{graphicx}
\usepackage{bbm}
\usepackage{tikz,ifthen}
\usepackage{fixltx2e}    

   \usetikzlibrary{math}
   \usetikzlibrary{shapes.misc}
   \usetikzlibrary{shapes.multipart,positioning,patterns,backgrounds}

\usepackage[singlelinecheck=false]{caption}
\newcommand{\Piaff}{\mathcal{L}}

\def\relint{\textnormal{ri}}
\usepackage[ruled,vlined,linesnumbered]{algorithm2e}

\usepackage{multirow}
\usepackage{array}
\usepackage{subcaption}

\newcommand{\ie}{{\it{i.e., }}}
\newcommand{\eg}{{\it{e.g.}}, }

\begin{document}
\def\biblio{}



\RUNAUTHOR{Bertsimas, Shtern, Sturt}

\RUNTITLE{Two-Stage Sample Robust Optimization}

\TITLE{Two-Stage Sample Robust Optimization}
\ARTICLEAUTHORS{%
\AUTHOR{Dimitris Bertsimas}
\AFF{Operations Research Center, Massachusetts Institute of Technology, \EMAIL{dbertsim@mit.edu}} 
\AUTHOR{Shimrit Shtern}
\AFF{The William Davidson Faculty of Industrial Engineering \& Management, Technion - Israel Institute of Technology, \EMAIL{shimrits@technion.ac.il}}
\AUTHOR{Bradley Sturt}
\AFF{College of Business Administration, 
University of Illinois at Chicago, \EMAIL{bsturt@uic.edu}}
}
\ABSTRACT{We investigate a simple approximation scheme, based on overlapping linear decision rules, for solving data-driven two-stage distributionally robust optimization problems with the type-$\infty$ Wasserstein ambiguity set.  Our main result establishes that this approximation scheme is {asymptotically optimal} for two-stage stochastic linear optimization problems; that is, under mild assumptions, the optimal cost and optimal first-stage decisions obtained by approximating the robust optimization problem converge to those of the underlying stochastic problem as the number of data points grows to infinity. These guarantees notably apply to two-stage stochastic  problems that do not have \emph{relatively complete recourse}, which arise frequently in applications.  In this context, we show through numerical experiments that the approximation scheme is practically tractable and produces decisions which significantly outperform those obtained from state-of-the-art  data-driven alternatives.

}%

\HISTORY{This paper was first submitted on February 11, 2018 and underwent three revisions. It was accepted for publication on November 1, 2020. }

\maketitle

\section{Introduction} \label{sec:mp:intro}
Dynamic decision making under uncertainty, in which decisions are made and uncertainty is revealed over time, form the foundation of a myriad of applications in operations research, control theory, and computer science. 
A prominent framework which encompasses many of these decision problems is \emph{two-stage stochastic linear optimization}. Introduced by \cite{dantzig1955linear} and \cite{beale1955minimizing}, this framework is used to address problem settings where an initial decision is made under uncertainty, after which random variables are revealed, followed
by the selection of a second-stage decision. The enduring study of two-stage stochastic linear optimization can be attributed to its prevalence in modern operational applications, such as inventory management, network design, and energy planning \citep{birge2011introduction}. 

A central challenge in solving these optimization problems is that the probability distribution of the random variables is rarely known in practice. Consequently, much research effort has focused on developing methodologies that leverage historical data to find near-optimal solutions to  two-stage stochastic linear optimization problems with unknown distributions. 
In this context, perhaps the most celebrated data-driven approach for two-stage stochastic linear optimization is the sample average approximation (SAA). In a nutshell, SAA finds first-stage decisions by solving a two-stage stochastic linear optimization problem in which the true distribution is replaced by the empirical distribution of the historical data; see \citet[Section 5]{shapiro2009lectures}. 

The SAA approach to two-stage stochastic linear optimization offers several attractive properties. For one, SAA can be solved exactly as a linear optimization problem with size (number of decision variables and constraints) which scales linearly in the number of data points. 
Moreover, SAA is {asymptotically optimal} under mild probabilistic assumptions, meaning that the optimal cost and optimal first-stage decisions produced by SAA are guaranteed to converge to those of the underlying stochastic problem as the number of data points tends to infinity \citep{shapiro2003monte,robinson1996analysis,king1991epi}. 

However, the attractiveness of SAA in two-stage problems can be curtailed when faced with limited historical data. In both single-stage and two-stage problems, when the number of data points is finite, the optimal cost of SAA will produce an undesirable {optimistically-biased} estimate of the cost of the stochastic problem; see, \eg  \cite{van2017data}. Even more prominently, to ensure that the first-stage decisions from SAA will have a {feasible} second-stage decision, an impractically large or infinite number of data points can be required \citep[Section 3]{nemirovski2006scenario}. 
Two-stage problems which do not have {relatively complete recourse}, \ie problems in which the second-stage problem is not always feasible, frequently occur in practice and ``in many applications is a rule rather than an exception" \citep[Page 16]{nemirovski2006scenario}. 

To obtain better approximations of two-stage stochastic linear optimization from limited data, recent work has investigated augmentations of SAA in which  {adversarial noise} is added to the data points. These approaches are often formulated as distributionally robust optimization problems, where the optimal first-stage decision is that which performs best in expectation under an adversarially chosen probability distribution \citep{wiesemann2014distributionally,delage2010}. 
By restricting the adversary to choose probability distributions which are close, in some sense, to the historical data, there is growing evidence that the first-stage decisions produced by two-stage distributionally robust optimization can have better average out-of-sample performance compared to those produced by SAA  \citep{hanasusanto2016conic,jiang2018risk}. 

Unfortunately, the potential benefits of distributionally robust optimization in two-stage problems are often hobbled by an increase in computational cost. 
In contrast to SAA in two-stage problems, which can generally be solved as a linear optimization problem, two-stage robust optimization and two-stage distributionally robust optimization with Wasserstein-based ambiguity sets are shown to be NP-hard \citep{feige2007robust,hanasusanto2016conic}.  Accordingly, 
 the development of approximation algorithms for addressing two-stage distributionally robust optimization with data-driven ambiguity sets have become a central focus of research. \cite{hanasusanto2016conic} show two-stage problems with Wasserstein ambiguity sets that can be formulated exactly using co-positive optimization, which can be  approximated using semidefinite optimization. \cite{chen2019RSO} propose event-wise adaptations for solving two-stage distributionally robust optimization with Wasserstein and $k$-means ambiguity sets, and 
\cite{jiang2018risk} study sampling approaches for solving two-stage distributionally robust optimization with phi-divergence ambiguity sets. 

In this paper, we investigate a solution approach for a class of Wasserstein-based {two-stage  distributionally robust optimization problems},  that aims to gracefully retain the attractive properties (scalability and asymptotic optimality) of SAA. In particular, the approach   can be tractably solved as a linear optimization problem, and is guaranteed to converge asymptotically to the stochastic problem under mild assumptions. We believe the approach thus offers an attractive step towards bridging the relative merits of distributionally robust optimization and SAA in the context of two-stage problems, particularly those without relatively complete recourse.
In greater detail: 

\begin{itemize}
\vspace{0.5em}
\item We consider a solution {approach}, referred to as a \emph{multi-policy approximation}, for addressing data-driven two-stage distributionally robust optimization with the type-$\infty$ Wasserstein ambiguity set. The {approach} forms an approximation of these robust problems by optimizing overlapping  decision rules, one for each uncertainty set around each data point. We show that the approximation quality is guaranteed to outperform \emph{any} traditional decision rule approximation (Theorem~\ref{thm:mp:hierarchy}), and in the case of linear decision rules, can be solved as a linear or second-order conic optimization problem (Proposition~\ref{prop:mp:reformulation}). Similar to SAA, the size of this optimization problem (number of decision variables and constraints) grows linearly in the number of data points.

\item We prove that the multi-policy approximation with linear decision rules is {asymptotically optimal} for two-stage stochastic linear optimization. That is, under mild assumptions, we show that the optimal cost and optimal first-stage decisions obtained from the multi-policy approximation  will converge almost surely to those of the underlying two-stage stochastic linear optimization problem as the number of data points tends to infinity (Theorem~\ref{thm:mp:conv}). From a practical perspective, such a guarantee provides  assurance that any bias or suboptimality of two-stage distributionally robust optimization with the type-$\infty$ Wasserstein ambiguity set, as well as that of the multi-policy approximation, disappear as more data is obtained. 
%

\end{itemize}

In addition to the above methodological results, we provide  numerical evidence that the multi-policy approximation can also be attractive in practice. First, in a two-stage network inventory management problem without relatively complete recourse, and across different probability distributions and various sizes of datasets, we show that the approximation produces first-stage decisions which significantly outperform those produced by alternative methods in feasibility (compared to SAA) and average cost (compared to other distributionally robust optimization approaches). 
Second, in a two-stage hospital scheduling problem, we show that the approach matches the out-of-sample performance of state-of-the-art approximation algorithms for two-stage distributionally robust optimization with the type-1 Wasserstein ambiguity set. 

Distributionally robust optimization with the type-$\infty$ Wasserstein ambiguity set was first studied in \cite{bertsimas2018multistage}, where this class of problems was shown to be equivalent to a robust optimization problem over multiple uncertainty sets. Robust optimization problems over multiple uncertainty sets, which we refer to as sample robust optimization, have also been studied in the literature in other settings \citep{erdougan2006ambiguous,erdougan2007two,xu2012distributional}.  In the context of data-driven multi-stage stochastic linear optimization, \cite{bertsimas2018multistage} showed under certain conditions that sample robust optimization is asymptotically optimal with respect to the underlying stochastic problem. While that paper showed that these sample robust optimization problems can be approximated by restricting to linear or piecewise-linear decision rules, {it does not analyze the tightness of these approximations}.

The present paper differs from the aforementioned literature in several significant ways. Most importantly,  
this work considers a particular approximation scheme which exploits the unique structure of two-stage sample robust optimization problems, and proves that its approximation gap converges to zero as the number of data points tends to infinity. 
Moreover, we also strengthen the understanding of the asymptotic behavior of sample robust optimization in the case of  two-stage  stochastic linear optimization problems with general (not necessarily light-tailed or continuous) probability distributions. Specifically, 
  we show that two-stage sample robust optimization is asymptotically optimal under nearly-identical assumptions to those of SAA (Theorem~\ref{thm:mp:conv}).

In view of the above discussion, the aim of this paper is \emph{not}  to motivate Wasserstein-based  two-stage distributionally robust optimization  from the perspective of  performance guarantees, but rather to present a scalable and asymptotically-optimal approximation scheme for solving a class of these problems. 
For the interested reader, in Appendix~\ref{appx:finite_sample}, we provide a review of existing probabilistic performance guarantees for distributionally robust optimization with the type-$\infty$ Wasserstein ambiguity set. 
Using those results, we also show therein a stylized example in which the first-stage decisions obtained from distributionally robust optimization with the type-$\infty$ Wasserstein ambiguity set will provably outperform those obtained by SAA.


A desirable feature of the approximation scheme and asymptotic optimality guarantees from this paper is that they apply to two-stage sample robust optimization problems in which partial knowledge about the support of the random variables is incorporated via a polyhedral set $\Xi \subseteq \R^d$. The set $\Xi \subseteq \R^d$ can be used by the decision-maker to incorporate prior knowledge that eliminates nonsensical scenarios that would render the robust optimization problem overly conservative. For example, this set can be used by the decision-maker to incorporate prior knowledge that capacities in a transportation network will be nonnegative (\eg Example~\ref{ex:support}) or that demands will not exceed an upper bound (\eg Section~\ref{sec:mp:cap_lot}).  Importantly,  the incorporation of such prior knowledge can be essential in order for the two-stage sample robust optimization problem to be feasible, as shown in Example~\ref{ex:support}.   After an earlier version of this paper was circulated online, a follow-up work by \cite{xie2020tractable} suggested reformulations and complexity results for two-stage distributionally robust optimization with type-$\infty$ Wasserstein ambiguity sets in the case of $\Xi =\R^d$.

Independently to our work, \cite{chen2019RSO} consider a similar event-wise adaptation algorithm for a class of distributionally robust optimization problems with event-wise ambiguity sets. In contrast, they do not analyze if or when such an approach is asymptotically optimal. In this paper, we propose and analyze the multi-policy approximation {approach} for two-stage distributionally robust optimization with the type-$\infty$ Wasserstein ambiguity set, which was not initially considered by \cite{chen2019RSO}. By utilizing the local structure of this class of problems, we show that the multi-policy approximation with linear decision rules is asymptotically optimal for the first time (Theorem~\ref{thm:mp:conv}), including for problems without relatively complete recourse.

Our paper is organized as follows. 
Section \ref{sec:mp:prelim} introduces the sample robust optimization approach for two-stage linear optimization.   
Section~\ref{sec:mp:approx} proposes the multi-policy approximation scheme using linear decision rules for two-stage sample robust optimization problems, and analyzes its tractability. 
Section \ref{sec:mp:asymptotic} establishes the asymptotic optimality of the multi-policy approximation scheme. 
Section~\ref{sec:mp:experiments} presents computational experiments. We conclude this paper in Section~\ref{sec:mp:conclusion}.


\subsubsection*{Notation} \label{sec:mp:Intro:notation}
We represent vectors and matrices by bold lowercase and uppercase letters, such as $\bx \in \R^n$ and $\bbt \in \R^{m \times n}$. For any integer $N\in \N$, we let $[N]$ be shorthand for the set $\{1,\ldots,N\}$. The space of all functions of the form $\by(\cdot): \R^d \to \R^r$ is given by  $\mathcal{R}^{d,r}$. 
The closed ball of radius $\epsilon \ge 0$ centered at a vector $\hat{\bz} \in \R^f$ is denoted by 
$B(\hat{\bz},\epsilon) \triangleq \{\bz \in \R^{f}: \|\bz - \hat{\bz}\| \le \epsilon\},$
where $\norm{\cdot}$ refers to any $\ell_p$-norm, and  $\| \cdot \|_*$ is its dual norm. 
 For any nonempty convex set $S \subseteq \R^f$, its relative interior is $\relint(S) \triangleq \{\bz \in S: \; \forall \hat{\bz} \in S, \;\exists \lambda > 1: \; \lambda \bz + (1-\lambda) \hat{\bz} \in S\}.$ 
Given two sets $S, T\subseteq \R^f$, their Minkowski sum is $S+T \triangleq \{ \bz + \bz':\; \bz \in S, \bz' \in T\}$.
Further notation and results from polyhedral theory are found in Appendix~\ref{appx:polyhedral}.


\section{Problem Setting} \label{sec:mp:prelim}
We consider two-stage stochastic linear optimization problems of the form
\begin{equation}\label{prob:mp:first_stage}\tag{OPT}
\begin{aligned}[t]
v^* \; \triangleq  \quad &\underset{\bx \in \R^n}{\textnormal{minimize}}&& \left \{ V^*(\bx) \triangleq \bc^\intercal \bx +  \Exp[Q(\bx,\bxi)]  \right \}.
\end{aligned}
\end{equation}
The first-stage decision $\bx \in \R^n$ is selected before the realization of any uncertainty, and $\bxi \in \R^d$ is a random variable with an underlying probability distribution. 
Without loss of generality, we assume that any deterministic linear constraints on the first-stage decision are embedded in the second-stage cost function. Given a first-stage decision and realization of the random variable, the second-stage cost is given by
\begin{equation*}
\begin{aligned}
Q(\bx,\bxi) \triangleq \quad 
&\underset{\by \in \R^r}{\text{minimize}} &&\bq^\intercal \by \\
&\textnormal{subject to}&& \bbt \bx + \bbw \by \ge \bh(\bxi),
\end{aligned}
\end{equation*}
where $\bq\in\R^r$, $\bbt\in \R^{m\times n}$, $\bbw\in \R^{m\times r}$, and the right-hand side of the constraints is an affine function of the form $\bh(\bxi) = \bh^0 + \bbh \bxi \in \R^m$. Following standard convention, the objective value of the above linear optimization problem is equal to infinity whenever there is no feasible second-stage (recourse) decision which satisfies its constraints. 

We  assume throughout this paper that the probability distribution and support of the random variable are unknown. Instead, our only information consists of historical data $\bxi{}^1,\ldots,\bxi{}^N \in \R^d$, which are independent and identical distributed samples of the underlying random variable, 
as well as a polyhedral set $\Xi \triangleq \{ \bzeta \in \R^d: \bbg \bzeta \ge \bg^0 \}$ which is a conservative superset of the support of the random variable, \ie $\Prb(\bxi \in \Xi) = 1$, where $\bbg\in\R^{\tilde{m}\times d}$ and $\bg^0\in \R^{\tilde{m}}$. A further discussion on the conservative superset of the support is found at the end of this section.  

In this paper, we investigate the following data-driven approach to Problem~\eqref{prob:mp:first_stage}. 
Given historical data, we first construct multiple uncertainty sets, one around each historical data point. Then, we choose a first-stage decision and estimate the optimal cost $v^*$ of the stochastic problem by solving the following robust optimization problem:
\begin{equation}\label{prob:mp:sro} \tag{SRO}
\begin{aligned}
\hat{v}_N^{\textnormal{SRO}} \triangleq \quad & \underset{\bx \in \R^n}{\textnormal{minimize}} && \left \{ \hat{V}_N^{\textnormal{SRO}}(\bx) \; \triangleq \; \bc^\intercal \bx + \frac{1}{N} \sum_{i=1}^N \sup_{\bzeta \in \mathcal{U}_N^i} Q(\bx,{\bzeta})  \right \}.
\end{aligned}
\end{equation}
Intuitively speaking, the above robust optimization problem finds a first-stage decision by averaging over the historical data; however, each historical data point is perturbed by an adversary within its uncertainty set. 
We focus on solving Problem~\eqref{prob:mp:sro} when the uncertainty sets are constructed as closed balls that are centered at each data point and intersected with the polyhedron $\Xi$, 
\begin{equation*}
\mathcal{U}^i_N \triangleq \{\bzeta\in \Xi: \; \|\bzeta-\bxi{}^i \|\leq \epsilon_N\},
\end{equation*}
where $\epsilon_N \ge 0$ is a parameter, chosen by the decision maker, which controls the size of the uncertainty sets. Under this construction of the uncertainty  sets, Problem~\eqref{prob:mp:sro} is equivalent to two-stage distributionally robust optimization the type-$\infty$ Wasserstein ambiguity set \citep[Proposition 3]{bertsimas2018multistage}.

We observe that Problem~\eqref{prob:mp:sro} is computationally demanding to solve exactly. Indeed, evaluating the second-stage cost $\hat{V}^{\textnormal{SRO}}_N(\bx)$ is NP-hard, even when $N = 1$, 
as it consists of maximizing a piecewise-linear convex function over a polyhedron.\footnote{To see why evaluating the optimization problem is NP-hard, consider the case where the polyhedron $\Xi \subseteq \R^d$ is bounded and $\epsilon_N = \infty$. Then the complexity results follows immediately from \citet[Theorem 5]{feige2007robust}. } The aim of this paper is to develop and analyze an approximation algorithm for Problem~\eqref{prob:mp:sro}. 
\begin{remark}
Problem~\eqref{prob:mp:sro} with these uncertainty sets can be interpreted as a robust generalization of the well-known sample average approximation (SAA): 
\begin{equation}\label{prob:mp:saa}\tag{SAA}
\begin{aligned}[t]
\hat{v}_N^{\textnormal{SAA}} \triangleq \; \min_{\bx \in \R^n} \left \{ \hat{V}_N^{\textnormal{SAA}}(\bx) \triangleq  \bc^\intercal \bx +  \frac{1}{N} \sum_{i=1}^NQ(\bx,{\bxi}{}^i)  \right \}. 
\end{aligned}
\end{equation}
Indeed, we readily observe for the special case of $\epsilon_N = 0$ that Problem~\eqref{prob:mp:sro} and Problem~\eqref{prob:mp:saa} are equivalent. 
\end{remark}
\begin{remark} \label{remark:support} 
From a modeling perspective,  the set $\Xi \subseteq \R^d$ can be used by the decision-maker to incorporate prior knowledge that eliminates nonsensical scenarios that would render the robust optimization problem overly conservative.  
More importantly, in some cases, the incorporation of such prior knowledge is strictly needed to ensure that   Problem~\eqref{prob:mp:sro} has feasible first-stage decisions. We illustrate this observation through the following example.

\begin{example} \label{ex:support} Consider a network flow problem with uncertain capacities. In the first-stage, we are tasked with selecting an inventory level $x_f \ge 0$ in each of $n$ facilities with the goal of satisfying known demands $d_j \ge 0$ in each of $m$ destinations. After the inventory levels are chosen, the capacities $\xi_{ij} \ge 0$ between facility $i$ and destination $j$ are observed. The demands in the destinations are satisfied by transferring inventory from the facilities, and the transfer quantity from facility $i$ to destination $j$ is captured by the second-stage decision variable $0 \leq y_{ij} \leq \xi_{ij}.$ 
	If there are historical data points where a capacity $\xi_{ij}$ of some edge is exactly zero, then Problem~\eqref{prob:mp:sro} will be infeasible for any positive radius $\epsilon_N > 0$ unless the set $\Xi$ enforces nonnegativity for the capacities along those edges.
	\end{example}
\end{remark}
\section{The Multi-Policy Approximation Scheme} \label{sec:mp:approx}
In this section, we present a simple approximation algorithm, based on overlapping linear decision rules, for solving Problem~\eqref{prob:mp:sro}. 
To motivate this approach, we start in Section~\ref{sec:mp:single_policy} by presenting a standard {single-policy} approximation. We then introduce the multi-policy approximation in Section~\ref{sec:mp:mp_approach}, which we show can obtain better approximations of Problem \eqref{prob:mp:sro} with a similar computational cost.
\subsection{Single-policy approximation} \label{sec:mp:single_policy}
We observe that Problem~\eqref{prob:mp:sro} can be equivalently represented as a single optimization problem where the recourse decision is a function of the uncertainty:
\begin{equation*}
\begin{aligned}
\hat{v}_N^{\textnormal{SRO}} \triangleq \quad &\underset{\bx \in \R^n,\; \by(\cdot) \in \mathcal{R}^{d,r}}{\text{minimize}} &&\bc^\intercal \bx + \frac{1}{N} \sum_{i=1}^{N} \sup_{\bzeta \in \mathcal{U}^i_N} \bq^\intercal\by(\bzeta) \\
&\text{subject to}&&\bbt \bx + \bbw \by(\bzeta) \geq \bh(\bzeta)\quad \forall \bzeta \in \cup_{i=1}^N \mathcal{U}^i_N.
\end{aligned}
\end{equation*}

\noindent A common approximation technique for two-stage problems is to restrict $\Rfun^{d,r}$ to a smaller space of policies over which we can efficiently optimize. 
More precisely, let $\Pi \subseteq \Rfun^{d,r}$ denote a restricted space of second-stage recourse policies. Then we obtain an approximation of Problem~\eqref{prob:mp:sro} by solving
\begin{equation}\label{eqn:mp:single_policy}\tag{SP}
\begin{aligned}
&\underset{\bx \in \R^n,\; \by(\cdot) \in \Pi}{\text{minimize}} &&\bc^\intercal \bx + \frac{1}{N} \sum_{i=1}^{N} \sup_{\bzeta \in \mathcal{U}^i_N} \bq^\intercal\by(\bzeta) \\
&\text{subject to}&&\bbt \bx + \bbw \by(\bzeta) \geq \bh(\bzeta)\quad \forall \bzeta \in \cup_{i=1}^N \mathcal{U}^i_N.
\end{aligned}
\end{equation}

\noindent We observe that any first-stage decision which is feasible for Problem \eqref{eqn:mp:single_policy} will be feasible for Problem~\eqref{prob:mp:sro}, and the optimal cost  of Problem \eqref{eqn:mp:single_policy} provides an upper bound approximation on the optimal cost of Problem \eqref{prob:mp:sro}. 

One choice of policies is the space of {linear decision rules}, denoted by
\begin{equation*}
\Piaff=\{\by(\cdot) \in \Rfun^{d,r}: \exists \by^0\in\R^{r},\bby\in \R^{r\times d}: \; \by(\bxi)=\by^0+\bby \bxi \}.
\end{equation*}
It has been shown that restrictions to linear decision rules in many two-stage robust optimization can result in tractable optimization problems  \citep{ben2004adjustable}. We note that the single-policy approximation can  be applied with alternative restricted spaces of policies other than $\mathcal{L}$, such as $K$-adaptivity \citep{hanasusanto2015k}, finite adaptability \citep{bertsimas2010finite}, and lifted linear decision rules \citep{Chen2009}. 
The approximation gap between a single-policy approximation and Problem~\eqref{prob:mp:sro} will generally decrease as the set of possible policies gets larger. 

Despite its tractability, the single-policy approximation has several weaknesses. First, Problem \eqref{eqn:mp:single_policy} may have fewer, if any, feasible first-stage decisions compared to Problem \eqref{prob:mp:sro}. This is due to the fact that some feasible first-stage decisions of Problem~\eqref{prob:mp:sro}, \ie the first-stage decisions that have a feasible second-stage recourse policy $\by(\cdot)\in\Rfun^{d,r}$,  may no longer be feasible after restricting $\Rfun^{d,r}$ to $\Pi$. Second, the approximation gap between the optimal cost of the single-policy approximation and the fully-adaptive problem can be significant. Characterizing the approximation gap resulting from single-policy approximations remains an active area of research in two-stage robust linear optimization. Importantly, unless the space $\Pi$ is very rich, we do not expect the objective value of the single-policy approximation to converge to that of the underlying stochastic optimization problem as $N \to \infty$ and $\epsilon_N \to 0$. 

\subsection{Multi-policy approximation} \label{sec:mp:mp_approach}

Motivated by the shortcomings of the single-policy approximation, we consider a different approach for approximating Problem~\eqref{prob:mp:sro}. The following approach, henceforth referred to as the multi-policy approximation, uses the same idea of restricting the space of recourse policies. In contrast to the single-policy approximation, the multi-policy approximation additionally allows for different heuristic policies to be optimized for the different uncertainty sets. 
Specifically, the multi-policy approximation uses the restricted family of policies $\Pi$, and optimizes over $N$ recourse policies $\by^1(\cdot),\ldots,\by^N(\cdot) \in \Pi$ such that $\by^i(\cdot)$ is feasible for all possible realizations in the $i$-th uncertainty set. 
The formulation of the multi-policy approximation is as follows:
\begin{equation}\label{prob:mp:multi_policy}\tag{MP}
\begin{aligned}
\hat{v}_N^{\textnormal{MP}}\triangleq \quad &\underset{\substack{\bx \in \R^n\\ \by^1(\cdot),\ldots,\by^N(\cdot) \in \Pi}}{\text{minimize}} &&\bc^\intercal \bx + \frac{1}{N} \sum_{i=1}^{N} \sup_{\bzeta \in \mathcal{U}^i_N} \bq^\intercal\by^i(\bzeta) \\
&\text{subject to}&&\bbt \bx + \bbw \by^i(\bzeta) \geq \bh(\bzeta) \quad \forall \bzeta \in \mathcal{U}^i_N, \; i \in [N].
\end{aligned}
\end{equation}

At first glance, the multi-policy approximation closely resembles the single-policy approximation. In fact, the two approximations have the same type and number of robust constraints, albeit the multi-policy approximation has $N$ times as many recourse policies over which to optimize. We readily observe however that the multi-policy approximation is a generalization of the single-policy approximation; indeed, the two problems are identical if the recourse policies $\by^1(\cdot),\ldots,\by^N(\cdot) $ are restricted to be equal to one another. The crucial distinction is that the multi-policy approximation gives the flexibility to find recourse policies that are locally optimal for the various uncertainty sets (see Figure~\ref{fig:mp:multi_policy_visualization}). 
\begin{figure}[t]
\centering
	\includegraphics[width=0.32\textwidth]{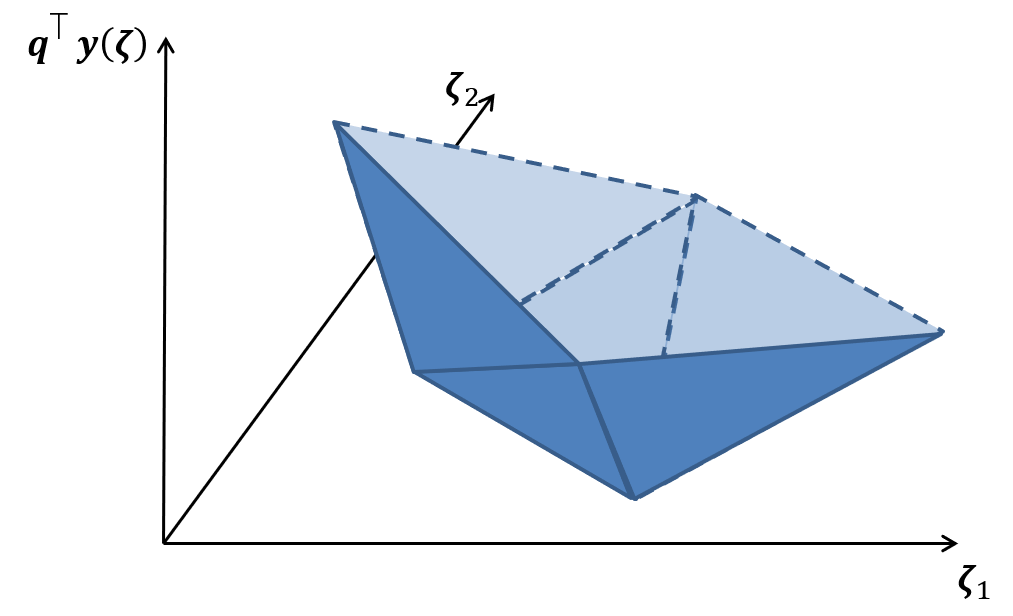}	
	\includegraphics[width=0.32\textwidth]{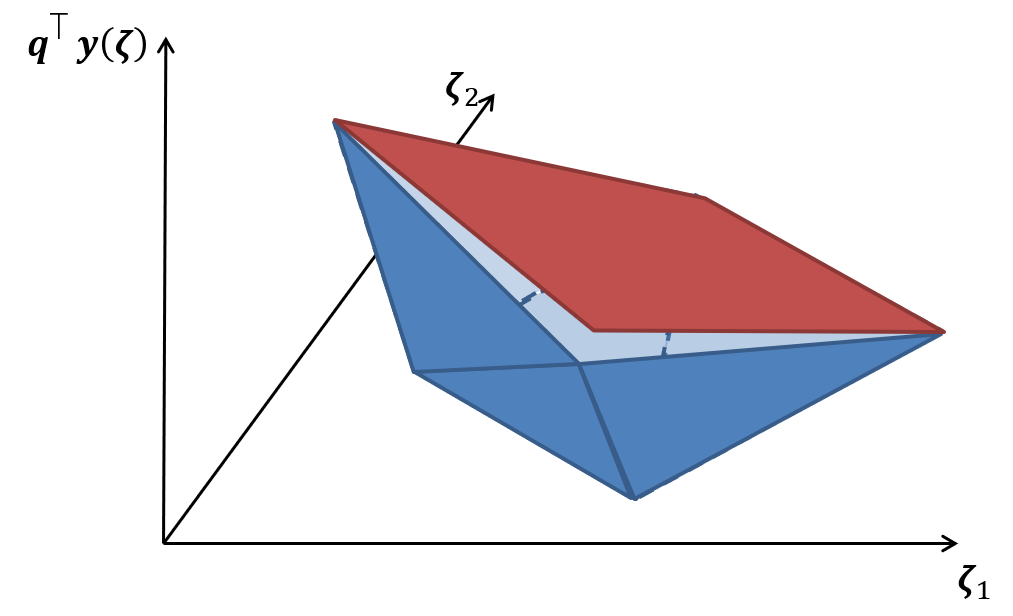}	
	\includegraphics[width=0.32\textwidth]{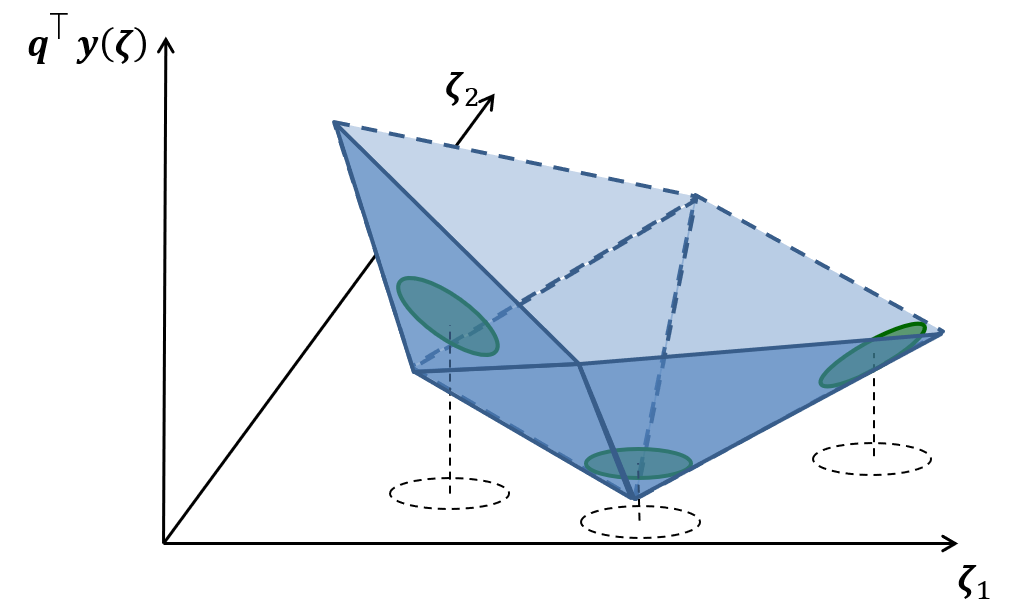}	
	\caption{Visualization of approximations for two-stage sample robust optimization. \underline{Left}: the objective value under the optimal recourse policy $\by(\cdot) \in \Rfun^{d,r}$. \underline{Center}: the objective value of the single-policy approximation using a linear decision rule. \underline{Right}: the objective value of the multi-policy approximation using linear decision rules.}\label{fig:mp:multi_policy_visualization}		
\end{figure}

Since the multi-policy approximation is a generalization of the single-policy approximation, it will never be more restrictive. Indeed, every feasible first-stage decision to the single-policy approximation is feasible for the multi-policy approximation, and the objective value of the multi-policy approximation is never greater than that of the single-policy approximation. We now show that the multi-policy approximation provides a valid upper approximation of Problem~\eqref{prob:mp:sro}. From this point onward, we use the notation $\hat{V}^{\textnormal{SP}}_N(\bx)$ and $\hat{V}^{\textnormal{MP}}_N(\bx)$ to denote the optimal cost of Problems~\eqref{eqn:mp:single_policy} and \eqref{prob:mp:multi_policy}, respectively, with the first-stage decision $\bx \in \R^n$ (where $\hat{V}^{\textnormal{SP}}_N(\bx)$, $\hat{V}^{\textnormal{MP}}_N(\bx)$ are set to infinity if the first-stage decision is infeasible for the respective problems). 

\begin{theorem} \label{thm:mp:hierarchy}
$\hat{v}^{\textnormal{SRO}}_N \leq \hat{v}^{\textnormal{MP}}_N \le \hat{v}^{\textnormal{SP}}_N $ and $\hat{V}_N^{\textnormal{SRO}}(\bx) \le \hat{V}^{\textnormal{MP}}_N(\bx) \le \hat{V}^{\textnormal{SP}}_N(\bx)$ for all $\bx \in \R^n$. 
\end{theorem}
\begin{proof}{Proof. }
Consider any first-stage decision $\bx \in \R^n$. If $\hat{V}^{\text{MP}}_N(\bx)=+\infty$, then the inequality $\hat{V}_N^{\textnormal{SRO}}(\bx) \le \hat{V}^{\text{MP}}_N(\bx)$ is trivially satisfied. Otherwise, assume there exist second-stage decision rules $\by^{1}(\cdot),\ldots,\by^N(\cdot) \in \Pi$ which are optimal for the multi-policy approximation with the first-stage decision $\bx \in \R^n$. For any realization $\bzeta \in \cup_{i=1}^N \mathcal{U}^i_N$, let us define the corresponding index $i(\bzeta)$ by
\begin{align*}
i(\bzeta) \triangleq \argmin_{i \in \{1,\ldots,N\}} \left \{ \bq^\intercal \by^{i}(\bzeta):  \;\bzeta \in \mathcal{U}^i_N
 \right\}.
\end{align*}
(If there are multiple optimal indices, choose the index which is smallest.) We define a new recourse policy $\bar{\by}(\cdot) \in \Rfun^{d,r}$ as
\begin{align*}
\bar{\by}(\bzeta) &\triangleq \begin{cases}
\by^{i(\bzeta)}(\bzeta),&\text{if } \bzeta \in \cup_{i=1}^N \mathcal{U}^i_N,\\
\bzero,&\text{otherwise}.
\end{cases}
\end{align*}
It follows from construction that the tuple $(\bx, \bar{\by}(\cdot))$ is feasible for Problem~\eqref{prob:mp:sro}. 
Therefore, 
\begin{align*}
\hat{V}_N^{\textnormal{SRO}}(\bx) \le \bc^\intercal \bx + \frac{1}{N} \sum_{i=1}^N \sup_{\bzeta \in \mathcal{U}^i_N} \bq^\intercal \bar{\by}(\bzeta) \le \bc^\intercal \bx + \frac{1}{N} \sum_{i=1}^N \sup_{\bzeta \in \mathcal{U}^i_N} \bq^\intercal \by^i(\bzeta) \le \hat{V}_N^{\textnormal{MP}}(\bx).
\end{align*}
In all cases, we have shown that the inequality $\hat{V}_N^{\textnormal{SRO}}(\bx) \le \hat{V}^{\textnormal{MP}}_N(\bx)$ holds for all first-stage decisions $\bx \in \R^n$. This immediately implies that $\hat{v}^{\textnormal{SRO}}_N \leq \hat{v}^{\textnormal{MP}}_N $. The final inequalities ($\hat{v}^{\textnormal{MP}}_N \le \hat{v}^{\textnormal{SP}}_N$ and $\hat{V}^{\textnormal{MP}}_N(\bx) \le \hat{V}^{\textnormal{SP}}_N(\bx)$ for all $\bx \in \R^n$) follow directly from the definitions of Problems~\eqref{eqn:mp:single_policy} and \eqref{prob:mp:multi_policy}, which concludes the proof. 
\Halmos \end{proof}

The above result demonstrates that Problem~\eqref{prob:mp:multi_policy} provides a valid upper approximation of Problem~\eqref{prob:mp:sro}, and any feasible first-stage decision to Problem~\eqref{prob:mp:multi_policy} will also be feasible for Problem~\eqref{prob:mp:sro}. 
More generally, Theorem~\ref{thm:mp:hierarchy} suggests that the idea of \emph{anticipativity} need not apply to the second-stage decision variables in two-stage problems. Intuitively speaking, after the random variables are observed, there is no future uncertainty to anticipate in two-stage problems. 
Thus, Theorem~\ref{thm:mp:hierarchy} shows that one can obtain a valid upper approximation of a two-stage problem by considering overlapping decision rules, even those which prescribe contradicting second-stage decisions from identical realizations of the uncertain parameters. In practice, having overlapping second-stage decision rules is not itself a significant issue: indeed, given a first-stage decision $\bx$ and a realization of the random variables $\bxi$, a second-stage decision $\by$ can always be found by solving a linear optimization problem, or by applying one of the second-stage decision rules found by the multi-policy approximation.

We conclude this section by showing that the multi-policy approximation emits a tractable representation. 
Specifically, using convex duality theory, Problem~\eqref{prob:mp:multi_policy} with linear decision rules can be reformulated into a finite-dimensional optimization problem. The proof is a standard application of Lagrangian duality and is thus omitted. 

\begin{proposition}\label{prop:mp:reformulation}
 If $\Pi = \mathcal{L}$, then Problem~\eqref{prob:mp:multi_policy} is equivalent to 
 	\begin{equation*}
	\begin{aligned}[t]
	& \underset{\substack{\bx\in\Real^n,\by{}^{i,0}\in \Real^{r},\bby{}^{i}\in \Real^{r\times d},\\
			\theta{}^{i,j}\in\R_+,\brho{}^{i,j}\in \R^{\tilde{m}}_+\\
			\forall i\in[N],j\in\{0,1,\ldots,m\}\\\\
	}}{\textnormal{minimize}}&& \bc^\intercal \bx+\frac{1}{N} \sum_{i=1}^N\left(\theta^{i,0}\epsilon_N  +(\hat{\bg}{}^i)^\intercal\brho{}^{i,0}+\bq^\intercal (\by{}^{0,i} + \bby{}^{i} \bxi{}^i)\right)\\
	& \textnormal{subject to} && \norm{( \bby^{i})^\intercal \bq + \bbg^\intercal\brho{}^{i,0}}_* \le \theta{}^{i,0}, && i \in [N],\\
	&&&\bei_j^\intercal\left[\bbt\bx+\bbw \left( \by{}^{0,i} + \bby{}^{i} \bxi{}^i\right)-\bh(\bxi{}^i)\right]\geq \theta^{i,j}\epsilon_N+ (\hat{\bg}{}^i)^\intercal \brho{}^{i,j}, &&  j\in [m],\; i \in [N],\\
	&&&
	\norm{(\bbh- \bbw\bby^{i})^\intercal\bei_j  + \bbg^\intercal \brho{}^{i,j}}_* \le \theta{}^{i,j},&&  j\in [m],\; i \in [N],
	\end{aligned}
	\end{equation*}
	where $\norm{\cdot}_*$ is the dual norm of $\norm{\cdot}$, $\bei_j$ is the $j$th vector of the identity matrix, and $\hat{\bg}{}^i \equiv \bbg \bxi{}^i - \bg^0\geq0$.
\end{proposition}
The type of optimization problem generated in this reformulation depends on the choice of the norm in the definition of the uncertainty sets. If the chosen norm is the $\ell_1$ or $\ell_\infty$, then the reformulation becomes a linear optimization problem; if the chosen norm is the $\ell_2$, we obtain a second-order conic optimization problem. In both cases, the resulting  optimization problem has $O(N)$ constraints and $O(N)$ variables, and is readily solved by a variety of off-the-shelf solvers.


\section{Asymptotic Optimality} \label{sec:mp:asymptotic}

In this section, we prove that the multi-policy approximation of two-stage sample robust optimization is asymptotically optimal. That is, provided that $\epsilon_N \to 0$ as more data is obtained, we show that the optimal cost and first-stage decisions of the multi-policy approximation with linear decision rules will converge almost surely to those of the underlying two-stage stochastic linear optimization problem (Theorem~\ref{thm:mp:conv}). From a practical perspective, such a guarantee provides assurance that any suboptimality of the multi-policy approximation disappears as more data is obtained. The guarantee can also be viewed as attractive from a theoretical standpoint, as it is established under mild probabilistic assumptions which are similar to those required for establishing the asymptotic optimality of Problem~\eqref{prob:mp:saa}. 

Our main result is the following:
\begin{theorem} \label{thm:mp:conv}
Let the following conditions hold: 
\begin{enumerate}[label={\bf [A\arabic*]},ref=A\arabic*]
\item \label{ass:linear_epsilon} 
$\Pi = \mathcal{L}$ and $\epsilon_N \to 0$ as $N \to \infty$. 
\item \label{ass:v_finite} 
 $v^*$ is finite and $\Exp[ \| \bxi \|] < \infty$. 
\item \label{ass:saa}
The set of optimal first-stage decisions for Problem~\eqref{prob:mp:first_stage} is nonempty and bounded. 
\item \label{ass:feas}
There exists a first-stage decision ${\bx} \in \R^n$ and a radius ${\epsilon} > 0$ such that, for all $\hat{\bzeta} \in \Xi$, there exists a  decision rule $\by^{\hat{\bzeta}}(\cdot) \in \mathcal{L}$ such that 
$\bbt {\bx} + \bbw {\by}^{\hat{\bzeta}}(\bzeta) \ge \bh(\bzeta)$ for all  $\bzeta \in B(\hat{\bzeta},{\epsilon}) \cap \Xi.$ 
\end{enumerate}
Then,
 \begin{align}
 \lim_{N \to \infty} \hat{v}_N^{\textnormal{SRO}} =  \lim_{N \to \infty} \hat{v}_N^{\textnormal{MP}} = v^* \quad \text{a.s.} \label{line:mp:conv}
 \end{align}
 Moreover, let $\hat{\bx}_N^{\textnormal{SRO}}$ and $\hat{\bx}_N^{\textnormal{MP}}$ be optimal first-stage decisions for Problems~\eqref{prob:mp:sro} and \eqref{prob:mp:multi_policy}, respectively. Then any accumulation point of  $\{ \hat{\bx}_N^{\textnormal{SRO}} \}_{N\in\N}$ or $\{ \hat{\bx}_N^{\textnormal{MP}} \}_{N\in\N}$ is almost surely an optimal first-stage decision for Problem~\eqref{prob:mp:first_stage}. 
\end{theorem}
 Before presenting the proof of the above theorem in the subsequent sections, let us first discuss and interpret its necessary conditions. 

Condition [\ref{ass:linear_epsilon}] says that we have chosen to use linear decision rules in the multi-policy approximation, and will choose the radius of the uncertainty sets to decrease to zero as more data is obtained. Both of these conditions are chosen by the practitioner. We remark that this condition does not preclude the possibility that the sequence of $\{\epsilon_N \}_{N \in \N}$ is chosen after observing the historical data. 

Conditions [\ref{ass:v_finite}-\ref{ass:saa}] are probabilistic assumptions on the underlying two-stage stochastic linear optimization problem. Since the probability distribution of the stochastic problem is presumed to be unknown, it is generally not possible to verify these conditions in practice. On the other hand, the probabilistic conditions [\ref{ass:v_finite}-\ref{ass:saa}] are standard and consistent with the existing literature on Problem~\eqref{prob:mp:saa} (see Section~\ref{sec:mp:proof_thm}). We note that any problem that does not have first-stage decisions ($n=0$) can be easily modified to satisfy condition~[\ref{ass:saa}] by introducing dummy variables and constraints (\eg $x_1 \in \R$ and $0 \le x_1 \le 1$). 

Condition [\ref{ass:feas}] essentially stipulates that there exists a first-stage decision which is feasible for the multi-policy approximation with linear decision rules on any dataset. At first glance this condition might appear to be limiting; nonetheless, despite our efforts, we have been unable to find an instance of two-stage stochastic linear optimization which satisfies [\ref{ass:v_finite}-\ref{ass:saa}] and $\{\bx \in \R^n: Q(\bx,\bzeta) < \infty \; \forall \bzeta \in \Xi\} \neq \emptyset$ but  does not satisfy [\ref{ass:feas}].
In Section~\ref{sec:mp:asymptotic:L_Feasibility}, we present necessary and sufficient conditions for [\ref{ass:feas}] and show an example of how those conditions can be applied.

Theorem~\ref{thm:mp:conv} may be viewed as attractive due to its generality. In particular, 
Theorem~\ref{thm:mp:conv} does not require boundedness of the set of \emph{feasible} first-stage decisions nor boundedness of the polyhedron $\Xi$,  and holds for problems without relative complete recourse. As a result, we believe that the theorem is sufficiently general to encompass the vast array of two-stage stochastic linear optimization problems that are encountered in practice. 

The reminder of this section is organized as follows. 
In Section~\ref{sec:mp:asymptotic:L_Feasibility}, we discuss condition [\ref{ass:feas}] in greater detail, and provide verifiable sufficient and necessary conditions for this condition to hold. In Section \ref{sec:mp:lemmas}, we provide and prove two intermediary lemmas. In Section~\ref{sec:mp:proof_thm}, we combine the results from the previous sections to prove Theorem~\ref{thm:mp:conv}. 


\subsection{Sufficient (and necessary) conditions for [\ref{ass:feas}]} \label{sec:mp:asymptotic:L_Feasibility}
Out of the four conditions used in Theorem~\ref{thm:mp:conv} to establish asymptotic optimality guarantees for the multi-policy approximation, [\ref{ass:feas}] is perhaps the most opaque. In this section, we present {sufficient} (and in some cases {necessary}) conditions for [\ref{ass:feas}], and show how they can be used to verify  [\ref{ass:feas}] in an example. 
In the process of establishing these conditions, we present Definition~\ref{def:Q_C} and Lemma~\ref{lem:mp:feas_then_C_feas},  which will also be used in Sections~\ref{sec:mp:lemmas} and \ref{sec:mp:proof_thm}.

Our discussion begins with the following simple but insightful sufficient condition for   [\ref{ass:feas}]: 

\begin{proposition} \label{prop:mp:ldr_first}
	Condition~[\ref{ass:feas}] holds if there exists a first-stage decision $\bx \in \R^n$ and a \emph{single} linear decision rule $\by(\cdot) \in \mathcal{L}$ which satisfy $\bbt \bx + \bbw \by(\bzeta) \ge \bh(\bzeta)$ for all  $\bzeta \in \Xi.$ 
\end{proposition}
\begin{proof}{Proof. }
	This result follows immediately from the definition of [\ref{ass:feas}]. 
	\Halmos \end{proof}
Intuitively speaking, the above proposition says that condition~[\ref{ass:feas}] holds if the single-policy approximation with linear decision rules (see Section~\ref{sec:mp:single_policy}) is feasible for all realizations in the polyhedron $\Xi \subseteq \R^d$. The attractiveness of this sufficient condition is that it can be assessed in polynomial time by solving a robust optimization problem.\footnote{The condition provided in Proposition~\ref{prop:mp:ldr_first} can be verified by solving a robust optimization given by
\begin{equation*}
\begin{aligned}
&\underset{\bx \in \R^n, \; \by^0 \in \R^r, \; \bby \in \R^{r \times d}}{\textnormal{minimize}}&& 0 \\
&\textnormal{subject to}&& \bbt \bx + \bbw (\by^0 + \bby \bzeta) \ge \bh^0 + \bbh \bzeta, \quad \forall \bzeta \in \Xi. 
\end{aligned}
\end{equation*}
Since $\Xi \subseteq \R^d$ is a polyhedron, it follows from standard reformulation techniques that the robust optimization problem can be solved in   polynomial time.} 

	Nonetheless, the condition presented in Proposition~\ref{prop:mp:ldr_first} is {not} a \emph{necessary} condition for [\ref{ass:feas}]. In other words, it is possible to construct problem instances which satisfy [\ref{ass:feas}] but do not satisfy the conditions in Proposition~\ref{prop:mp:ldr_first}.  To demonstrate this, we first provide an example in which the condition from Proposition~\ref{prop:mp:ldr_first} does {not} hold. Then, at the end of the present Section~\ref{sec:mp:asymptotic:L_Feasibility}, we will revisit this example and show that it indeed satisfies [\ref{ass:feas}].   
\begin{example}\label{ex:mp:5}
	Consider the following problem wherein we seek a decision rule $y(\bzeta)$ which satisfies 
	\begin{align}
	\begin{aligned}\label{eq:ex5_1}
	y(\bzeta) &\geq \zeta_1-\zeta_2 \\
	y(\bzeta) &\geq \zeta_2-\zeta_1 \\
	y(\bzeta) &\leq \zeta_1+\zeta_2+2 \\
	y(\bzeta) &\leq -\zeta_2-\zeta_1+2,\\
	&\quad \forall \bzeta \in \Xi =\{\bzeta\in\R^2: \norm{\bzeta}_\infty\leq 1\}
	\end{aligned}
	\end{align}
	We readily observe that $y(\bzeta)=|\zeta_1-\zeta_2|$ is a feasible decision rule to \eqref{eq:ex5_1}. However, \citet[Proposition 3]{bertsimas2018adaptive} show there does not exist a linear decision rule which satisfies \eqref{eq:ex5_1}. Therefore, the sufficient conditions of Proposition~\ref{prop:mp:ldr_first} are not satisfied in this example.   \halmos
\end{example}	

We next present a stronger result (Lemma~\ref{lem:mp:feas_then_C_feas}) which provides {necessary} and {sufficient} conditions for [\ref{ass:feas}]. 
The statement of the lemma 
requires the following notation.
\begin{definition} \label{def:Q_C}
For any first-stage decision $\bx \in \R^n$, realization $\hat{\bzeta} \in \Xi$, radius $\epsilon > 0$, and set of recourse matrices $\mathcal{C} \subseteq \R^{r \times d}$, 
\begin{align*}
\begin{aligned}
Q^{\mathcal{C}}_{\epsilon}(\bx,\hat{\bzeta}) \triangleq \quad & \underset{\by^0 \in \R^r, \; \bby \in \mathcal{C}}{\text{minimize}}&& \max_{\bzeta \in B({\hat{\bzeta} },\epsilon) \cap \Xi} \bq^\intercal (\by^0 + \bby \bzeta)  \\
& \text{subject to}&&\bbt \bx + \bbw(\by^0 + \bby \bzeta) \geq \bh(\bzeta) ,\;\forall \bzeta \in  B({\hat{\bzeta} },\epsilon) \cap \Xi.
\end{aligned}
\end{align*}
\end{definition}
The quantity $Q^{\mathcal{C}}_{\epsilon}(\bx,\hat{\bzeta})$ is understood as the second-stage cost of the multi-policy approximation when $\Pi$ is a restricted space of linear decision rules, namely, the space of linear decision rules 
with recourse matrices which are elements of $\mathcal{C} \subseteq \R^{r \times d}$. 
For the sake of developing intuition, we note that the objective function in the multi-policy approximation may be stated equivalently as $\hat{V}^{\textnormal{MP}}_N(\bx) = \bc^\intercal \bx + \frac{1}{N} \sum_{i=1}^N Q^{\R^{r \times d}}_{\epsilon_N}(\bx,{\bxi}{}^i)$ when $\Pi = \mathcal{L}$.

In view of the above definition, we now present   two necessary and sufficient conditions for [\ref{ass:feas}]. 

\begin{lemma} \label{lem:mp:feas_then_C_feas}
	Conditions~[\ref{ass:feas}], [\ref{lem:mp:feas:II}], and [\ref{lem:mp:feas:III}] are equivalent. 
	\begin{enumerate}[label={\bf [A\arabic*]},ref=A\arabic*]
	\setcounter{enumi}{4}
		\item There exists a first-stage decision $\bx\in \Real^n $ and a radius $\epsilon>0$ for which the following hold:\label{lem:mp:feas:II}
	\begin{enumerate}[(a)]
			\item $Q(\bx,\bzeta)<\infty$ for all $\bzeta\in\Xi$.
			\item For each  minimal face $F$ of $\Xi$, there exists a recourse matrix $\bby^F\in\Real^{r\times d}$ such that $Q{}^{\left\{\bby^F\right\}}_{{\epsilon}}(\bx,\bzeta)<\infty$ for all $\bzeta\in F$.
		\end{enumerate}
		\item There exists a first-stage decision $\bx\in \Real^n $, a radius $\epsilon>0$, and a finite set of recourse matrices $\mathcal{C}\subseteq\Real^{r\times d}$  such that  $Q_{\epsilon}^\mathcal{C}(\bx,\bzeta)<\infty$ for all $\bzeta\in \Xi$.\label{lem:mp:feas:III}
	\end{enumerate}
\end{lemma}
\begin{proof}{Proof.}
 See Appendix~\ref{appx:C_feas}. \Halmos
\end{proof}

Let us interpret the conditions developed in Lemma~\ref{lem:mp:feas_then_C_feas}. 
Regarding [\ref{lem:mp:feas:II}], we recall that if a polyhedron $\Xi \subseteq \R^d$ has at least one extreme point, then the extreme points of the polyhedron are its minimal faces \citep[Section 3]{Conforti2014}. Therefore, if there exists a first-stage decision $\bx \in \R^n$ which satisfies $Q(\bx,\bzeta) < \infty$ for all $\bzeta \in \Xi$, and if the polyhedron $\Xi \subseteq \R^d$ has at least one extreme point, then [\ref{lem:mp:feas:II}] can be verified by checking if there are linear decision rules which are feasible in a ball around each extreme point of $\Xi$. 
We illustrate this procedure for verifying condition~[\ref{lem:mp:feas:II}] (and ultimately [\ref{ass:feas}]) by  returning to Example~\ref{ex:mp:5}.
\begin{repeatexample}[Example~\ref{ex:mp:5}, continued.]
	We recall that the decision rule $y(\bzeta)=|\zeta_1-\zeta_2|$ satisfies  the constraints \eqref{eq:ex5_1}, and there is no first-stage decision in this problem. We therefore conclude that condition [\ref{lem:mp:feas:II}(a)]  is satisfied. Next,  we observe that the extreme points of $\Xi$ are given by $$\ext(\Xi) = \{(1,1),(-1,1),(1,-1),(-1,-1) \}.$$
	Consider the radius $\epsilon= {1}/{\sqrt{2}}$ and observe that, for each extreme point $\hat{\bzeta}$ of $\Xi$, the set  $
	B(\hat{\bzeta},\epsilon) \cap \Xi$ 
	is contained within the simplex formed by the convex hull of $\hat{\bzeta}$ and its two adjacent extreme points. In other words, for each extreme point $\hat{\bzeta}$, the set $B(\hat{\bzeta},\epsilon) \cap \Xi$ is contained in a simplex defined as
	\begin{align*}
	P_{\bzeta} &\triangleq \conv \left( \left \{ \begin{pmatrix} \hat{\zeta}_1 \\ \hat{\zeta}_2 \end{pmatrix},\begin{pmatrix} -\hat{\zeta}_1 \\ \hat{\zeta}_2 \end{pmatrix},\begin{pmatrix} \hat{\zeta}_1 \\ -\hat{\zeta}_2 \end{pmatrix}
	\right \} \right).
	\end{align*}
	We also observe that the extreme points are affinely independent.  
	Therefore, it follows from identical reasoning as \citet[Theorem 1]{Bertsimas2012poweraffine} that there exists a {linear} decision rule which satisfies the constraints \eqref{eq:ex5_1} for all  $\bzeta \in P_{\hat{\bzeta}}$. Since the set $B(\hat{\bzeta},\epsilon) \cap \Xi$ is contained in the simplex $P_{\hat{\bzeta}}$ for each extreme point, we have shown that condition~[\ref{lem:mp:feas:II}(b)] holds.  We conclude that condition~[\ref{lem:mp:feas:II}] holds, and since Lemma~\ref{lem:mp:feas_then_C_feas} implies that conditions~[\ref{lem:mp:feas:II}] and [\ref{ass:feas}] are equivalent, we conclude that condition [\ref{ass:feas}] holds as well. 
	\halmos \end{repeatexample}	

In summary, we have offered two potential procedures for verifying condition~[\ref{ass:feas}]. The first procedure, and by far the simplest, is to check the sufficient condition in Proposition~\ref{prop:mp:ldr_first} by solving a robust optimization problem. For problems in which this first procedure does not produce an affirmative conclusion, a second (albeit more complex) procedure can potentially be undertaken by checking condition~[\ref{lem:mp:feas:II}]. Using the second procedure, we have provided evidence    (Example~\ref{ex:mp:5}) that condition~[\ref{ass:feas}] can hold even when a two-stage problem does not have a linear decision rule which is feasible for all realizations in the polyhedron $\Xi$. 
This result is viewed as positive, as it indicates that [\ref{ass:feas}] may be a mild condition.  
Finally, as discussed at the end of the previous section, we have not found a two-stage problem which satisfies [\ref{ass:v_finite}-\ref{ass:saa}] and $\{\bx \in \R^n: Q(\bx,\bzeta) < \infty \; \forall \bzeta \in \Xi\} \neq \emptyset$ but  does not satisfy condition [\ref{ass:feas}]. 
The question of whether these conditions are equivalent is open for future research.

While it was not utilized in this section,  [\ref{lem:mp:feas:III}] from Lemma~\ref{lem:mp:feas_then_C_feas} offers another necessary and sufficient condition for [\ref{ass:feas}]. Condition~[\ref{lem:mp:feas:III}] has a similar interpretation as [\ref{ass:feas}], except that the former imposes a feasibility requirement which must hold for linear decision rules with a finite number of recourse matrices. This necessary and sufficient condition for [\ref{ass:feas}] will play an important role in the proof found in Section~\ref{sec:mp:proof_thm}. 

\subsection{Intermediary lemmas} \label{sec:mp:lemmas}
Our proof of Theorem~\ref{thm:mp:conv}, found in Section~\ref{sec:mp:proof_thm}, is the culmination of three intermediary results (Lemmas~\ref{lem:mp:feas_then_C_feas}, \ref{lem:mp:feas}, and \ref{lem:mp:bound}). The first of those results, Lemma~\ref{lem:mp:feas_then_C_feas}, was presented in the previous section. In the present section, we state and prove  Lemmas~\ref{lem:mp:feas} and \ref{lem:mp:bound}. 

To simplify our exposition, we assume throughout Section~\ref{sec:mp:lemmas} that conditions~[\ref{ass:v_finite}-\ref{ass:saa}] hold.  Let $\bx^* \in \R^n$ denote an optimal first-stage decision for Problem~\eqref{prob:mp:first_stage}, the existence of which follows directly from condition [\ref{ass:saa}]. Moreover, let the set of realizations $\bzeta \in \Xi$ which are feasible for the optimal first-stage decision be denoted by  
\begin{align*}
{\Xi}^* \triangleq \left \{ \bzeta \in \Xi: \; \text{there exists } \by \in \R^r \text{ such that } \bbt \bx^* + \bbw \by \ge \bh(\bzeta) \right \}.
\end{align*}

We now present the first result of this section (Lemma~\ref{lem:mp:feas}). Recall the notation of $Q^{\mathcal{C}}_\epsilon(\bx,\hat{\bzeta})$ from the previous section (Definition~\ref{def:Q_C}). 
The purpose of Lemma~\ref{lem:mp:feas} is to show, for a collection of recourse matrices $\mathcal{C} \subseteq \R^{r \times d}$, that there exists a first-stage decision $\bx' \in \R^n$ arbitrarily close to $\bx^*$ which satisfies $Q^{\mathcal{C}}_{{\epsilon}}(\bx',\bzeta) < \infty$ for all $\bzeta \in \Xi^*$. 
\begin{lemma} \label{lem:mp:feas}
Let $\bx\in \Real^n$, $\mathcal{C}\subseteq\Real^{r\times d}$, and $\bar{\epsilon}>0$ satisfy $Q^{\mathcal{C}}_{\bar{\epsilon}}(\bx,\bzeta) < \infty$ for all $\bzeta\in \Xi^*$. Then,
for every $\lambda\in(0,1)$ there exists an $\epsilon>0$ such that $Q^{\mathcal{C}}_{\epsilon}(\lambda \bx^*+(1-\lambda)\bx,\bzeta) < \infty$ for all $\bzeta\in \Xi^*$.
\end{lemma}
\begin{proof}{Proof.}
Fix $\lambda \in (0,1)$, and consider any realization $\hat{\bzeta} \in \Xi^*$. 
In  Appendix~\ref{appx:shrinkage}, we show that there exists a radius ${\epsilon} > 0$ (depending only on $\lambda$, $\bar{\epsilon}$, $\Xi$, and $\Xi^*$) 
and a realization $\bar{\bzeta} \in \Xi^*$ such that 
\begin{align}
B(\hat{\bzeta},{\epsilon}) \cap \Xi \subseteq  \{\lambda \bzeta +(1-\lambda) \bar{\bzeta}:B(\bar{\bzeta},\bar{\epsilon})\cap \Xi\}.
\label{line:mp:shrinkage_used}
\end{align}
Since $\bar{\bzeta} \in \Xi^*$, it follows that  $Q^{\mathcal{C}}_{\bar{\epsilon}}({\bx},\bar{\bzeta}) < \infty$, and thus there exists a ${\by}{}^0 \in \R^n$ and ${\bby} \in \mathcal{C}$ such that
\begin{align}
\bbt {\bx} + \bbw ({\by}^{0} + {\bby} \bzeta) \ge \bh(\bzeta), \quad \forall \bzeta \in B(\bar{\bzeta},\bar{\epsilon}) \cap \Xi. \label{line:mp:x_hat_feasible}
\end{align}
Moreover, since $\bar{\bzeta} \in \Xi^*$, by definition of $\Xi^*$ there exists a vector $\by^{0,*} \in \R^r$ such that
\begin{align}
\bbt {\bx^*} + \bbw \by^{0,*} \ge \bh(\bar{\bzeta}). \label{line:mp:x_star_feasible}
\end{align}
We now take a convex combination of lines \eqref{line:mp:x_hat_feasible} and \eqref{line:mp:x_star_feasible}:
\begin{align*}
\lambda \left( \bbt {\bx} + \bbw ({\by}^{0} + {\bby} \bzeta) \right)  + (1-\lambda) \left( \bbt \bx^*  + \bbw {\by}^{0,*} \right)  &\ge \lambda \bh(\bzeta) + (1-\lambda) \bh(\bar{\bzeta}) , \quad \forall \bzeta \in B(\bar{\bzeta},\bar{\epsilon}) \cap \Xi.
\end{align*}
Rearranging the left side of the inequality, and applying the linearity of $\bh(\cdot)$, the above can be rewritten as
\begin{align*}
 \bbt (\lambda \bx+(1-\lambda)\bx^*) + \bbw \left(   \lambda {\by}^{0}+(1-\lambda) {\by}^{0,*} -(1- \lambda) {\bby} \bar{\bzeta}+  {\bby} \left( \lambda \bzeta +(1-\lambda) \bar{\bzeta}  \right)\right) &\ge \bh( \lambda \bzeta +(1-\lambda) \bar{\bzeta}), \\
 &\quad  \forall \bzeta \in B(\bar{\bzeta},\bar{\epsilon}) \cap \Xi.
\end{align*}
Defining $\hat{\by}^{0} \triangleq  \lambda {\by}^{0}+ (1-\lambda) {\by}^{0,*}  - (1-\lambda) \hat{\bby} \bar{\bzeta}$,
and applying \eqref{line:mp:shrinkage_used}, we have shown that
\begin{align*}
 \bbt (\lambda \bx+(1-\lambda)\bx^*) + \bbw \left( \hat{\by}^{0}+  {\bby} \bzeta \right) &\ge \bh( \bzeta ), \quad \forall \bzeta \in B(\hat{\bzeta},{\epsilon}) \cap \Xi.
\end{align*}
The above shows that $Q^{\mathcal{C}}_{\epsilon}(\lambda \bx + (1-\lambda)\bx^*,\hat{\bzeta}) < \infty$. Since $\hat{\bzeta} \in \Xi^*$ was chosen arbitrarily,  and since the choice of $\epsilon$ did not depend on $\hat{\bzeta}$, our proof is complete.  
\Halmos \end{proof}

The final result of this section, Lemma~\ref{lem:mp:bound}, requires the following additional notation. Define the following polyhedral set: 
\begin{align*}
D \triangleq \{ \bdelta \in \R^{m}: \bbw^\intercal \bdelta = \bq, \; \bdelta \ge \bzero \}.
\end{align*}
We observe that $D$ is the feasible set associated with the dual problem of the second stage,  
\begin{align}
Q(\bx,\bzeta) =\min_{ \by\in\R^r}\left\{\bq^\top\by : \bbw\by\geq \bh(\bzeta)-\bbt\bx\right\}= \max_{\blambda\in D} \blambda^\intercal(\bh(\bzeta)-\bbt\bx), \label{line:poop}
\end{align}
where the second equality follows from strong duality. 
Condition~[\ref{ass:v_finite}] implies that 
$D\neq \emptyset$ and has at least one extreme point\footnote{Condition~\ref{ass:v_finite} says that $v^*$ is finite, and thus $Q(\bx^*,\bxi) > -\infty$ almost surely. Therefore, it follows from weak duality for linear optimization that  
\begin{align*}
- \infty < Q(\bx^*, \bxi ) &\le \max_{\bdelta \in D} \left \{ \bdelta^\intercal \left(\bh({\bxi}) - \bbt {\bx} \right) \right \} \quad \text{almost surely},
\end{align*}
and thus $D$ must be nonempty. Since the set satisfies $D\neq \emptyset$ and is in standard form, it has at least one extreme point.},
and we denote its set of extreme points by $\ext(D)$. 

We now present the final intermediary lemma, Lemma~\ref{lem:mp:bound}, which  provides an upper bound on the gap between the second-stage cost of the multi-policy approximation and that of the stochastic problem: 
\begin{lemma} \label{lem:mp:bound}
	Let $\bx\in\Real^n$, $\mathcal{C}\subseteq\Real^{r\times d}$, and ${\epsilon}\ge0$ satisfy $Q^{\mathcal{C}}_{{\epsilon}}(\bx,\bzeta)<\infty$ for all $\bzeta\in \Xi^*$. Then, 
	\begin{align*}Q^{\mathcal{C}}_\epsilon(\bx,\bzeta)  \le Q(\bx^*,\bzeta) + \epsilon \eta^{\mathcal{C}} + L\| \bx - \bx^*\|_*\quad \forall \bzeta\in \Xi^*,\end{align*}
	where \begin{align*}
	\eta^{\mathcal{C}} \triangleq  \sup_{\bby \in \mathcal{C},\;\bdelta \in \ext(D)}\left \{ \left \| {\bby}^\intercal \bq \right\|_* + \sum_{j=1}^m \delta_j  \left \|\bh_j - {\bby}^\intercal \bw_j  \right \|_* \right\} \;\;\text{and}\;\; L \triangleq \max_{\bdelta \in \ext(D)} \left \| \bbt^\intercal \bdelta \right \|.
	\end{align*}
\end{lemma}
\begin{proof}{Proof.}
	We begin by showing that the quantities $Q^{\mathcal{C}}_\epsilon(\bx,\hat{\bzeta})$, $Q(\bx,\hat{\bzeta})$, and $Q(\bx^*,\hat{\bzeta})$ are finite for all realizations $\hat{\bzeta}\in \Xi^*$. Indeed, it follows immediately from the definition of $\Xi^*$ and the construction of $\bx$, $\mathcal{C}$, and $\epsilon$ that 
	\begin{align*}
	Q^{\mathcal{C}}_\epsilon(\bx,\hat{\bzeta}), \;Q(\bx^*,\hat{\bzeta}) < \infty \quad \forall \hat{\bzeta} \in \Xi^*. 
	\end{align*}
	Furthermore, since the polyhedron $D \subseteq \R^m$ is nonempty, it follows from line~\eqref{line:poop} that
	\begin{align*}
		Q(\bx,\hat{\bzeta}), Q(\bx^*,\hat{\bzeta}) &>-\infty \quad \forall \hat{\bzeta} \in \Xi^*.
			\end{align*}
	 Since Definition~\ref{def:Q_C} implies that the inequality 
	$Q(\bx,\hat{\bzeta}) \le Q^{\mathcal{C}}_{\epsilon}(\bx,\hat{\bzeta})$ holds for all realizations $\hat{\bzeta}\in \Xi^*$, we have shown that the quantities $Q^{\mathcal{C}}_\epsilon(\bx,\hat{\bzeta})$, $Q(\bx,\hat{\bzeta})$, and $Q(\bx^*,\hat{\bzeta})$ are finite for all realizations $\hat{\bzeta}\in \Xi^*$.  
	
	
	Now consider any arbitrary realization $\hat{\bzeta} \in \Xi^*$. 
	Since $Q^{\mathcal{C}}_\epsilon(\bx,\hat{\bzeta})$ is finite, we observe that there exists a recourse matrix $\tilde{\bby} \in \mathcal{C}$ such that the following optimization problem has a finite optimal cost:
	\begin{align}
	 \left \{ 
	\begin{aligned}
	& \underset{\by^0 \in \R^{r}}{\text{minimize}}&& \max_{\bzeta \in B({\hat{\bzeta}},\epsilon) \cap \Xi} \bq^\intercal \left(\by^0 + \tilde{\bby} \bzeta \right) \\
	& \text{subject to}&&\bbt {\bx} + \bbw \left(\by^0 + \tilde{\bby} \bzeta \right) \geq \bh(\bzeta) ,\;\forall \bzeta \in  B({\hat{\bzeta}}, \epsilon) \cap \Xi.
	\end{aligned} \right. \label{line:mp:nice_ub}
	\end{align}
	In particular,  we note that the optimal cost of \eqref{line:mp:nice_ub} is an upper bound on $Q^{\mathcal{C}}_\epsilon (\bx,\hat{\bzeta}) > - \infty$, and the existence of such a recourse matrix follows from $Q^{\mathcal{C}}_\epsilon (\bx,\hat{\bzeta}) < \infty$.
	
	 With our goal of obtaining an upper bound on \eqref{line:mp:nice_ub}, we develop the following intermediary result: 
	\begin{align}
	&\left \{ \begin{aligned} & \underset{\by^0 \in \R^{r}}{\text{minimize}}&& \bq^\intercal\by^0  \\
	& \text{subject to}&&\bbt \bx + \bbw \left(\by^0 + \tilde{\bby} \bzeta \right) \geq \bh(\bzeta) ,\;\forall \bzeta \in  B(\hat{\bzeta},\epsilon) \cap \Xi
	\end{aligned} \right. \notag \\
	&= \left \{ \begin{aligned} & \underset{\by^0 \in \R^{r}}{\text{minimize}}&& \bq^\intercal\by^0  \\
	& \text{subject to}&&\boldsymbol{t}_j^\intercal \bx + \bw_j^\intercal \by^0 - h_j^0 \geq \max_{\bzeta  \in  B(\hat{\bzeta},\epsilon) \cap \Xi} ( \bh_j - \tilde{\bby}^\intercal \bw_j )^\intercal \bzeta \quad \forall j \in [m]
	\end{aligned} \right .\notag\\ 
	&= \max_{\bdelta \in D} \left \{ \sum_{j=1}^m \delta_j \left( h_j^0 - \boldsymbol{t}_j^\intercal \bx +  \max_{\bzeta  \in  B(\hat{\bzeta},\epsilon) \cap \Xi} ( \bh_j - \tilde{\bby}^\intercal \bw_j )^\intercal \bzeta\right)\right \}\notag\\
	&= \max_{\bdelta \in \ext(D)} \left \{ \sum_{j=1}^m \delta_j \left( h_j^0 - \boldsymbol{t}_j^\intercal \bx +  \max_{\bzeta  \in  B(\hat{\bzeta},\epsilon) \cap \Xi} ( \bh_j - \tilde{\bby}^\intercal \bw_j )^\intercal \bzeta\right)\right \}\notag\\
	&\le \max_{\bdelta \in \ext(D)} \left \{ \sum_{j=1}^m \delta_j \left( h_j^0 - \boldsymbol{t}_j^\intercal \bx + \max_{\bzeta  \in  B(\hat{\bzeta},\epsilon)} ( \bh_j - \tilde{\bby}^\intercal \bw_j )^\intercal \bzeta\right)\right \}\notag \\
	&=\max_{\bdelta \in \ext(D)} \left \{  \bdelta^\intercal \left( \bh(\hat{\bzeta}) - \bbt \bx - \bbw \tilde{\bby} \hat{\bzeta} \right) + \epsilon \sum_{j=1}^m \delta_j  \left \|\bh_j - \tilde{\bby}^\intercal \bw_j  \right \|_* \right \}. \label{line:mp:second_bound_mp}
	\end{align}
	Indeed, the first equality is a standard reformulation technique in robust optimization. The second equality stems from strong duality for linear optimization, which holds since the problem has a finite objective value. The third equality follows from the fundamental theorem of linear programming, which holds because the set of extreme points of $D$ is nonempty and the problem has a finite optimal cost. The inequality holds because we have removed constraints from the inner maximization problems. The final equality follows directly from the definition of the dual norm.  
	
	We now combine \eqref{line:mp:nice_ub} and  \eqref{line:mp:second_bound_mp}:
	\begin{align} 
	Q^{\mathcal{C}}_\epsilon (\bx,\hat{\bzeta}) \; 
	&\le \; \; \max_{\bzeta \in B({\hat{\bzeta}},\epsilon) \cap \Xi} \bq^\intercal \tilde{\bby} \bzeta \; \; + \;\; \left \{ \begin{aligned}
	& \underset{\by^0 \in \R^{r}}{\text{minimize}}&& \bq^\intercal\by^0  \\
	& \text{subject to}&&\bbt\bx + \bbw \left(\by^0 + \tilde{\bby} \bzeta \right) \geq \bh(\bzeta) ,\;\forall \bzeta \in  B({\hat{\bzeta}},\epsilon) \cap \Xi.
	\end{aligned} \right.  \notag \\
	&\le \bq^\intercal\tilde{\bby}\bzeta+\epsilon\norm{\tilde{\bby}^\intercal\bq}_* + \max_{\bdelta \in \ext(D)} \left \{  \bdelta^\intercal \left( \bh(\hat{\bzeta}) - \bbt \bx - \bbw \tilde{\bby} \hat{\bzeta} \right) + \epsilon \sum_{j=1}^m \delta_j  \left \|\bh_j - \tilde{\bby}^\intercal \bw_j  \right \|_* \right \} \notag \\
	&= \max_{\bdelta \in \ext(D)} \left \{ \left( \bq - \bbw^\intercal \bdelta\right)^\intercal \tilde{\bby} \hat{\bzeta} + \bdelta^\intercal \left( \bh(\hat{\bzeta}) - \bbt \bx \right) + \epsilon \left(\norm{\tilde{\bby}^\intercal\bq}_* +  \sum_{j=1}^m \delta_j  \left \|\bh_j - \tilde{\bby}^\intercal \bw_j  \right \|_* \right) \right \} \notag \\
	&= \max_{\bdelta \in \ext(D)} \left \{  \bdelta^\intercal \left( \bh(\hat{\bzeta}) - \bbt \bx \right) + \epsilon \left(\norm{\tilde{\bby}^\intercal\bq}_* +  \sum_{j=1}^m \delta_j  \left \|\bh_j - \tilde{\bby}^\intercal \bw_j  \right \|_* \right) \right \}, \notag  \\
	&\le \max_{\bdelta \in \ext(D)} \left \{  \bdelta^\intercal \left( \bh({\hat{\bzeta}}) - \bbt {\bx} \right) \right \} +\epsilon  \max_{\bdelta \in \ext(D)} \left \{ \norm{\tilde{\bby}^\intercal\bq}_* + \sum_{j=1}^m \delta_j  \left \|\bh_j - \tilde{\bby}^\intercal \bw_j  \right \|_* \right \} \notag \\
	&\le \max_{\bdelta \in \ext(D)} \left \{  \bdelta^\intercal \left( \bh({\hat{\bzeta}}) - \bbt {\bx} \right) \right \} + \epsilon \max_{\bdelta \in \ext(D),\bby \in \mathcal{C}}\left \{ \left \| {\bby}^\intercal \bq \right\|_* + \sum_{j=1}^m \delta_j  \left \|\bh_j - {\bby}^\intercal \bw_j  \right \|_* \right\} \notag \\
	&= Q({\bx},{\hat{\bzeta}}) + \eta^\mathcal{C} \epsilon. \label{line:say_hi_to_chris_for_me_may_27_2020}
	\end{align}
	Indeed, the first inequality follows from separating out the objective function of \eqref{line:mp:nice_ub} and using the fact that the optimal cost of  \eqref{line:mp:nice_ub} is an upper bound on $Q^{\mathcal{C}}_\epsilon(\bx,\hat{\bzeta})$. The second inequality follows from the definition of the dual norm and from line~\eqref{line:mp:second_bound_mp}. 
	The first equality follows from rearranging terms, and the second equality follows because $\bbw^\intercal \bdelta = \bq$ for all $\bdelta \in D$. 
	The third inequality follows from separating the single maximization problem into two separate maximization problems. The fourth inequality follows from maximizing over the recourse matrices. The final equality follows from the definition of $\eta^\mathcal{C}$, line~\eqref{line:poop}, and the fundamental theorem of linear programming, which holds because $Q(\bx,\hat{\bzeta})$ is finite. 
	
		Finally, 
	\begin{align}
	Q(\bx,\hat{\bzeta}) - Q({\bx}^*,\hat{\bzeta})   &= \max_{\bdelta \in \textnormal{ext}(D)} \left \{ \bdelta^\intercal (\bh(\hat{\bzeta}) - \bbt {\bx}) \right \} - \max_{\bdelta \in \textnormal{ext}(D)} \left \{ \bdelta^\intercal (\bh(\bzeta) - \bbt {\bx}^*) \right \} \notag \\
	&\le  \max_{\bdelta \in \textnormal{ext}(D)} \left \{ \left(\bdelta^\intercal (\bh(\hat{\bzeta}) - \bbt {\bx}) \right) - \left( \bdelta^\intercal (\bh(\hat{\bzeta}) - \bbt {\bx}^*) \right) \right \} \notag \\
	&=  \max_{\bdelta \in \textnormal{ext}(D)} \left \{ \bdelta^\intercal \bbt ({\bx} - {\bx}^*)\right \} \notag  \\
	&\le  \left \| {\bx} - {\bx}^* \right \|_* \max_{\bdelta \in \ext(D)} \left \| \bbt^\intercal \bdelta \right \| \notag \\
	&\le L \| \bx - \bx^* \|_*. 
	\label{line:mp:bound_on_saa_for_x_1_x_2}
	\end{align}
	Indeed, the first equality follows from line~\eqref{line:poop} and the fundamental theorem of linear programming, which holds because  $Q(\bx,\hat{\bzeta})$ and $Q(\bx^*,\hat{\bzeta})$ are finite and have at least one extreme point. The first inequality follows from using the same maximizer from both maximum problems, the second inequality follows from the definition of the dual norm, and the third inequality follows from the definition of $L$. 
	
	Combining \eqref{line:say_hi_to_chris_for_me_may_27_2020} and \eqref{line:mp:bound_on_saa_for_x_1_x_2}, and since $\hat{\bzeta} \in \Xi^*$ was chosen arbitrarily, our proof is complete. 
	\Halmos \end{proof}


\subsection{Proof of Theorem~\ref{thm:mp:conv}} \label{sec:mp:proof_thm}
In this section we present the proof of Theorem~\ref{thm:mp:conv}. We first discuss the proof from a broad viewpoint, focusing on its overarching strategy and highlighting its key steps, and then utilize the technical intermediary lemmas from the previous sections to prove the theorem. 

Theorem~\ref{thm:mp:conv} is essentially a combination of two results. The first result, presented in line~\eqref{line:mp:conv}, says that the optimal costs of two-stage sample robust optimization and its multi-policy approximation converge almost surely to the optimal cost of two-stage stochastic linear optimization. The second result establishes convergence guarantees for  the sequences of optimal first-stage decisions. It turns out, conveniently, that the second result will follow as a direct consequence of line~\eqref{line:mp:conv} together with a convergence guarantee of \cite{robinson1996analysis}  for near-optimal solutions of epiconvergent functions. 
Our efforts therefore focus on the proof of line~\eqref{line:mp:conv}.  

There are two main steps in our proof of line~\eqref{line:mp:conv}. The first step is to show that $\liminf_{N \to \infty} \hat{v}_N^{\textnormal{SRO}}$ and $\liminf_{N \to \infty} \hat{v}_N^{\textnormal{MP}}$ are upper bounds on $v^*$ almost surely. Fortunately this step is straightforward to show, as both quantities are upper bounds on $\liminf_{N \to \infty} \hat{v}_N^{\textnormal{SAA}}$, and it follows directly from results of \cite{king1991epi} and \cite{robinson1996analysis} that the optimal cost of the sample average approximation will converge almost surely to $v^*$ under conditions [\ref{ass:v_finite}-\ref{ass:saa}] (see below for details).  
The second step in our proof of line~\eqref{line:mp:conv}  is  to show  that $\limsup_{N \to \infty} \hat{v}_N^{\textnormal{SRO}}$ and $\limsup_{N \to \infty} \hat{v}_N^{\textnormal{MP}}$ are lower bounds on $v^*$ almost surely; this is proved by combining Lemmas~\ref{lem:mp:feas_then_C_feas}, \ref{lem:mp:feas}, and \ref{lem:mp:bound} from Sections~\ref{sec:mp:asymptotic:L_Feasibility} and \ref{sec:mp:lemmas}. 

In view of the above discussion, we now present the  proof of Theorem~\ref{thm:mp:conv}.

\begin{proof}{Proof of Theorem~\ref{thm:mp:conv}.}
	We begin by showing that  $\liminf_{N \to \infty} \hat{v}_N^{\textnormal{SRO}}$ and $\liminf_{N \to \infty} \hat{v}_N^{\textnormal{MP}}$ are upper bounds on $v^*$ almost surely.     Indeed, it follows directly from conditions [\ref{ass:v_finite}-\ref{ass:saa}],  \citet[Theorem 3.1]{king1991epi},  and \citet[Corollary 3.11]{robinson1996analysis} that
	$\lim_{N \to \infty} \hat{v}_N^{\textnormal{SAA}} = v^*$ almost surely. 
	Moreover, we readily observe that the inequality $\hat{v}_N^{\textnormal{SRO}} \ge \hat{v}_N^{\textnormal{SAA}}$ always holds, and Theorem~\ref{thm:mp:hierarchy} further established the inequality $\hat{v}_N^{\textnormal{MP}} \ge \hat{v}_N^{\textnormal{SRO}}$.  Therefore, 
	\begin{align*}
	v^*  \overset{ \text{a.s.}}{=} \lim_{N \to \infty} \hat{v}_N^{\textnormal{SAA}} \le \liminf_{N \to \infty} \hat{v}_N^{\textnormal{SRO}}  \le \liminf_{N \to \infty} \hat{v}_N^{\textnormal{MP}}.
	\end{align*}
	
	We next prove the other direction, namely,  that $\limsup_{N \to \infty} \hat{v}_N^{\textnormal{SRO}}$ and $\limsup_{N \to \infty} \hat{v}_N^{\textnormal{MP}}$ are lower bounds on $v^*$ almost surely. Indeed,  Lemma~\ref{lem:mp:feas_then_C_feas} shows that [\ref{ass:feas}] implies [\ref{lem:mp:feas:III}]. Following the definition of condition~[\ref{lem:mp:feas:III}], there exists a first-stage decision $\bar{\bx}\in\R^n$, a radius $\bar{\epsilon}>0$, and a finite set of recourse matrices $\mathcal{C}\subseteq\Real^{r\times d}$ such that
	$$Q_{\bar{\epsilon}}^\mathcal{C}(\bar{\bx},\bzeta)<\infty \quad \forall \bzeta\in  \Xi^*\subseteq\Xi.$$ 
	Choose any arbitrary $\lambda \in (0,1)$
	Then Lemma~\ref{lem:mp:feas} implies that there exists a radius $\epsilon > 0$ such that 
	$$Q^{\mathcal{C}}_{\epsilon}(\lambda\bar{\bx}+(1-\lambda)\bx^*,\bzeta) < \infty, \quad  \forall \bzeta \in \Xi^*.$$
	For notational convenience, let $\bx \triangleq \lambda\bar{\bx}+(1-\lambda)\bx^*$.
	Moreover, condition~[\ref{ass:linear_epsilon}] implies that there exists an integer $\bar{N} \in \N$ such that $\epsilon_N \le \epsilon$ for all $N \ge \bar{N}$. 
	Therefore,  
	\begin{align}
	\limsup_{N \to \infty} \hat{v}_N^{\textnormal{SRO}} \le \limsup_{N \to \infty} \hat{v}_N^{\textnormal{MP}} &\le \limsup_{N \to \infty} \hat{V}^{\textnormal{MP}}_N(\bx) \notag \\
	 &= \limsup_{N \to \infty} \left \{ \bc^\intercal \bx + \frac{1}{N} \sum_{i=1}^N Q^{\R^{r \times d}}_{\epsilon_N}(\bx,\bxi{}^i) \right \} \notag \\
	&\le \limsup_{N \to \infty} \left \{ \bc^\intercal \bx + \frac{1}{N} \sum_{i=1}^N Q^{\mathcal{C}}_{\epsilon_N}(\bx,\bxi{}^i) \right \} \notag\\
	&\le \limsup_{N \to \infty} \left \{ \bc^\intercal \bx + \frac{1}{N} \sum_{i=1}^N \left( Q(\bx^*,\bxi{}^i) + \epsilon_N \eta^{\mathcal{C}}  + L\| \bx - \bx^*\|_*\right)\right \} \quad \text{a.s.} \notag \\
	&= \limsup_{N \to \infty} \left \{ \bc^\intercal \bx + \frac{1}{N} \sum_{i=1}^N  Q(\bx^*,\bxi{}^i)  \right \} + L\| \bx - \bx^*\|_* \notag \\
	&= \limsup_{N \to \infty} \left \{ \bc^\intercal \bx + \frac{1}{N} \sum_{i=1}^N  Q(\bx^*,\bxi{}^i)  \right \} + L \lambda \| \bar{\bx} - \bx^*\|_* \notag \\
	&=v^* + L\lambda \| \bar{\bx} - \bx^*\|_*  \quad \text{a.s.} \notag 
	\end{align}
	Indeed, the first line follows from Theorem~\ref{thm:mp:hierarchy}, as well as the fact that the first-stage decision $\bx$ is potentially suboptimal for the multi-policy approximation. The second line follows from $\Pi = \mathcal{L}$, which is given by condition~[\ref{ass:linear_epsilon}]. The third line holds because the collection of recourse matrices $\mathcal{C}$ is a subset of $\R^{r \times d}$. The fourth line follows from Lemma~\ref{lem:mp:bound}, along with the observation that  $Q^{\mathcal{C}}_{\epsilon_N}(\bx,\bzeta) \le Q^{\mathcal{C}}_{\epsilon}(\bx,\bzeta) < \infty$ for all  $ N \ge \bar{N}$ and $\bzeta \in \Xi^*$, and holds  almost surely because $\Prb(\bxi^i \in \Xi^* \; \text{for all } i) = 1$. The fifth line follows from $\epsilon_N \to 0$ as $N \to \infty$, which is given by condition~[\ref{ass:linear_epsilon}]. The sixth line follows from the definition of $\bx$. The final line follows from the strong law of large numbers and the fact that $\bx^*$ is optimal for Problem~\eqref{prob:mp:first_stage}. Taking $\lambda \in (0,1)$ to be arbitrarily close to 0, we conclude that 
		\begin{align*}
	 \limsup_{N \to \infty} \hat{v}_N^{\textnormal{SRO}}  \le \limsup_{N \to \infty} \hat{v}_N^{\textnormal{MP}} \overset{ \text{a.s.}}{\le}  v^* 
	\end{align*}
	Combining the above results, our proof of line~\eqref{line:mp:conv} is complete. 
	
		
	We now establish convergence of the optimal first-stage decisions. Let $\{ \hat{\bx}_N^{\textnormal{MP}} \}_{N\in\N}$ be a sequence of optimal first-stage decisions for Problem~\eqref{prob:mp:multi_policy}.
	Then it follows from line \eqref{line:mp:conv} that
	\begin{align*}
	v^* \overset{\text{a.s.}}{=} \lim_{N \to \infty} \hat{v}_N^{\textnormal{MP}} = \lim_{N \to \infty} \hat{V}_N^{\textnormal{MP}}(\hat{\bx}_N^{\textnormal{MP}}) \ge  \lim_{N \to \infty} \hat{V}_N^{\textnormal{SAA}}(\hat{\bx}_N^{\textnormal{MP}}) \ge  \lim_{N \to \infty} \hat{v}_N^{\textnormal{SAA}} \overset{\text{a.s.}}{=} v^*.
	\end{align*}
	The above implies that, almost surely, there is a sequence of nonnegative numbers $\{ \eta_N \}$ converging to zero which satisfy $\hat{V}_N^{\textnormal{SAA}}(\hat{\bx}_N^{\textnormal{MP}}) \le \hat{v}_N^{\textnormal{SAA}} + \eta_N$ for all $N \in \N$.  Therefore, it follows directly from conditions~[\ref{ass:v_finite}-\ref{ass:saa}] and \citet[Proposition 2.1 and Theorem 3.1]{king1991epi} that any accumulation point of the sequence $\{ \hat{\bx}_N^{\textnormal{MP}} \}_{N\in\N}$ is an optimal first-stage decision for Problem~\eqref{prob:mp:first_stage} almost surely. The analogous convergence result for optimal first-stage decisions of Problem~\eqref{prob:mp:sro} follows by identical reasoning and is thus omitted. 
	\halmos 
\end{proof}


\section{Computational Experiments} \label{sec:mp:experiments}
In this section, we demonstrate the practical value of solving two-stage sample robust optimization using the multi-policy approximation. More specifically, we analyze the computational tractability and out-of-sample performance of the following various data-driven approaches: 
\begin{enumerate}
	\item {\bf MP Affine} - The multi-policy approximation  for two-stage sample robust optimization using linear decision rules, where uncertainty sets in the SRO formulation use the $\ell_2$ norm.  
	\item {\bf SP Affine} - The single-policy approximation for two-stage sample robust optimization using linear decision rules, where uncertainty sets in the SRO formulation use the $\ell_2$ norm.  
	\item {\bf Wass SDP} - The semidefinite (SDP) conic approximation of \citet{hanasusanto2016conic} for two-stage distributionally robust optimization with the type-2 Wasserstein ambiguity set using the $\ell_2$ norm.   
	\item {\bf Wass SW} - The event-wise adaptation from \citet{chen2019RSO} for two-stage distributionally robust optimization using the same ambiguity set as Wass SDP. 
	\item {\bf Approx PCM} - The lifted linear decision rule approach of \cite{bertsimas2018adaptive} for two-stage distributionally robust optimization, where the ambiguity set is defined by the first and second moments (estimated from the training data).
	\item {\bf SAA} - The sample average approximation.
\end{enumerate}
To compare the methods, we first generate a testing dataset of size $\tilde{N} = 10^4$. Then, for varying values of $N$, we generate $M=100$ training datasets of size $N$. For each training set $j \in [M]$ and each method $\mathcal{A}$, we compute the optimal first-stage decision $\bx^{\mathcal{A},j}_N$ and the corresponding objective value $\hat{v}^{\mathcal{A},j}_N$. 
The expected cost using the first-stage decision is estimated by
\begin{align*}
V(\bx^{\mathcal{A},j}_N) = \bc^\intercal \bx^{\mathcal{A},j}_N + \frac{1}{\tilde{N}} \sum_{i=1}^{\tilde{N}} Q(\bx^{\mathcal{A},j}_N,\tilde{\bxi}{}^i).
\end{align*}
We compare each method along the following metrics: 
\begin{enumerate}
	\item {\bf Tractability} - The running time for the method, averaged over the $M$ training sets.
	\item {\bf Feasibility} - 
	The proportion of realizations in the test set for which the first-stage decision is feasible: 
	$$\% \text{Infeasible Realizations}= \frac{1}{M}\sum_{i=1}^M \frac{\sum_{i=1}^{\tilde{N}}\mathbbm{1} \left( Q(\bx^{\mathcal{A},j}_N,\tilde{\bxi}{}^i) = \infty \right)}{\tilde{N}}.$$
	\item {\bf Optimality gap} - The relative gap between the expected cost using $\bx^{\mathcal{A},j}_N$ and the optimal expected cost:
	$$\Delta V_N^{\mathcal{A}}=\frac{1}{M}\sum_{j=1}^{M}\frac{V(\bx^{\mathcal{A},j}_N)-v^*}{v^*},$$
	where $v^*$ is estimated by solving SAA for an independent dataset of size $10^5$.
	\item {\bf Prediction error} - The relative difference between method's optimal cost and the expected cost of its optimal first stage decision: 
	$$\Delta \hat{v}_N^{\mathcal{A}}=\frac{1}{M}\sum_{j=1}^{M}\frac{\hat{v}^{\mathcal{A},j}_N-V(\bx^{\mathcal{A},j}_N)}{V(\bx^{\mathcal{A},j}_N)}.$$
\end{enumerate}

\subsection{Capacitated Network Inventory Management} \label{sec:mp:cap_lot}
\subsubsection{Problem Description.} We consider a two-stage capacitated network inventory problem. There are $n$ locations, and each location $i$ has an unknown demand ${\xi}_i$ which must be satisfied. The demand can be satisfied by existing stock $x_i$, which is bought in advance at the location, or by transporting an amount $y_{ji}$ of units from location $j$, which is determined after the demand is realized. The cost of buying stock in advance at location $i$ is $a_i$ per unit, and the cost of transporting each unit from location $i$ to $j$ is $c_{ij}$. Each location has a limited stock capacity of $K$ units, and no more than $b_{ij}$ units can be transported from $i$ to $j$. 
Our goal is to find optimal initial stock which minimizes the expected total cost
\begin{equation}\label{prob:mp:LotSizing}
\begin{aligned}
v^* = \quad &\underset{\bx\in\R^n}{\text{minimize}}&& \sum_{i=1}^{n} a_ix_i+\Exp[Q(\bx,\bxi)], 
\end{aligned}
\end{equation}
where the second-stage cost is 
\begin{equation*}
\begin{aligned}
Q(\bx,\bxi) = \quad &\underset{\by \in \R^{n \times n}}{\textnormal{minimize}}&&\sum_{j\neq i} c_{ij}y_{ij}\\
&\text{subject to} && x_i-\sum_{j\neq i} y_{ij}+\sum_{j\neq i} y_{ji} \geq \xi_i\\
&&&0 \le x_i \leq K \\
&&&0\le y_{ij} \leq b_{ij}.
\end{aligned}
\end{equation*}
We assume that the underlying probability distribution is unknown. Instead, our only information comes from historical data, $\bxi^1,\ldots,\bxi{}^N$, and knowledge that the support is contained in
\begin{equation}\label{eqn:mp:lotsizing_supersupport}
\Xi= \left \{\bxi\in\Real^n: 0 \le \xi_i\leq  K \right\}.
\end{equation}

\subsubsection{Experiments.} We generate a network of size $n=10$, where each location is drawn from a standard 2D Gaussian distribution. For each locations $i \neq j$, we let $c_{ij}$ be the Euclidean distance between the locations, $a_i=1$,  $K=20$, and $b_{ij}=K/(n-1)u_{ij}$ where $u_{ij}$ are i.i.d. random variables generated from a standard uniform distribution.

We consider the robustness parameters for Problem~\eqref{prob:mp:sro} with rates of $\epsilon_N =10N^{-1/10}$ and $\epsilon_N =20N^{-1/10}$. The choice of these robustness parameters is inspired by the probabilistic performance guarantees that can be obtained by the type-$\infty$ Wasserstein ambiguity set for well-conditioned bounded distributions, as described in Appendix~\ref{appx:finite_sample}. Roughly speaking, these rates for the robustness parameter provide assurance that, if the underlying distribution happens to satisfy certain conditions and the number of data points $N$ is sufficiently large, then any feasible first-stage decision of SRO, MP Affine, or SP Affine will be feasible for the stochastic problem. In practice, the robustness parameters can be determined, for example, by performing $k$-fold cross-validation over the historical data.  

In contrast, Approx PCM, Wass SDP and Wass SW are guaranteed to produce first-stage decisions which are feasible for the stochastic problem.  However, these approaches will restrict their search to first-stage decisions which have a feasible second-stage decision for all realizations $\bzeta \in \Xi$. Therefore, these methods will yield only one feasible first-stage decision, namely, $\bx=(K,\ldots,K)$. For this first-stage decision,  the optimal second-stage decision rule is given by $\by(\bzeta)=\bzero$ for all $\bzeta\in\Xi$. Thus, regardless of the values of the historical data,  the optimal costs of Approx PCM, Wass SDP and Wass SW  will be equal to $K\sum_{i=1}^n a_i$. 
  
Finally, we consider three probability distributions (uniform, normal, and lognormal) for the demand components (with mean $K/2$ and standard deviation $K/\sqrt{12}$) and used rejection sampling so that the (unknown) support of the each multivariate distribution is
\begin{equation}\label{eqn:mp:lotsizing_truesupport}
\tilde{\Xi}= \left \{\bxi\in\Real^n: 0 \le \xi_i\leq  K,\; i\in[n],\;\;\sum_{i\in[n]} \xi_i\leq \sqrt{n}K \right\}.
\end{equation}

\subsubsection{Results.} 
\paragraph{Tractability.}
The MP and SP methods' running times are approximately the same, ranging between $0.4$ and $1.5$ seconds per data point for all dataset sizes $N$. The SAA running times range between one and two milliseconds per data point. Since the remaining methods have trivial closed form solution of $\bx=(K,\ldots,K)$, their running times are omitted. 

\begin{figure}[t]
	\centering
	\includegraphics[width=0.8\textwidth]{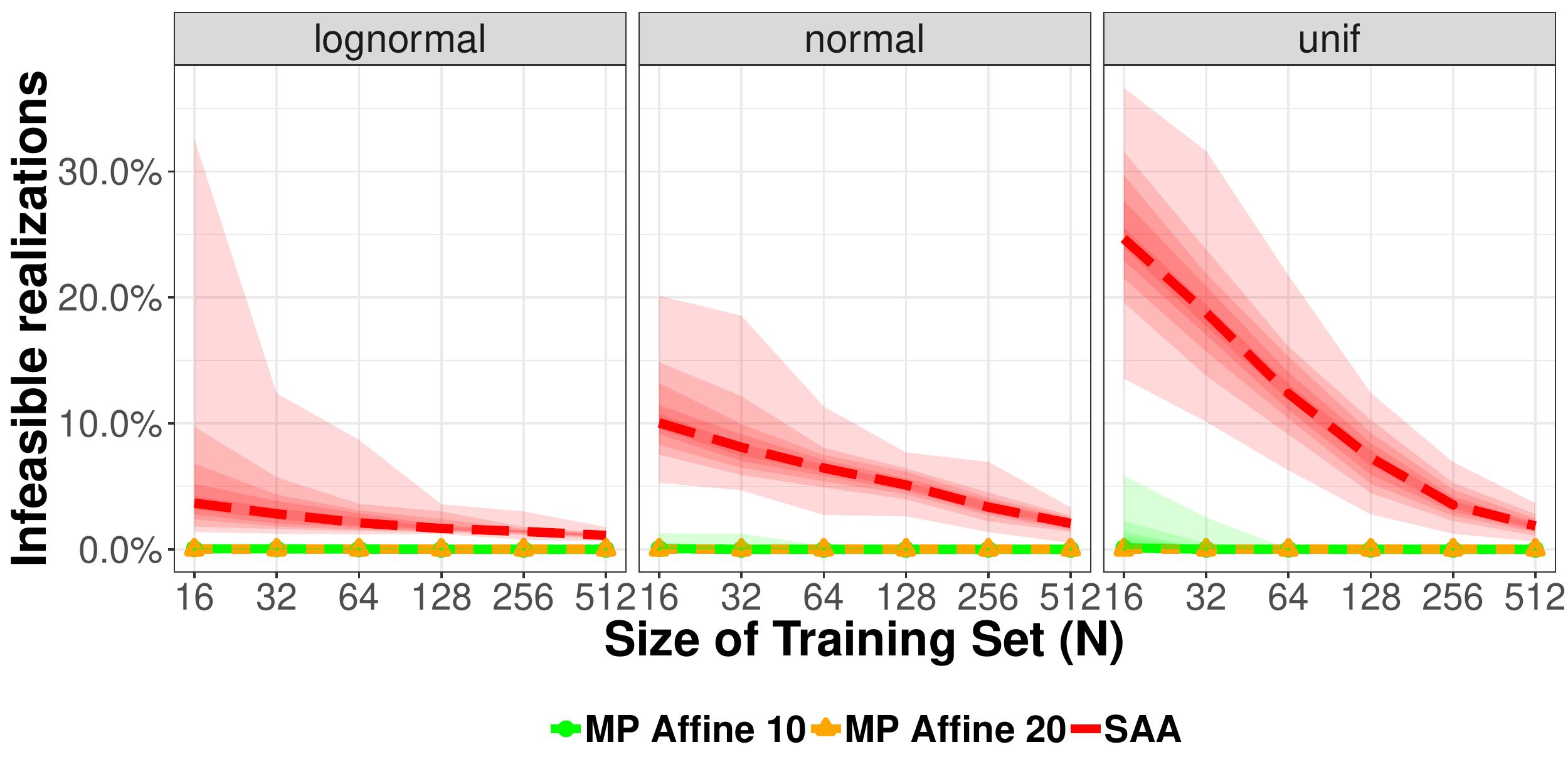}
	\caption{The percent of realization for which the method's first-stage decision was feasible in the inventory management  experiments. The lines represent the 50\% percentiles over the $M=100$ training datasets and the shading represents the distributions, getting lighter as the we get further away from the 50\% percentile and ending with the minimum and maximum value. SP Affine 10 and SP Affine 20 have similar feasibility guarantees as MP Affine 10 and MP Affine 20 and are omitted for clarity. }\label{fig:mp:Lotsizing_OutOfSampleFeasibility}
\end{figure}
\paragraph{ Feasibility.} 
 Figure~\ref{fig:mp:Lotsizing_OutOfSampleFeasibility} compares the out-of-sample feasibility of the different approaches. 
 These results demonstrate that the feasibility performance of the multi-policy approximation significantly outperform that of SAA across all distributions and choices of the robustness parameter $\epsilon_N$. Consistent with the theoretical performance guarantees from Appendix~\ref{appx:finite_sample}, the results show that multi-policy approximation of SRO can be effective in addressing a problem without relative complete recourse. Moreover, the results demonstrate that the multi-policy approximation generates solutions which are feasible with high probability even for moderate size data-sets, without restricting the solutions to those feasible for every realization in $\Xi$ as the other distributionally robust approaches. 
 
 \begin{figure}
	\centering
	\includegraphics[width=0.8\textwidth]{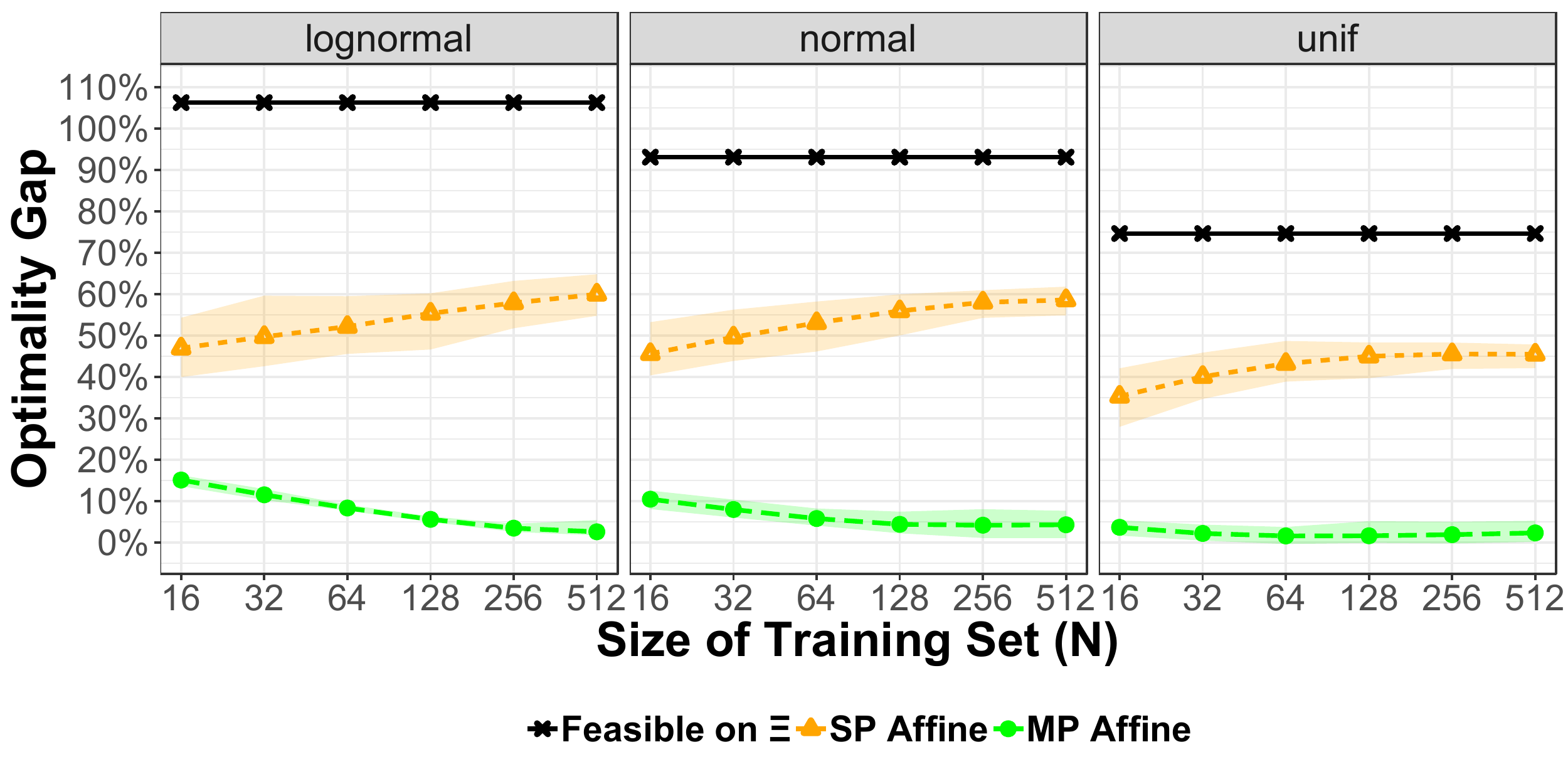}
	\caption{The optimality gap for the inventory management experiments. The lines represent the mean optimality gap over the 100 training datasets and the shadowing marks the minimum and maximum values. We used the radius $\epsilon_N =10N^{-1/10}$, and the MP Affine and SP Affine performance is averaged only over the training datasets which produced first-stage solutions which were feasible on the entire testing dataset. The SAA performance is not shown since it never produced a first-stage decisions which was feasible for all points in the testing dataset.}		\label{fig:mp:Lotsizing_OptimalityGap}
\end{figure}
\paragraph{Optimality.}Figure~\ref{fig:mp:Lotsizing_OptimalityGap} presents the optimality gap for the various methods. 

The results for the multi-policy approximation are consistent with the asymptotic optimality guarantee from Section~\ref{sec:mp:asymptotic}, showing that optimality gap of MP Affine decreases are more data is obtained and the gap is nearly zero when $N=512$ for all distributions. 
In contrast, the optimality gap of SP Affine does not improve as more data is obtained, but rather increases as the support is more accurately estimated, and feasibility is required on a larger set. Moreover, the distributionally robust methods, which require feasibility on the entirety of $\Xi$, always produces a first-stage decision with poor average performance.

\begin{figure}[t]
	\centering
	\includegraphics[width=0.8\textwidth]{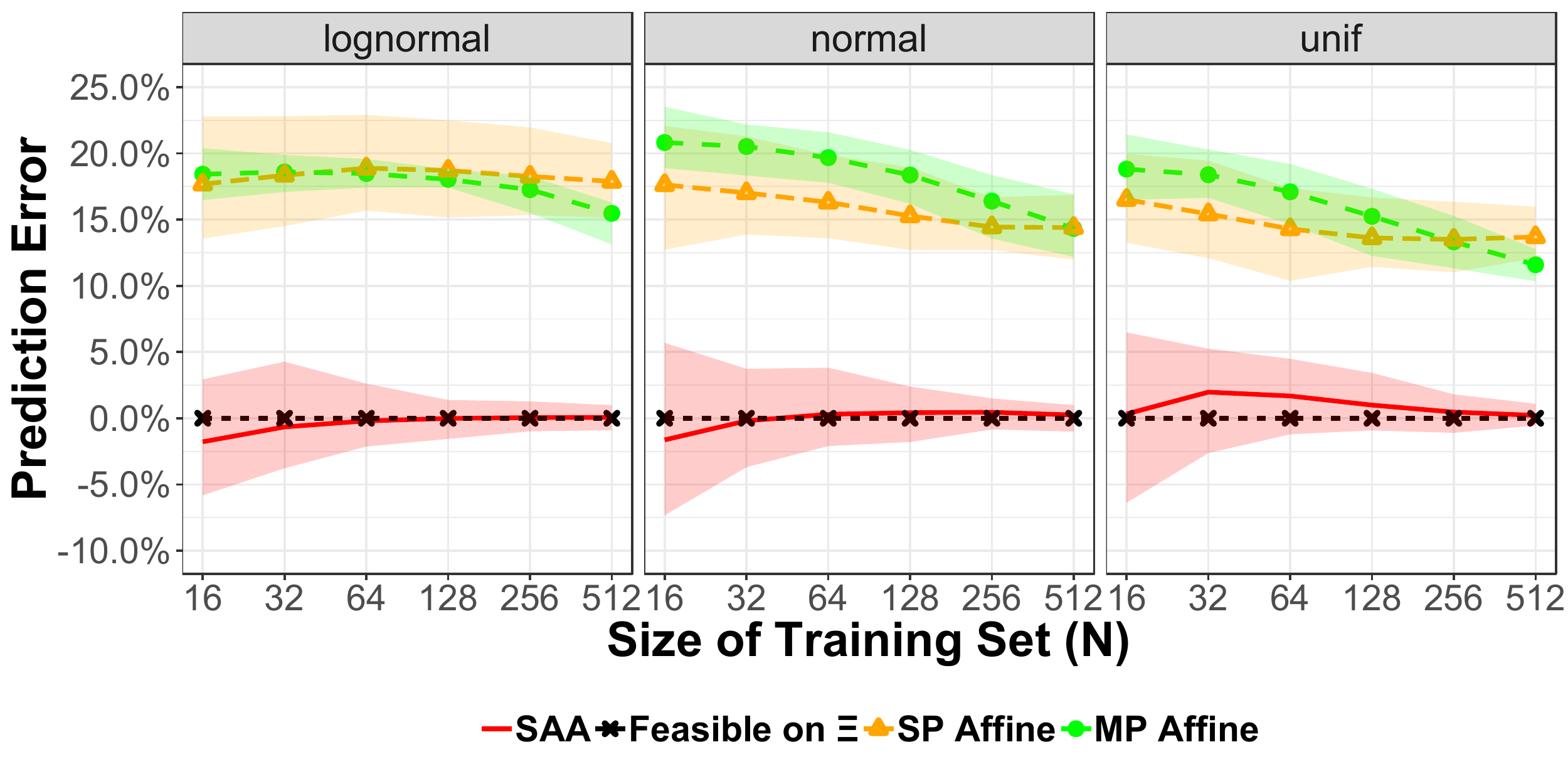}
	\caption{The prediction error for the inventory management experiments. The lines represent the mean prediction error over the 100 test sets and the shadowing mark the min and max values. For each method, the out of sample prediction error was computed only on the feasible realizations in the testing set for each first-stage decision.}
	\label{fig:mp:Lotsizing_OutOfSampleTestRatio} 
\end{figure}
\paragraph{Prediction.} Figure~\ref{fig:mp:Lotsizing_OutOfSampleTestRatio} presents the prediction gap of the various methods. 
We see that MP Affine and SP Affine offer meaningful upper bounds on the average cost of their prescribed first-stage decisions. In contrast, SAA both underestimates and overestimates the true performance, depending on the distribution type and the number of points in the training dataset. Since requiring feasibility on all of $\Xi$ results in a constant first-stage decision and second-stage decision equal to zero, the distributionally robust methods trivially achieve an exact prediction. 

\subsection{Medical Scheduling}
In this section, we perform an experiment which provides numerical evidence that the performance of our approach can match the state-of-the-art of approaches for distributionally robust optimization with the type-1 Wasserstein ambiguity set.

\subsubsection{Problem Description.}  We consider the following medical scheduling problem based on \cite{bertsimas2018adaptive} in which a clinic is tasked with scheduling a physician's daily appointments of $n$ patients. The patients are scheduled to arrive in ascending order, i.e. patient $1$ is scheduled before patient $2$, and so on. The physician works from time $0$ to $T$, and is paid overtime afterwards. The goal is to schedule the appointments such that the total waiting times and overtime costs are minimized.

The first stage decision $\bx \in \R^{n}_+$ is the schedule, where $x_i$ is the appointment length allocated to patient $i$. Thus, the appointment of patient $i+1$ is scheduled to begin at time $\sum_{j=1}^i x_j$. 
All appointments must be scheduled within the physician regular hours, represented by the constraint  
$\sum_{i=1}^{n} x_i\leq T$.
The actual length of the $i$-th patient's appointment is $\xi_i \ge 0$. The second-stage decision $\by \in \R^{n+1}_+$ corresponds to the waiting times; for each $i \in [n]$, $y_i$ is the waiting time for patient $i$, and $y_{n+1}$ is the overtime required by the physician. The first patient will be admitted upon arrival at time $0$, which is enforced by setting $y_1=0$. Given a realization $\bxi$ of appointment lengths, the waiting times are found via the following recursive formula:
\begin{align*}
y_{i+1} = \max\{y_{i} + \xi_{i} - x_{i},0\},
\end{align*}
The physician costs the clinic $c$ per unit of overtime, and each patient costs the clinic $1$ per unit of time spent waiting. Our goal is to find the schedule which minimizes the expected cost:
\begin{equation*}
\begin{aligned}
v^* = \quad &\underset{\bx\in\R^n}{\text{minimize}}&&\Exp[Q(\bx,\bxi)], 
\end{aligned}
\end{equation*}
where
\begin{equation*}
\begin{aligned}
Q(\bx,\bxi) = \quad &\underset{\by\in \R^{n+1}}{\textnormal{minimize}}&& \sum_{i=1}^{n} y_i + c y_{n+1} &&\\
& \text{subject to} && y_{i+1} \ge y_i +{\xi}_i - x_i,&& i=1,\ldots,n,\\
&&&y_i \ge 0,&& i=1,\ldots,n+1,\\
&&&\sum_{i=1}^{n} x_i\leq T&&\\
&&& x_i \ge 0,&& i=1,\ldots,n.
\end{aligned}
\end{equation*}

\subsubsection{Experiments.}
We consider an example with $n=8$ patients and cost parameter $c=2$. For each patient, we generate a mean $\mu_i$ uniformly over $[30,60]$ and generate a standard deviation $\sigma_i$ uniformly over $[0,0.3\mu_i]$. We consider three probability distributions (uniform, normal,  lognormal) with the corresponding mean and standard deviation, and we use rejection sampling to have nonnegative realizations. 
The physician's regular hours are set to $T=\sum_{i=1}^n\mu_i+0.5\norm{\bsigma}_2$.  Since no additional information on appointment lengths is known,  we set $\Xi=\R^{n}_+$ for all methods. For SP Affine, MP Affine, Wass SDP, and Wass SW, 
we use a robustness parameter $\epsilon_N =N^{-1/8}$. 

\subsubsection{Results.} 
\begin{table}[t]
	\centering
	\caption{Running times (in seconds) for the medical scheduling problem. The average (standard deviation) over the $M=100$ training datasets.}\label{tbl:MedicalScheduling_RunningTimes}
	\begin{tabular}{lrrrrrr}
		\hline
		& \multicolumn{6}{c}{Size of training datasets, $N$}\\
		Method & 16 & 32 & 64 & 128 & 256 & 512 \\ 
		\hline
	SAA & 0.00 (0.00) & 0.00 (0.00) & 0.01 (0.00) & 0.01 (0.00) & 0.02 (0.00) & 0.05 (0.00) \\ 
	Approx PCM & 0.02 (0.00) & 0.02 (0.00) & 0.02 (0.00) & 0.02 (0.00) & 0.02 (0.00) & 0.02 (0.00) \\ 
	SP Affine & 0.16 (0.01) & 0.30 (0.02) & 0.69 (0.06) & 1.60 (0.17) & 3.94 (0.41) & 10.15 (1.31) \\ 
	MP Affine & 0.16 (0.01) & 0.33 (0.02) & 0.85 (0.07) & 2.01 (0.18) & 6.49 (0.62) & 9.96 (0.97) \\ 
	Wass SDP & 7.90 (0.79) & 17.32 (1.54) & 34.67 (3.05) & 79.16 (7.34) & 166.96 (14.22) & 323.12 (32.88) \\ 
	Wass SW & 0.09 (0.00) & 0.19 (0.01) & 0.48 (0.03) & 1.34 (0.13) & 3.77 (0.29) & 4.56 (0.35) \\ 
		\hline
	\end{tabular}
	\vspace{-5pt}
\end{table}

\paragraph{{ Tractability.}} 
The running times for each method are shown in Table~\ref{tbl:MedicalScheduling_RunningTimes}. We observe that the multi-policy approximation remains computationally tractable for all values of $N$. The running times of Approx PCM method do not depend on the size $N$ of the training datasets, since this approach works with aggregated data, and the running times of the other methods generally scale linearly in $N$. 

\paragraph{{ Feasibility.}} As this problem has relative complete recourse, the second-stage problem is always feasible for any nonnegative first-stage decisions. 

\begin{figure}[t]
	\centering
	\includegraphics[width=0.85\textwidth]{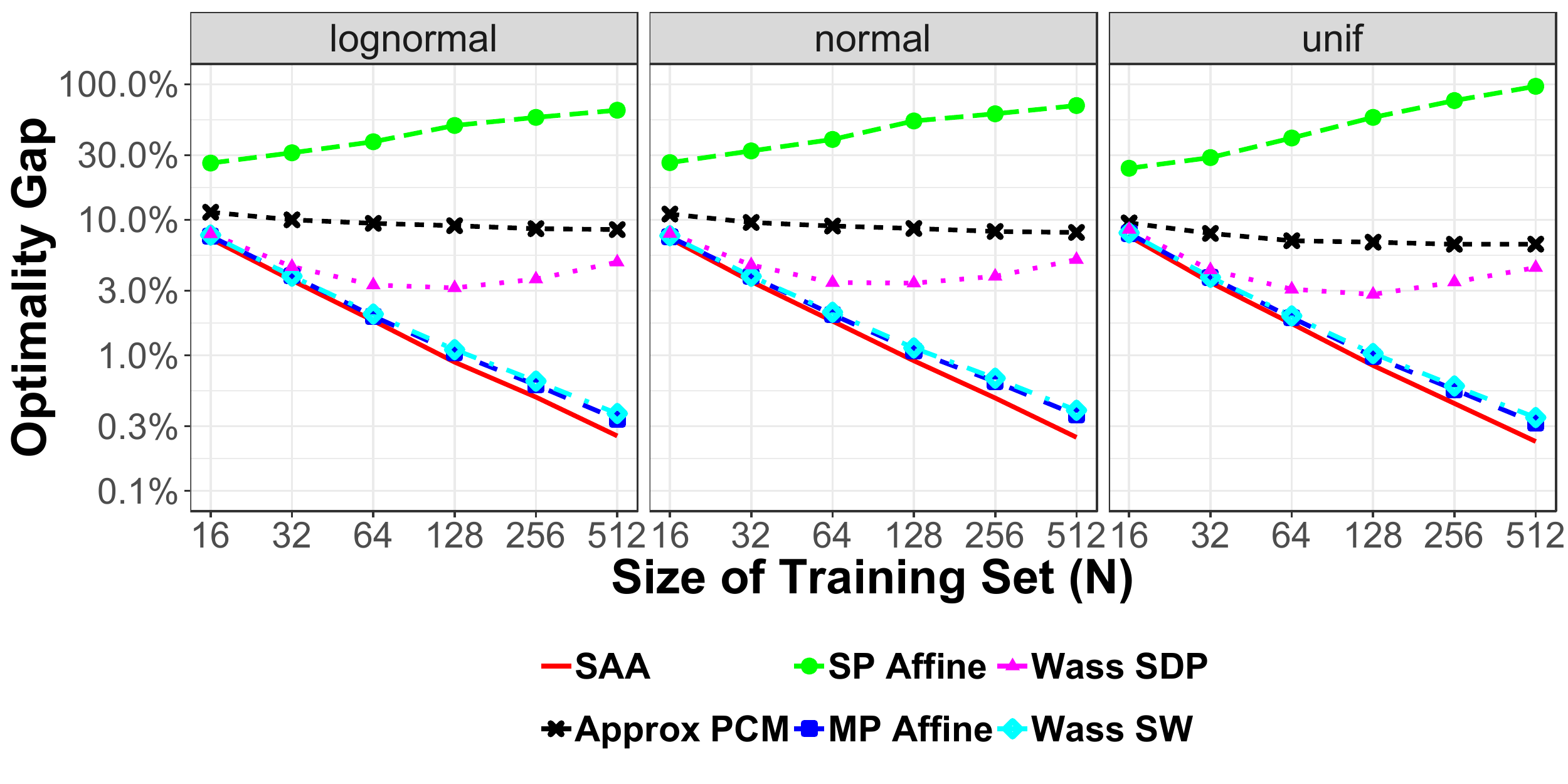}
	\caption{Out-of-sample performance for the medical scheduling problem. Average optimality gap over $M=100$ training datasets of different sizes. Each pane shows the optimality gap when both training and test datasets are generated from a specific distribution: lognormal, normal, or uniform.}\label{fig:mp:MedicalScheduling_OptimalityGap}
\end{figure}
\paragraph{{ Optimality.}}
Figure~\ref{fig:mp:MedicalScheduling_OptimalityGap} presents the optimality gap for the various methods. 
We observe that the optimality gaps of MP Affine, Wass SW, and SAA are nearly identical and converge to zero as $N$ grows larger.

\begin{figure}
	\centering
	\includegraphics[width=0.85\textwidth]{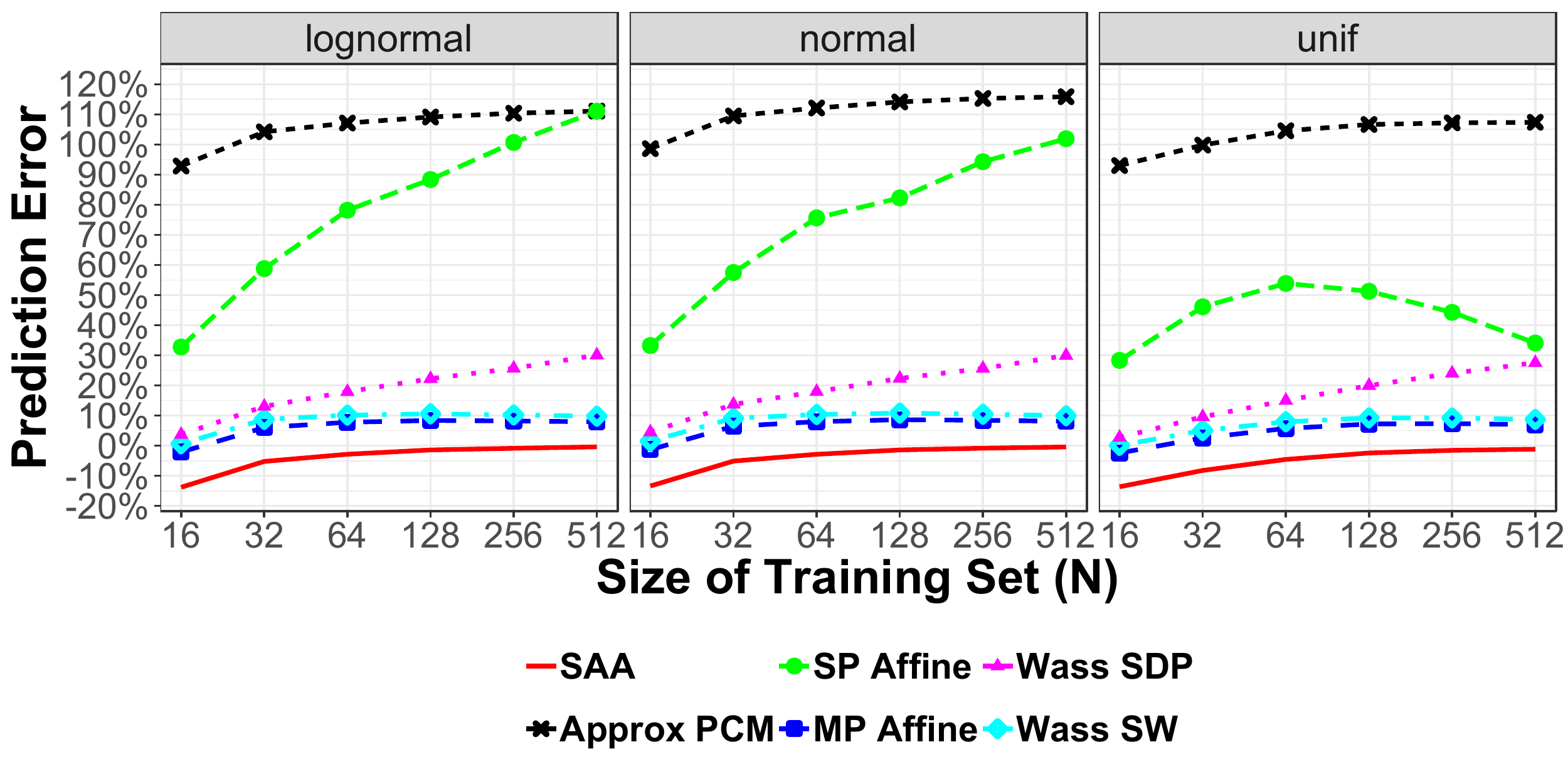}
	\caption{Predictive performance for the medical scheduling problem. Each line shows the mean prediction error over $M=100$ training datasets of different sizes. Each pane shows the optimality gap when both training and test sets were taking from a specific distribution, either lognormal, normal, or uniform.}\label{fig:mp:MedicalScheduling_PredictibilityGap}		
\end{figure}
\paragraph{ Prediction.} Figure~\ref{fig:mp:MedicalScheduling_PredictibilityGap} presents the average prediction error for each method. SAA produced a negative prediction error for between 50\% and 80\% of the training datasets, depending on the value of $N$.  
In contrast,  MP Affine and Wass SW methods produce nearly identical prediction errors, consistently between $5\%-10\%$,  thus producing a reliable upper bound on the true out-of-sample performance.


\section{Conclusion and Extensions} \label{sec:mp:conclusion}
In this paper, we investigated an approximation scheme, based on overlapping decision rules, for addressing data-driven two-stage distributionally robust optimization problems with the type-$\infty$ Wasserstein ambiguity set. The multi-policy approximation is simple, tractable, and provably asymptotically optimal under mild assumptions. 
A natural question is whether the multi-policy approach, or variant thereof, can be leveraged to obtain near-optimal approximations for other robust formulations. We remark that a fundamental aspect of our proof of asymptotic optimality for multi-policy approach in sample robust optimization is that each uncertainty set shrinks in size as more data points are obtained. 
Thus, the local region for which each linear decision rule is optimized becomes smaller over time. For this reason, an extension of our asymptotic optimality guarantees to other distributionally robust optimization settings may rely on this property as well. 


\bibliographystyle{plainnat}
\bibliography{DataDrivenRobustBibliography}

\begin{thebibliography}{34}
\providecommand{\natexlab}[1]{#1}
\providecommand{\url}[1]{\texttt{#1}}
\expandafter\ifx\csname urlstyle\endcsname\relax
  \providecommand{\doi}[1]{doi: #1}\else
  \providecommand{\doi}{doi: \begingroup \urlstyle{rm}\Url}\fi

\bibitem[Beale(1955)]{beale1955minimizing}
Evelyn~ML Beale.
\newblock On minimizing a convex function subject to linear inequalities.
\newblock \emph{Journal of the Royal Statistical Society. Series B
  (Methodological)}, pages 173--184, 1955.

\bibitem[Ben-Tal et~al.(2004)Ben-Tal, Goryashko, Guslitzer, and
  Nemirovski]{ben2004adjustable}
Aharon Ben-Tal, Alexander Goryashko, Elana Guslitzer, and Arkadi Nemirovski.
\newblock Adjustable robust solutions of uncertain linear programs.
\newblock \emph{Mathematical Programming}, 99\penalty0 (2):\penalty0 351--376,
  2004.

\bibitem[Bertsimas and Caramanis(2010)]{bertsimas2010finite}
Dimitris Bertsimas and Constantine Caramanis.
\newblock Finite adaptability in multistage linear optimization.
\newblock \emph{IEEE Transactions on Automatic Control}, 55\penalty0
  (12):\penalty0 2751--2766, 2010.

\bibitem[Bertsimas and Goyal(2012)]{Bertsimas2012poweraffine}
Dimitris Bertsimas and Vineet Goyal.
\newblock {On the power and limitations of affine policies in two-stage
  adaptive optimization}.
\newblock \emph{Mathematical Programming}, 134\penalty0 (2):\penalty0 491--531,
  2012.

\bibitem[Bertsimas et~al.(2018)Bertsimas, Shtern, and
  Sturt]{bertsimas2018multistage}
Dimitris Bertsimas, Shimrit Shtern, and Bradley Sturt.
\newblock A data-driven approach to multi-stage stochastic linear optimization.
\newblock \emph{Preprint}, 2018.
\newblock URL
  \url{http://www.optimization-online.org/DB_HTML/2018/11/6907.html}.

\bibitem[Bertsimas et~al.(2019)Bertsimas, Sim, and
  Zhang]{bertsimas2018adaptive}
Dimitris Bertsimas, Melvyn Sim, and Meilin Zhang.
\newblock Adaptive distributionally robust optimization.
\newblock \emph{Management Science}, 65\penalty0 (2):\penalty0 604--618, 2019.

\bibitem[Birge and Louveaux(2011)]{birge2011introduction}
John~R Birge and Francois Louveaux.
\newblock \emph{Introduction to Stochastic Programming}.
\newblock Springer Science \& Business Media, 2011.

\bibitem[Chen and Zhang(2009)]{Chen2009}
Xin Chen and Yuhan Zhang.
\newblock {Uncertain linear programs: extended affinely adjustable robust
  counterparts}.
\newblock \emph{Operations Research}, 57\penalty0 (6):\penalty0 1469--1482,
  2009.

\bibitem[Chen et~al.(2020)Chen, Sim, and Xiong]{chen2019RSO}
Zhi Chen, Melvyn Sim, and Peng Xiong.
\newblock Robust stochastic optimization made easy with rsome.
\newblock \emph{Management Science}, 2020.

\bibitem[Conforti et~al.(2014)Conforti, Cornuejols, and Zambelli]{Conforti2014}
Michele Conforti, Gerard Cornuejols, and Giacomo Zambelli.
\newblock \emph{Integer Programming}.
\newblock Springer, 2014.

\bibitem[Dantzig(1955)]{dantzig1955linear}
George~B Dantzig.
\newblock Linear programming under uncertainty.
\newblock \emph{Management Science}, 1\penalty0 (3-4):\penalty0 197--206, 1955.

\bibitem[Delage and Ye(2010)]{delage2010}
Erick Delage and Yinyu Ye.
\newblock Distributionally robust optimization under moment uncertainty with
  application to data-driven problems.
\newblock \emph{Operations research}, 58\penalty0 (3):\penalty0 595--612, 2010.

\bibitem[Erdo{\u{g}}an and Iyengar(2006)]{erdougan2006ambiguous}
E~Erdo{\u{g}}an and Garud Iyengar.
\newblock Ambiguous chance constrained problems and robust optimization.
\newblock \emph{Mathematical Programming}, 107\penalty0 (1-2):\penalty0 37--61,
  2006.

\bibitem[Erdo{\u{g}}an and Iyengar(2007)]{erdougan2007two}
E~Erdo{\u{g}}an and Garud Iyengar.
\newblock On two-stage convex chance constrained problems.
\newblock \emph{Mathematical Methods of Operations Research}, 65\penalty0
  (1):\penalty0 115--140, 2007.

\bibitem[Esfahani and Kuhn(2018)]{esfahani2018data}
Peyman~Mohajerin Esfahani and Daniel Kuhn.
\newblock Data-driven distributionally robust optimization using the
  wasserstein metric: Performance guarantees and tractable reformulations.
\newblock \emph{Mathematical Programming}, 171\penalty0 (1-2):\penalty0
  115--166, 2018.

\bibitem[Feige et~al.(2007)Feige, Jain, Mahdian, and Mirrokni]{feige2007robust}
Uriel Feige, Kamal Jain, Mohammad Mahdian, and Vahab Mirrokni.
\newblock Robust combinatorial optimization with exponential scenarios.
\newblock In \emph{International Conference on Integer Programming and
  Combinatorial Optimization}, pages 439--453. Springer, 2007.

\bibitem[Fournier and Guillin(2015)]{fournier2015rate}
Nicolas Fournier and Arnaud Guillin.
\newblock On the rate of convergence in wasserstein distance of the empirical
  measure.
\newblock \emph{Probability Theory and Related Fields}, 162\penalty0
  (3):\penalty0 707--738, 2015.

\bibitem[G{\"u}ler(2010)]{guler2010foundations}
Osman G{\"u}ler.
\newblock \emph{Foundations of Optimization}, volume 258.
\newblock Springer Science \& Business Media, 2010.

\bibitem[Hanasusanto and Kuhn(2018)]{hanasusanto2016conic}
Grani~A Hanasusanto and Daniel Kuhn.
\newblock Conic programming reformulations of two-stage distributionally robust
  linear programs over wasserstein balls.
\newblock \emph{Operations Research}, 66\penalty0 (3):\penalty0 849--869, 2018.

\bibitem[Hanasusanto et~al.(2015)Hanasusanto, Kuhn, and
  Wiesemann]{hanasusanto2015k}
Grani~A Hanasusanto, Daniel Kuhn, and Wolfram Wiesemann.
\newblock K-adaptability in two-stage robust binary programming.
\newblock \emph{Operations Research}, 63\penalty0 (4):\penalty0 877--891, 2015.

\bibitem[Hoffman(1952)]{hoffman2003approximate}
Alan~J Hoffman.
\newblock On approximate solutions of systems of linear inequalities.
\newblock \emph{Journal of Research of the National Bureau of Standards},
  49\penalty0 (4), 1952.

\bibitem[Jiang and Guan(2018)]{jiang2018risk}
Ruiwei Jiang and Yongpei Guan.
\newblock Risk-averse two-stage stochastic program with distributional
  ambiguity.
\newblock \emph{Operations Research}, 66\penalty0 (5):\penalty0 1390--1405,
  2018.

\bibitem[King and Wets(1991)]{king1991epi}
Alan~J. King and Roger~J.B. Wets.
\newblock Epi-consistency of convex stochastic programs.
\newblock \emph{Stochastics and Stochastic Reports}, 34\penalty0
  (1-2):\penalty0 83--92, 1991.

\bibitem[Liu et~al.(2019)Liu, Liu, and Lu]{liu2019rate}
Anning Liu, Jian-Guo Liu, and Yulong Lu.
\newblock On the rate of convergence of empirical measure in
  $\infty$-wasserstein distance for unbounded density function.
\newblock \emph{Quarterly of Applied Mathematics}, 77\penalty0 (4):\penalty0
  811--829, 2019.

\bibitem[Nemirovski and Shapiro(2006)]{nemirovski2006scenario}
Arkadi Nemirovski and Alexander Shapiro.
\newblock Scenario approximations of chance constraints.
\newblock In Giuseppe Calafiore and Fabrizio Dabbene, editors,
  \emph{Probabilistic and Randomized Methods for Design Under Uncertainty},
  pages 3--47. Springer, 2006.

\bibitem[Robinson(1996)]{robinson1996analysis}
Stephen~M Robinson.
\newblock Analysis of sample-path optimization.
\newblock \emph{Mathematics of Operations Research}, 21\penalty0 (3):\penalty0
  513--528, 1996.

\bibitem[Rockafellar(1970)]{rockafellar1970convex}
R~T Rockafellar.
\newblock \emph{Convex Analysis}.
\newblock Princeton University Press, 1970.

\bibitem[Shapiro(2003)]{shapiro2003monte}
Alexander Shapiro.
\newblock Monte carlo sampling methods.
\newblock In \emph{Handbooks in Operations Research and Management Science},
  volume~10, chapter~6, pages 353--425. Elsevier, 2003.

\bibitem[Shapiro et~al.(2009)Shapiro, Dentcheva, and
  Ruszczy{\'n}ski]{shapiro2009lectures}
Alexander Shapiro, Darinka Dentcheva, and Andrzej Ruszczy{\'n}ski.
\newblock \emph{Lectures on Stochastic Programming: Modeling and Theory}.
\newblock SIAM, 2009.

\bibitem[Trillos and Slep{\v{c}}ev(2015)]{trillos2014rate}
Nicol{\'a}s~Garcia Trillos and Dejan Slep{\v{c}}ev.
\newblock On the rate of convergence of empirical measures in
  infinity-transportation distance.
\newblock \emph{Canadian Journal of Mathematics}, 67\penalty0 (6):\penalty0
  1358--1383, 2015.

\bibitem[Van~Parys et~al.(2017)Van~Parys, Esfahani, and Kuhn]{van2017data}
Bart~PG Van~Parys, Peyman~Mohajerin Esfahani, and Daniel Kuhn.
\newblock From data to decisions: Distributionally robust optimization is
  optimal.
\newblock \emph{arXiv preprint arXiv:1704.04118}, 2017.

\bibitem[Wiesemann et~al.(2014)Wiesemann, Kuhn, and
  Sim]{wiesemann2014distributionally}
Wolfram Wiesemann, Daniel Kuhn, and Melvyn Sim.
\newblock Distributionally robust convex optimization.
\newblock \emph{Operations Research}, 62\penalty0 (6):\penalty0 1358--1376,
  2014.

\bibitem[Xie(2020)]{xie2020tractable}
Weijun Xie.
\newblock Tractable reformulations of two-stage distributionally robust linear
  programs over the type-$\infty$ wasserstein ball.
\newblock \emph{Operations Research Letters}, 48\penalty0 (4):\penalty0
  513--523, 2020.

\bibitem[Xu et~al.(2012)Xu, Caramanis, and Mannor]{xu2012distributional}
Huan Xu, Constantine Caramanis, and Shie Mannor.
\newblock A distributional interpretation of robust optimization.
\newblock \emph{Mathematics of Operations Research}, 37\penalty0 (1):\penalty0
  95--110, 2012.

\end{thebibliography}

\clearpage


\begin{APPENDICES}
\setlength{\parskip}{1em}


\section{Preliminary Results}\label{appx:polyhedral}
In this section, we present preliminary definitions and results for polyhedral sets, which are utilized in  Appendices~\ref{appx:C_feas} and \ref{appx:shrinkage}. This section is organized as follows: 

\begin{itemize}
\item In Appendix~\ref{appx:poly:general}, we review standard definitions and classic results for polyhedral theory.  
\item In Appendix~\ref{appx:poly:radius_point}, we define the ``radius of a point"  (Definition~\ref{defn:radius_point}) and provide two results about it (Lemmas~\ref{lem:mp:radius_z} and \ref{lemma:coneshrinkege}). 
\item In Appendix~\ref{appx:poly:radius_face}, we define the ``radius of a polyhedron"  (Definition~\ref{defn:radius_face}) and show that it is always strictly positive (Lemma~\ref{lem:mp:radius}). 
\item In Appendix~\ref{appx:poly:angle}, we define the "sine of the angle" (Definition~\ref{defn:maximal_tangent}) between two polyhedral sets.
\end{itemize}

\subsection{Standard definitions and notation}\label{appx:poly:general}

We begin by establishing the notation and basic concepts regarding polyhedral theory which are used in the subsequent appendices. The material in the present section is fairly standard, and we refer the reader to \citet[Section 3]{Conforti2014} for additional details on polyhedra.

We start by recalling that $\Xi \subseteq \R^d$ is a nonempty polyhedron given by
$$\Xi=\left\{\bzeta\in\R^d:\bg_i^\intercal \bzeta\geq g_i^0,\;i\in[\tilde{m}]\right\}.$$
Throughout the appendices we assume, without any loss of generality, that this representation of the polyhedron is \emph{minimal}, meaning that the omission of any of the inequalities would result in a different polyhedron, 
and that $\norm{\bg_i}_*=1$ for each constraint. 

For notational convenience, we define the hyperplane associated with the $i$th constraint of $\Xi$ as 
\begin{equation*}\mathcal{H}_i\triangleq \left\{\bzeta \in \R^d:\bg_i^\intercal \bzeta= g_i^0\right\}.\end{equation*}
Since we have assumed that $\| \bg_i\|_* = 1$, we note that the following holds for all realizations $\bzeta \in \Xi$:
\begin{align}\label{eq:dist_point_hyper}
\dist(\bzeta,\mathcal{H}_i)\triangleq \min_{\bzeta'\in \mathcal{H}_i}\norm{\bzeta'-\bzeta} = \bg_i^\intercal\bzeta-g^0_i.\end{align}  
For any realization $\bzeta \in \Xi$, let its \emph{active index set} with respect to the polyhedron $\Xi$ be defined as
\begin{equation*}I_\Xi(\bzeta)\triangleq \left\{i\in[\tilde{m}]:\bzeta\in \mathcal{H}_i\right\}.\end{equation*}

Our proofs in the appendices, as well as the statement of Lemma~\ref{lem:mp:feas_then_C_feas} in Section~\ref{sec:mp:asymptotic:L_Feasibility}, will involve iterating over the faces of a polyhedron. For an introduction and basic results of faces and facets, see  \citet[Section 3.8]{Conforti2014}. For any polyhedron $T\subseteq \R^d$, we denote its set of faces by $\Face(T)$ and its set of facets by $\Facet(T)$. We denote the dimension of a polyhedron $T \subseteq \R^d$ by $\dim(T) \in \N$ and the dimension of the minimal faces of $T$ as $\dimlower(T) \in \N$. 

Finally, we denote the distance between two polyhedra $T, T'\subseteq \R^d$ by
\begin{align}\label{eq:dist_poly}
\dist(T,T') &\triangleq \inf_{\bz \in T,\bz' \in T'} \norm{\bz-\bz'}= \inf_{\bz \in T}\dist(\bz,T').
\end{align}
Since $T,T' \subseteq \R^d$ are closed and convex sets, we recall that $\dist(T,T') > 0$ if and only if $T \cap T' = \emptyset$ \citep[Theorem 11.4 and Corollary 11.4.1]{rockafellar1970convex}.

\subsection{Radius of a point}\label{appx:poly:radius_point}
We next introduce terminology and basic results regarding the ``radius of a point" with respect to a polyhedral set. 
In essence, the radius captures the distance between a point (which is contained in a polyhedral set) and the hyperplanes which define that polyhedral set, and is rigorously defined as follows.  
\begin{definition} \label{defn:radius_point}
	For any point $\bzeta \in \Xi$,  its {radius} with respect to $\Xi$ is $\Delta_\Xi(\bzeta) \triangleq 
	\min\limits_{i\in[\tilde{m}]\setminus I_\Xi(\bzeta)} \dist(\bzeta,\mathcal{H}_i)$.
\end{definition}
The following result shows that the radius of a point with respect to $\Xi$ is always strictly positive. 
\begin{lemma}\label{lem:mp:radius_z}
	$\Delta_\Xi(\bzeta) > 0$ for all $\bzeta\in \Xi$.
\end{lemma}
\begin{proof}{Proof.}
	Consider any realization $\bzeta \in \Xi$. If $I_\Xi(\bzeta)=[\tilde{m}]$, then $\Delta_\Xi(\bzeta)=\infty>0$. Otherwise, there exists a constraint $i\in[\tilde{m}]\setminus I_\Xi(\bzeta)$ such that $\Delta_\Xi(\bzeta) = \dist(\bzeta, \mathcal{H}_i)$. 
	Since $\bzeta \notin \mathcal{H}_i$, it follows that $\dist(\bzeta,\mathcal{H}_i)>0$. In all cases, we have shown that $\Delta_\Xi(\bzeta)>0$. \Halmos
\end{proof}
The final result in this section presents a key property of the radius of a point, which will be used in the proofs presented in both Appendices~\ref{appx:C_feas} and \ref{appx:shrinkage}. Specifically, the following Lemma~\ref{lemma:coneshrinkege}, illustrated in Figure~\ref{fig:mp:same_shape}, shows that the shape of a ball around $\hat{\bzeta} \in \Xi$ which is intersected with the polyhedron $\Xi$ is the same for all balls around $\hat{\bzeta}$ with radius less than or equal to $\Delta_\Xi(\hat{\bzeta})$.

\begin{figure}[t]
	\centering
	\includegraphics{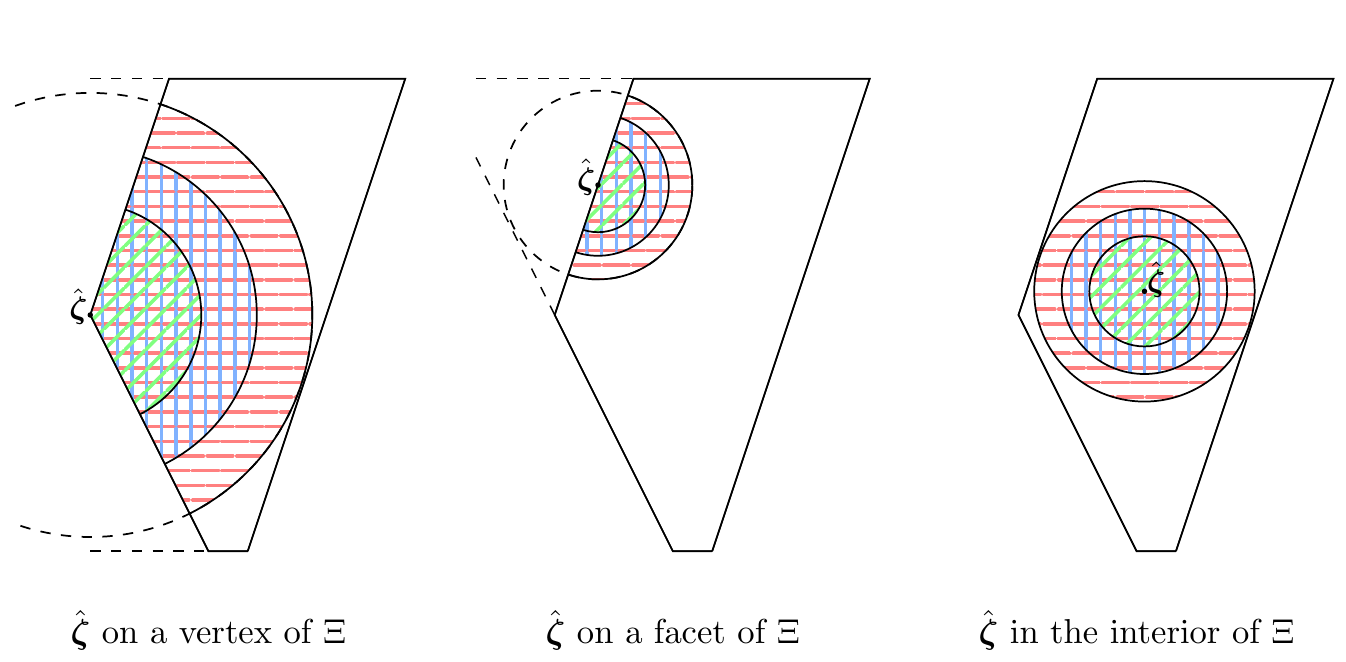}
	\caption{
		For each of the three figures, the regions filled with horizontal (red), vertical (blue), and diagonal (green) lines represent $B(\hat{\bzeta},\lambda\Delta_\Xi(\hat{\bzeta})) \cap \Xi$ for $\lambda=1$, $0.75$, and $0.5$, respectively. Since the shape of the sets do not change for each value of $\lambda$, we  observe that
		$ \left\{\lambda \bzeta + (1-\lambda) \hat{\bzeta}: \bzeta\in  B(\hat{\bzeta}{}^i,\Delta_\Xi(\hat{\bzeta}))\cap\Xi\right\}=B(\hat{\bzeta},\lambda\Delta_\Xi(\hat{\bzeta})) \cap \Xi$. 
	}	
	\label{fig:mp:same_shape}
\end{figure}
%

\begin{lemma}\label{lemma:coneshrinkege}
	Let $\hat{\bzeta}\in \Xi$. Then  for all $\epsilon\in (0,\Delta_\Xi(\hat{\bzeta})]$ and $\lambda\in (0,1]$,
	$$B({\hat{\bzeta}},\lambda \epsilon)\cap \Xi= \left\{\lambda\bzeta+(1-\lambda)\hat{\bzeta}:\bzeta\in B({\hat{\bzeta}},{\epsilon})\cap \Xi \right\}.$$
	
\end{lemma}
\begin{proof}{Proof.}
	Let $\hat{\bzeta} \in \Xi$, and consider any $\epsilon \in (0,\Delta_\Xi(\hat{\bzeta})]$ and $\lambda \in (0,1]$ (note that the existence of such an $\epsilon$ follows from Lemma~\ref{lem:mp:radius_z}). 
	First, choose any arbitrary
	\begin{align*} 
	\tilde{\bzeta}\in \left\{\lambda\bzeta+(1-\lambda)\hat{\bzeta}:\bzeta\in B({\hat{\bzeta}},{\epsilon})\cap \Xi \right\}.
	\end{align*}
	Letting $\bzeta\triangleq \frac{1}{\lambda} \tilde{\bzeta}-\frac{1-\lambda}{\lambda} \hat{\bzeta}$ and observing that $\bzeta \in B({\hat{\bzeta}},{\epsilon})$, it follows that 
	$\tilde{\bzeta}\in B({\hat{\bzeta}},\lambda{\epsilon})$. Moreover, since $\Xi$ is a convex set and $\bzeta,\hat{\bzeta}\in \Xi$, we also observe that $\tilde{\bzeta}\in\Xi$. Since $\tilde{\bzeta}$ was chosen arbitrarily, we have shown that $$B({\hat{\bzeta}},\lambda \epsilon)\cap \Xi \supseteq \left\{\lambda\bzeta+(1-\lambda)\hat{\bzeta}:\bzeta\in B({\hat{\bzeta}},{\epsilon})\cap \Xi \right\}.$$
	
	It remains to prove the other direction. Indeed, choose any arbitrary
	\begin{align} 
	\tilde{\bzeta}\in B(\hat{\bzeta}, \lambda\epsilon)\cap \Xi \label{line:mp:defn_tilde_bzeta}
	\end{align}
	and define
	$$\bzeta \triangleq \frac{1}{\lambda}\tilde{\bzeta}+\left(1-\frac{1}{\lambda} \right)\hat{\bzeta}.$$
	
	First, we observe that
	$$\norm{\bzeta-\hat{\bzeta}}= \norm{\frac{1}{\lambda}\tilde{\bzeta}+\left(1-\frac{1}{\lambda} \right)\hat{\bzeta}-\hat{\bzeta}}=\frac{1}{\lambda}\norm{\tilde{\bzeta}-\hat{\bzeta}}\leq {\epsilon},$$
	where the inequality follows from \eqref{line:mp:defn_tilde_bzeta}. 
	Therefore, we have shown that
	\begin{align}
	\bzeta \in B(\hat{\bzeta},{\epsilon}). \label{line:mp:xi_in_ball}
	\end{align}
	
	Second, consider any $i\in[\tilde{m}]\setminus I_\Xi(\hat{\bzeta})$. Then,
	\begin{align}
	\bg_i^\intercal\bzeta-g^0_i &\ge \min_{\bzeta \in B(\hat{\bzeta},{\epsilon})} \left \{\bg_i^\intercal\bzeta-g^0_i \right \} \notag \\
	&= -{\epsilon}\norm{\bg_i}_*+g^0_i - \bg_i^\intercal \hat{\bzeta} \notag \\
	&= -{\epsilon} + \dist(\hat{\bzeta},\mathcal{H}_i).\label{line:mp:hyperlane_part_one}
	\end{align}
	Indeed, the inequality follows from \eqref{line:mp:xi_in_ball}, the first equality follows from the definition of the dual norm, and the final equality follows from a combination of $\|\bg_i\|_* = 1$, $\hat{\bzeta} \in \Xi$, and \eqref{eq:dist_point_hyper}. 
	Furthermore, 
	\begin{equation}
	\dist(\hat{\bzeta}, \mathcal{H}_i)\geq \Delta_\Xi(\hat{\bzeta})\geq\epsilon, \label{line:mp:hyperlane_part_two}
	\end{equation} 
	where the first inequality follows from Definition~\ref{defn:radius_point}, and the second inequality follows because $\epsilon \in (0, \Delta_\Xi(\hat{\bzeta})]$. 
	Combining \eqref{line:mp:hyperlane_part_one} and \eqref{line:mp:hyperlane_part_two}, we have thus shown that
	\begin{align}
	\bg_i^\intercal \bzeta \ge g^0_i, \quad \forall i\in[\tilde{m}]\setminus I_\Xi(\hat{\bzeta}).\label{line:mp:xi_feas_part_one}
	\end{align}
	
	Third, consider any constraint $i  \in I_\Xi(\hat{\bzeta})$. Then,
	\begin{align}
	\bg_i^\intercal\bzeta &=\left(1-\frac{1}{\lambda}\right)\bg^\intercal_i\hat{\bzeta}+\frac{1}{\lambda}\bg_i^\intercal\tilde{\bzeta} \notag \\
	&= \left(1-\frac{1}{\lambda}\right)g_i^0+\frac{1}{\lambda}\bg_i^\intercal\tilde{\bzeta} \notag \\
	&\ge  \left(1-\frac{1}{\lambda}\right)g_i^0+\frac{1}{\lambda}g_i^0 \notag \\
	&= g_i^0. \label{line:mp:xi_feas_part_two}
	\end{align}
	Indeed, the first equality follows from the definition of $\bzeta$, the second equality follows because $i \in I_\Xi(\hat{\bzeta})$, and the inequality follows because $\tilde{\bzeta} \in \Xi$ and $\lambda > 0$.
	
	Combining \eqref{line:mp:xi_in_ball}, \eqref{line:mp:xi_feas_part_one}, and \eqref{line:mp:xi_feas_part_two}, we have shown that $\bzeta \in B(\hat{\bzeta},\epsilon) \cap \Xi$.
	In other words, we have shown that there exists a realization $\bzeta \in B(\hat{\bzeta},\epsilon) \cap \Xi$ such that
	\begin{align*}
	\tilde{\bzeta} = \lambda \bzeta + (1-\lambda) \hat{\bzeta},
	\end{align*}
	and thus 
	\begin{align*}
	\tilde{\bzeta}  \in  
	\left\{\lambda \bzeta + (1-\lambda) \hat{\bzeta}: \bzeta\in  B(\hat{\bzeta},\epsilon)\cap\Xi\right\}. 
	\end{align*}
	Since $\tilde{\bzeta} \in B(\hat{\bzeta}, \lambda\epsilon)\cap \Xi$ was chosen arbitrarily, our proof is complete. 
	\Halmos
\end{proof}

\vspace{1em}

\subsection{Radius of a polyhedron} \label{appx:poly:radius_face}
In the previous section, we introduced the ``radius of a point" (Definition~\ref{defn:radius_point}), which measured the distance from a point to the defining hyperplanes of a polyhedron. In this section, we develop an extension of this definition, which we refer to as the ``radius of a polyhedron". 
In essence, for a given polyhedron $T \subseteq \R^d$ which is a subset of the polyhedron $\Xi \subseteq \R^d$, the radius $\Delta_\Xi(T)$ of the polyhedron $T$ captures the minimal distance that any  $F \in \Face(T)$ can be moved before encountering one of the defining hyperplanes of $\Xi$ which $F$ does not initially intersect. A visualization of this definition is found in Figure~\ref{fig:mp:radius_polyhedron}, and the precise definition is given below.

\begin{figure}[t]
	\centering
	\includegraphics{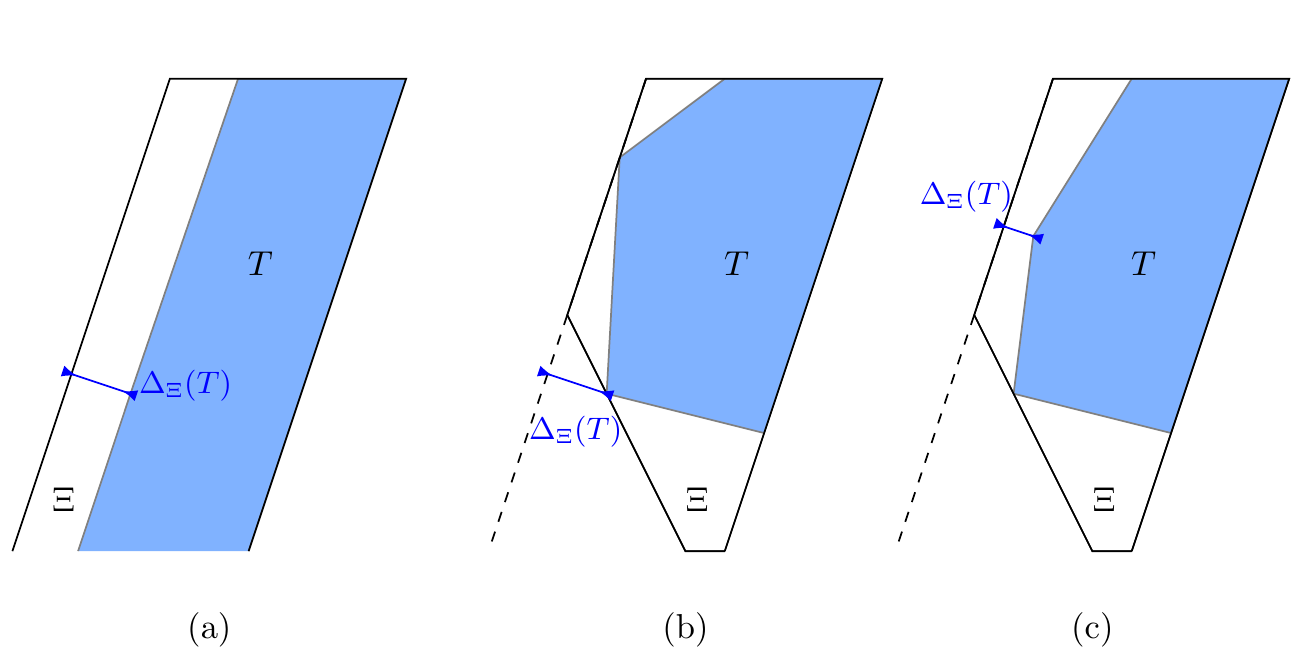}
	\caption{Each of the three figures shows $\Delta_\Xi(T)$ for polyhedra $T \subseteq \Xi \subseteq \R^2$.
	}	
	\label{fig:mp:radius_polyhedron}
\end{figure}
\begin{definition} \label{defn:radius_face}
	For any polyhedron $T\subseteq\Xi$, its radius with respect to $\Xi$ is
	$\Delta_\Xi(T) \triangleq \min\left\{\rho_1(T),\rho_2(T)\right\}$, 
	where
	$$\rho_1(T) \triangleq \min\limits_{i\in[\tilde{m}]:\mathcal{H}_i\cap T=\emptyset} \dist(T,\mathcal{H}_i), \text{ and } \rho_2(T) \triangleq \min\limits_{F\in \Facet(T)} \Delta_\Xi(F).$$
\end{definition}
\vspace{1em}
Similar to the radius of a point, we now show that the radius of a polyhedron is strictly positive.
\begin{lemma}\label{lem:mp:radius}
	If $T\subseteq \Xi$ is a nonempty polyhedron, then $\Delta_\Xi(T) > 0$.
\end{lemma}
\begin{proof}{Proof.}
	Our proof consists of two parts. First, we show that 
	\begin{align}
	\rho_1(T)>0. \label{line:mp:rho_one_always_positive}
	\end{align}
	Indeed, if $T\cap \mathcal{H}_i \neq \emptyset$ for each constraint $i \in [\tilde{m}]$, then $\rho_1(T) = \infty$. Otherwise, for all constraints $i\in [\tilde{m}]$ which satisfy $T\cap \mathcal{H}_i=\emptyset$, we observe that $\dist(T,\mathcal{H}_i)>0$, which implies that $\rho_1(T)$ is a minimization over a finite set of positive numbers. Thus, in all cases, we have shown that \eqref{line:mp:rho_one_always_positive} holds. 
	
	Second, we show that $\rho_2(T)>0$. By the recursion, this is equivalent to showing that $\Delta_\Xi(F)>0$ for any nonempty $F \in \Face(T)$. However, since any nonempty $F \in \Face(T)$ is a polyhedron contained in $\Xi$, the first part of the proof implies that $\rho_1(F)>0$, thus, it is left to show that
	\begin{align}
	\rho_2(T)>0, \quad \forall F \in \Face(T). \label{line:mp:rho_two_always_positive}
	\end{align}
	Our proof of \eqref{line:mp:rho_two_always_positive} follows from induction: 
	\begin{itemize}
		\item {\bf Base case:} Consider any $F \in \Face(T)$ which is a minimal face of $T$ (or equivalently, satisfies $\dim(F) = \dimlower(T)$). It follows from the definition of a minimal face that $\Facet(F)=\emptyset$.
		Therefore, it follows that 
		$\rho_2(F)=\infty $
		for all minimal faces of $T$.
		
		\item {\bf Induction step:} Fix any dimension $k \in \{\dimlower(T),\ldots,\dim(T)-1\}$, and assume that $\rho_2(F')>0$ for all $F'\in\Face(T)$ with $\dim(F') \le k$. 
		
		Consider any $F \in \Face(T)$ with $\dim(F) = k+1$. 
		Since $F$ is not a minimal face, it has a finite and nonzero number of facets. Thus, there exists a $F' \in \Facet(F)$ such that $$\rho_2(F)= \Delta_\Xi(F') =\min \{ \rho_1(F'), \rho_2(F') \}.$$ We have  shown in line~\eqref{line:mp:rho_one_always_positive} that $\rho_1(F') > 0$. Moreover, since $\dim(F')=k$, it follows from the induction hypothesis that $\rho_2(F')>0$. Therefore, we have shown that $\rho_2(F)>0$. 
	\end{itemize}
	We have thus shown that $\rho_1(T) > 0$ and $\rho_2(T) > 0$, which concludes the proof.
	\Halmos
\end{proof}
\subsection{Angle between two polyhedra} \label{appx:poly:angle}
We conclude Appendix~\ref{appx:polyhedral} by developing terminology to capture the ``sine of the angle'' between two polyhedral sets. This quantity, visualized in Figure~\ref{fig:mp:angle_polyhedron}, is defined as follows:
\begin{definition} \label{defn:maximal_tangent}
	The sine of the angle between a nonempty polyhedron $T \subseteq \Xi$ and $\Xi \subseteq \R^d$ is given by $1/\theta^T_\Xi$, where
	\begin{align*} 
	\theta^T_\Xi &\triangleq \max \left \{1, \max_{\substack{F \in \Face(T), \; i \in [\tilde{m}]: \\ F \cap \mathcal{H}_i\neq \emptyset}} \theta_i(F) \right \}, 
	\end{align*}
	and 
	\begin{align*}
	\theta_i(F) \triangleq \inf \left \{ \theta \ge 0: \dist(\bzeta,F\cap \mathcal{H}_i)  \le \theta \dist(\bzeta,\mathcal{H}_i),\; \forall \bzeta \in F \right \}.
	\end{align*}
\end{definition}
\begin{figure}[t]
	\centering
	\includegraphics{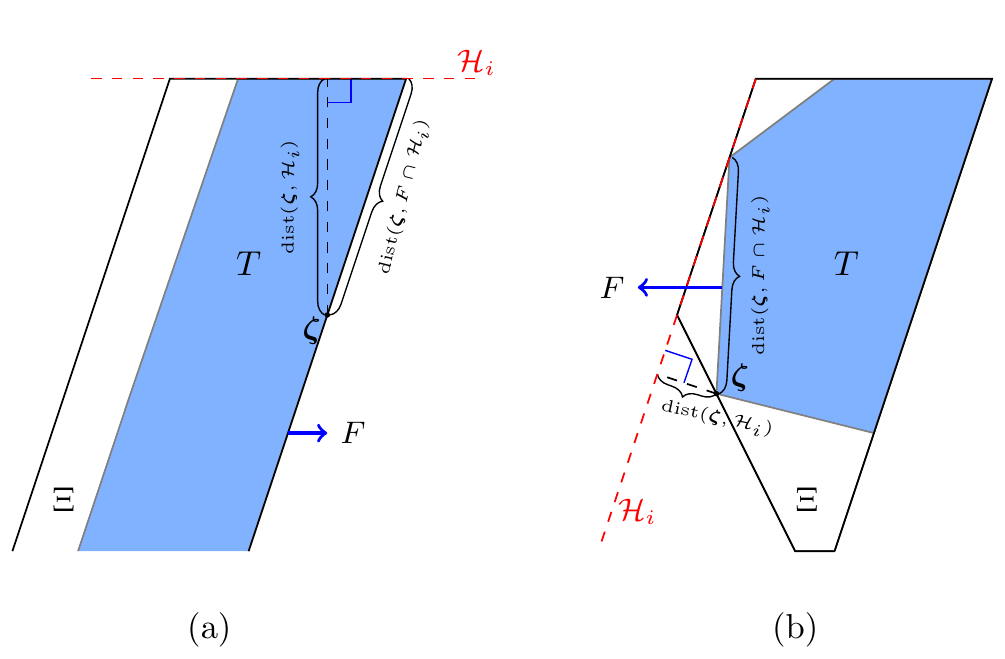}
	\caption{Each of the two figures shows $\theta_\Xi^T$, \ie the sign of the angle between $T$ and $\Xi$, for a polyhedron $T \subseteq \Xi \subseteq \R^2$. Specifically, it shows the point $\bzeta$, a face $F$ of $T$, and a Hyperplane $\mathcal{H}_i$ defining a constraint of $\Xi$ for which $\theta_\Xi^T=\theta_i(F)=\frac{\dist(\bzeta,F\cap\mathcal{H}_i)}{\bzeta,\mathcal{H}_i)}$. As we can see, the value of $\theta_\Xi^T$ is related to sine of the minimal angle between the polyhedral sets.
	}	
	\label{fig:mp:angle_polyhedron}
\end{figure}

To prove that sine is always strictly positive, or equivalently that $\theta^T_\Xi$ is always finite, we require the following well known result which connects the distance of a point from a face of a polyhedral set to the distance from the affine subspace defining that face. A simple proof of this lemma can be found in \citep[Pages 299-301]{guler2010foundations}.
\begin{lemma}[\citet{hoffman2003approximate}]\label{lemma:Hoffman}
	Let $S\subseteq \R^d$ be a polyhedron and let $L \subseteq \R^d$ be an affine subspace. 
	If $S\cap L\neq \emptyset$, then there exists $\theta \ge 0$ such that
	$$\dist{(\bz,S\cap L)}\leq \theta \dist(\bz, L), \quad \forall \bz \in S.$$
\end{lemma}
We are now ready to show that the finiteness of $\theta^T_\Xi$.
\begin{lemma} \label{lem:mp:angle}
	If $T \subseteq \Xi$ is a nonempty polyhedron, then $1 \le \theta_\Xi^T < \infty$. 
\end{lemma}
\begin{proof}{Proof.}
	It trivially holds from Definition~\ref{defn:maximal_tangent} that $\theta^T_\Xi \ge 1$, so it remains to show that $\theta^T_\Xi < \infty$.  If $F \cap \mathcal{H}_i = \emptyset$ for all $F \in \Face(T)$ and $i \in [\tilde{m}]$, then $$\theta^T_\Xi = \max \{1,-\infty \} = 1.$$
	Otherwise, consider any $F \in \Face(T)$ and $i \in [\tilde{m}]$ such that $F \cap \mathcal{H}_i \neq \emptyset$.  Then, it follows from Hoffman's lemma (Lemma~\ref{lemma:Hoffman}) that $\theta_i(F) < \infty$.  Since we are maximizing over finitely many faces and hyperplanes, $\theta_\Xi^T$ must be finite as well. 
	\halmos \end{proof}

\section{Proof of Lemma~\ref{lem:mp:feas_then_C_feas} from Section~\ref{sec:mp:asymptotic:L_Feasibility}}\label{appx:C_feas}
In this appendix, we utilize the definitions and results from Appendix~\ref{appx:polyhedral} to prove Lemma~\ref{lem:mp:feas_then_C_feas}. We repeat the lemma below for convenience. 
\begin{repeatlemma} 
	Conditions~[\ref{ass:feas}], [\ref{lem:mp:feas:II}], and [\ref{lem:mp:feas:III}] are equivalent. 
	\begin{enumerate}[label={\bf [A\arabic*]},ref=A\arabic*]
		\setcounter{enumi}{4}
		\item There exists a first-stage decision $\bx\in \Real^n $ and a radius $\epsilon>0$ for which the following holds:
		\begin{enumerate}[(a)]
			\item $Q(\bx,\bzeta)<\infty$ for all $\bzeta\in\Xi$. \label{eq:lem:mp:item_feas}
			\item For each  minimal face $F$ of $\Xi$, there exists a recourse matrix $\bby^F\in\Real^{r\times d}$ such that $Q{}^{\left\{\bby^F\right\}}_{{\epsilon}}(\bx,\bzeta)<\infty$ for all $\bzeta\in F$. \label{eq:lem:mp:item_feas_2}
		\end{enumerate}
		\item There exists a first-stage decision $\bx\in \Real^n $, a radius $\epsilon>0$, and a finite set of recourse matrices $\mathcal{C}\subseteq\Real^{r\times d}$  such that  $Q_{\epsilon}^\mathcal{C}(\bx,\bzeta)<\infty$ for all $\bzeta\in \Xi$.
	\end{enumerate}
\end{repeatlemma}
Our proof of Lemma~\ref{lem:mp:feas_then_C_feas} is split into three parts. Specifically, we show that  [\ref{lem:mp:feas:III}] implies [\ref{ass:feas}] (Lemma~\ref{lem:mp:C_feas_then_L_feas}), [\ref{ass:feas}] implies [\ref{lem:mp:feas:II}] (Lemma~\ref{lem:mp:L_feas_C_feas_on_min_face}), and [\ref{lem:mp:feas:II}] implies [\ref{lem:mp:feas:III}]  (Lemma~\ref{lem:mp:C_feas_on_min_face_C_feas}). These three results when combined prove Lemma~\ref{lem:mp:feas_then_C_feas}.   
\begin{customlemma}{1A}\label{lem:mp:C_feas_then_L_feas}
	[\ref{lem:mp:feas:III}] implies [\ref{ass:feas}]. 
\end{customlemma} 
\begin{proof}{Proof.}
	Let $\bx\in \R$, $\mathcal{C}\subseteq\R^{r\times d}$, and $\bar{\epsilon}>0$ satisfy condition~[\ref{lem:mp:feas:III}].
	Since $\mathcal{C}\subseteq\R^{r\times d}$, it follows from the Definition~\ref{def:Q_C} that
	$Q^{\R^{r\times d}}_{\bar{\epsilon}}(\bx,\bzeta)\leq Q^{\mathcal{C}}_{\bar{\epsilon}}(\bx,\bzeta)<\infty$ for all $\bzeta\in \Xi$, 
	which immediately implies that [\ref{ass:feas}] holds. \halmos 
\end{proof}
To prove the other two directions (Lemmas~\ref{lem:mp:feas_then_C_feas}B and \ref{lem:mp:feas_then_C_feas}C), 
we utilize the following additional result (Lemma~\ref{lem:mp:shifting_minimal_face_balls}) which relates the ``radius of points" (Definition~\ref{defn:radius_point}) to sets of the form $B(\hat{\bzeta},\epsilon)\cap\Xi$. More precisely, Lemma~\ref{lem:mp:shifting_minimal_face_balls} shows that if two points $\hat{\bzeta},\bar{\bzeta}\in\Xi$ are contained in the relative interior of the same face of $\Xi$,  and if  $\epsilon \ge 0 $ is no larger than the radius of either point, then $B(\hat{\bzeta},\epsilon)\cap\Xi$ and $B(\bar{\bzeta},\epsilon)\cap\Xi$ have the same geometric shape; an illustration is provided in Figure~\ref{fig:mp:shifting_balls}. 
\begin{figure}[t]
	\centering
	\definecolor{mycolor}{rgb}{.5,0.7,1}

%
%
%
\begin{tikzpicture}[scale=0.8]
\def\zax{2.5}
\def\zay{2.5}
\def\zbx{3.66666}
\def\zby{6}
\def\epsilona{1.55}
\def\epsilonb{0.98}

\filldraw [fill=white] (2,1)--(4,7)--(1,7)--(0,4)--(1.5,1)--(2,1);
\begin{scope}
\clip   (2,1)--(4,7)--(1,7)--(0,4)--(1.5,1)--(2,1);
\draw [black,line width=0.5pt,fill=red!50] (\zax,\zay) circle (\epsilona);
\draw [black,line width=0.5pt,fill=mycolor] (\zbx,\zby) circle (\epsilonb);
\def\n{5}
\foreach \i in {1,...,\n}{
	\draw [black,line width=0.5pt,dashed,fill=mycolor,opacity=0.2] ({\i*\zax/\n+(1-\i/\n)*\zbx},{\i/\n*\zay+(1-\i/\n)*\zby}) circle (\epsilonb);
	\draw [black,line width=0.5pt,dashed] ({\i*\zax/\n+(1-\i/\n)*\zbx},{\i/\n*\zay+(1-\i/\n)*\zby}) circle (\epsilonb);
}
\end{scope}

\draw [line width=0.5pt,black] (2,1)--(4,7)--(1,7)--(0,4)--(1.5,1)--(2,1);
\node [] at (\zax-0.15,\zay+0.10) {{\small  $\hat{\bzeta}$}};
\node [] at (\zbx-0.15,\zby+0.10) {{\small  $\bar{\bzeta}$}};
\begin{scope}
\clip  (2,8)--(-1,4)--(2,0) -- (4,0) -- (4,8);
\fill [black] (\zax,\zay) circle (1pt);
\fill [black] (\zbx,\zby) circle (1pt);
\end{scope}
\node at (1.5,0) {$\hat{\bzeta}$ and $\bar{\bzeta}$ on a facet of $\Xi$};
\end{tikzpicture}\qquad \qquad
\begin{tikzpicture}[scale=0.8]
\def\zax{1.6}
\def\zay{4.3}
\def\epsilona{1.42}
\def\zbx{2.9}
\def\zby{6}
\def\epsilonb{0.72}

\filldraw [fill=white] (2,1)--(4,7)--(1,7)--(0,4)--(1.5,1)--(2,1);
\begin{scope}
\clip   (2,1)--(4,7)--(1,7)--(0,4)--(1.5,1)--(2,1);
\draw [black,line width=0.5pt,fill=mycolor] (\zax,\zay) circle (\epsilona);
\draw [black,line width=0.5pt,fill=red!50] (\zbx,\zby) circle (\epsilonb);
\def\n{5}
\foreach \i in {1,...,\n}{
	\draw [black,line width=0.5pt,dashed,fill=red!50,opacity=0.2] ({\i*\zax/\n+(1-\i/\n)*\zbx},{\i/\n*\zay+(1-\i/\n)*\zby}) circle (\epsilonb);
	\draw [black,line width=0.5pt,dashed] ({\i*\zax/\n+(1-\i/\n)*\zbx},{\i/\n*\zay+(1-\i/\n)*\zby}) circle (\epsilonb);
}
\end{scope}
\draw [line width=0.5pt,black] (2,1)--(4,7)--(1,7)--(0,4)--(1.5,1)--(2,1);
\begin{scope}
\clip  (2,8)--(-1,4)--(2,0) -- (4,0) -- (4,8);
\node [] at (\zax+0.15,\zay+0.15) {{\small  $\bar{\bzeta}$}};
\fill [black] (\zax,\zay) circle (1pt);
\node [] at (\zbx+0.15,\zby+0.15) {{\small  $\hat{\bzeta}$}};
\fill [black] (\zbx,\zby) circle (1pt);
\end{scope}
\node at (1.5,0) {$\hat{\bzeta}$ and $\bar{\bzeta}$ in the interior of $\Xi$};
\end{tikzpicture} 
		\caption{For each of the figures, the red region is the set $B(\hat{\bzeta}{},\Delta_\Xi(\hat{\bzeta}{})) \cap \Xi$ and the blue region is the set $B(\bar{\bzeta}{},\Delta_\Xi(\bar{\bzeta}{})) \cap \Xi$. The dotted regions show that the smaller set can be shifted between the two points while maintaining its shape. 
	}
	\label{fig:mp:shifting_balls}
\end{figure}
\begin{lemma} \label{lem:mp:shifting_minimal_face_balls}
	Let $F \in \Face(\Xi)$. Then for all $\hat{\bzeta}, \bar{\bzeta} \in \relint(F)$ and all $\epsilon \in[0,\min\{\Delta_\Xi(\hat{\bzeta}),\Delta_\Xi(\bar{\bzeta})\}]$, 
	\begin{align*}
	B(\hat{\bzeta}, \epsilon) \cap \Xi + \bar{\bzeta} - \hat{\bzeta} = B(\bar{\bzeta},\epsilon) \cap \Xi.
	\end{align*}
	Moreover, if $F$ is a minimal face of $\Xi$, then the equality holds true for any $\epsilon\geq 0$.
\end{lemma}
\begin{proof}{Proof.}
	Let $F \in \Face(\Xi)$, $\hat{\bzeta}, \bar{\bzeta} \in \relint(F)$ and 
	\begin{align}
	\epsilon \in[0,\min\{\Delta_\Xi(\hat{\bzeta}),\Delta_\Xi(\bar{\bzeta})\}]. \label{line:mp:giving_title_to_epsilon}
	\end{align} 
	 Our proof consists of showing that 
	\begin{align}
	B(\hat{\bzeta},\epsilon) \cap \Xi+\bar{\bzeta} - \hat{\bzeta}\subseteq B(\bar{\bzeta},\epsilon) \cap \Xi, \label{line:mp:okay_lets_be_creative}
	\end{align}
	and the converse follows by identical reasoning.
	
	We readily observe that
	\begin{align*}
	B(\hat{\bzeta},\epsilon) \cap \Xi + \bar{\bzeta} - \hat{\bzeta}\subseteq  B(\bar{\bzeta},\epsilon), 
	\end{align*}
	so it remains to be shown that
	\begin{align}
	B(\hat{\bzeta},\epsilon) \cap \Xi + \bar{\bzeta} - \hat{\bzeta}\subseteq  \Xi. \label{line:mp:choice_of_bzeta_in_second_lemma_2}
	\end{align}
	Indeed, choose any arbitrary $\bzeta \in B(\hat{\bzeta},\epsilon) \cap \Xi$, and observe that the active index sets of the two points satisfy $I_{\Xi}(\hat{\bzeta}) = I_{\Xi}(\bar{\bzeta})$ since $\hat{\bzeta}$ and $\bar{\bzeta}$ are in the relative interior of the same face of $\Xi$. 
	Recall from Appendix~\ref{appx:polyhedral} that the $i$-th constraint of $\Xi$ is given by  $g^0_i-\bg_i^\intercal\bzeta\ge 0$.   
	Therefore, for any constraint $i\notin I_\Xi(\bar{\bzeta})$, 
	\begin{align}
	g^0_i-\bg_i^\intercal(\bzeta + \bar{\bzeta} - \hat{\bzeta})&\geq \Delta_\Xi(\bar{\bzeta}) - \bg_i^\intercal(\bzeta-\hat{\bzeta}) \notag \\ 
	&\geq \Delta_\Xi(\bar{\bzeta}) - \epsilon \notag \\
	&\geq0. \label{line:mp:i_am_bad_at_naming_things}
	\end{align}
	Indeed, the first inequality follows from  Definition~\ref{defn:radius_point} and line~\eqref{eq:dist_point_hyper} 
	(see Appendix~\ref{appx:poly:general}). The second inequality follows from $\| \bg_i \|_* = 1$, $\|\bzeta-\hat{\bzeta}\|\leq \epsilon$, and H\"{o}lder's inequality. Finally, line~\eqref{line:mp:i_am_bad_at_naming_things} follows from \eqref{line:mp:giving_title_to_epsilon}. Moreover, for any constraint $i \in I_\Xi(\bar{\bzeta})$, 
	\begin{align}
	g^0_i-\bg_i^\intercal(\bzeta + \bar{\bzeta} - \hat{\bzeta}) = g^0_i-\bg_i^\intercal\bzeta \ge 0, \label{line:mp:i_am_bad_at_naming_things_2}
	\end{align}
	where the equality follows from $\bg_i^\intercal\hat{\bzeta}=\bg_i^\intercal\bar{\bzeta}=g^0_i$, and the inequality follows from $\bzeta\in \Xi$. Combining \eqref{line:mp:i_am_bad_at_naming_things} and \eqref{line:mp:i_am_bad_at_naming_things_2}, and recalling that $\bzeta$ was chosen arbitrarily, we have shown that line~\eqref{line:mp:choice_of_bzeta_in_second_lemma_2} holds. 
	
	Suppose now that $F$ is a minimal face of $\Xi$, and consider the vector $\br \triangleq \hat{\bzeta}-\bar{\bzeta}.$ Since every minimal face is an affine subspace, 
	$$\bar{\bzeta}+\alpha\br\in F, \quad \forall \alpha \in \R,$$
	and since $F\subseteq\Xi$, 
	$$\bg_i^\intercal(\bar{\bzeta}+\alpha\br)\leq g^0_i, \quad \forall \alpha \in \R, \; i \in [\tilde{m}]. $$ 
	Consequently, $$\bg_i^\intercal\br=0, \quad \forall i\in[\tilde{m}].$$ 
	Therefore, for all $\epsilon \ge 0$, $i \in [\tilde{m}]$, and $\bzeta \in B(\hat{\bzeta},\epsilon) \cap \Xi,$
	we have shown that
	\begin{align*}
	g_i^0 - \bg_i^\intercal (\bzeta + \bar{\bzeta} - \hat{\bzeta}) = g_i^0 - \bg_i^\intercal \bzeta + \bg_i^\intercal \br = g_i^0 - \bg_i^\intercal \bzeta \ge 0,
	\end{align*} 
	which concludes the proof for minimal faces.
	\Halmos \end{proof}
With the above lemma, we are now ready to prove the second and third parts of Lemma~\ref{lem:mp:feas_then_C_feas}
\begin{customlemma}{1B}\label{lem:mp:L_feas_C_feas_on_min_face}
	 	[\ref{ass:feas}] implies [\ref{lem:mp:feas:II}]. 
\end{customlemma}
\begin{proof}{Proof.}
	Let the first-stage decision $\bx\in \R^n$ and radius ${\epsilon}>0$ satisfy condition~[\ref{ass:feas}]. Then it follows from the definitions of $Q(\bx,\cdot)$ and $Q^{\R^{r\times d}}_{{\epsilon}}(\bx,\cdot)$ that
	\begin{align*}
	Q(\bx,\bzeta)\leq Q^{\R^{r\times d}}_{{\epsilon}}(\bx,\bzeta)<\infty,\; \forall \bzeta\in \Xi,
	\end{align*}
	where the last inequality follows from condition~[\ref{ass:feas}]. Therefore we have shown that the first-stage decision $\bx \in \R^n$ satisfies condition~[\ref{lem:mp:feas:II}]\ref{eq:lem:mp:item_feas}.
	
	Consider any arbitrary minimal face $F$ of the polyhedron $\Xi \subseteq \R^n$, and fix a realization $\bar{\bzeta} \in F$. Then it follows from condition~[\ref{ass:feas}] that there exists a vector $\bar{\by}^{0} \in \R^r$ and matrix $\bby \in \R^{r \times d}$ such that 
	\begin{equation} 
	\bbt {\bx} + \bbw\left(\bar{\by}^{0} + \bby \bzeta \right) \ge \bh(\bzeta),\; \forall \bzeta \in B(\bar{\bzeta}, \bar{\epsilon}) \cap \Xi .\label{eqn:mp:hat_z_bar_delta_good}
	\end{equation}
	We will henceforth use the notation $\bby^F \triangleq \bby$.
	
	Now consider any arbitrary realization $\hat{\bzeta} \in F$. Since $F$ is a minimal face, it is an affine subspace, which implies that 
	$\bar{\bzeta} + \alpha (\hat{\bzeta}- \bar{\bzeta}) \in F$ for all $\alpha \in \R$. 
	Therefore, condition [\ref{ass:feas}] implies that
	\begin{align}\label{lem:mp:ass_farkas}
	\left \{ \by^0 \in \R^{r}: \; \bbt \bx + \bbw \by^0 \ge \bh^0 + \bbh \bar{\bzeta} + \bbh (\alpha(\hat{\bzeta} - \bar{\bzeta})) \right \} \neq \emptyset, \quad \forall \alpha \in \R.
	\end{align}
Using \eqref{lem:mp:ass_farkas}, we will now prove that 
		\begin{equation}
		\left \{ \tilde{\by}^0 \in \R^r: \;  \bbw\tilde{\by}^0\geq -\bbh(\bar{\bzeta}-\hat{\bzeta}) \right \} \neq \emptyset. \label{eqn:mp:rayvector}\end{equation}
		Indeed, suppose for the sake of contradiction that line~\eqref{eqn:mp:rayvector} is false. Then, it follows from Farkas' lemma 
		 that there exists a vector $\bp \in \R^m$ which satisfies
		\begin{align*}
		\bp^\intercal  \bbh (\bar{\bzeta}-\hat{\bzeta})< 0,\; \bbw^\intercal \bp = \bzero,\; \bp\ge 0.
		\end{align*}
		Since $\bp^\intercal \bbh (\bar{\bzeta}-\hat{\bzeta}) < 0$, there exists a sufficiently small ${\alpha}<0$ such that
		\begin{align*}
		\alpha \bp^\intercal  \bbh (\bar{\bzeta}-\hat{\bzeta})  - \bp^\intercal \left(\bh(\bar{\bzeta}) -\bbt\bx \right) < 0 ,\; \bbw^\intercal \bp = \bzero, \; \bp \ge \bzero.
		\end{align*}
		Applying Farkas' lemma once again, we conclude that there does \emph{not} exist a $\by^0 \in \R^r$ that satisfies
		\begin{align*}
		\bbt \bx +  \bbw \by^0 \ge \bh \left(\bar{\bzeta} + \alpha(\hat{\bzeta} - \bar{\bzeta}) \right). 
		\end{align*}
		However, this forms a contradiction with line~\eqref{lem:mp:ass_farkas}. Therefore, we have proven that line~\eqref{eqn:mp:rayvector} holds. 

	We now combine lines~\eqref{eqn:mp:hat_z_bar_delta_good} and \eqref{eqn:mp:rayvector} to show that the inequality $Q{}^{\left\{\bby^F\right\}}_{{\epsilon}}(\bx,\hat{\bzeta})<\infty$ is satisfied. Indeed, 
	 it follows from Lemma~\ref{lem:mp:shifting_minimal_face_balls}, the fact that $F$ is a minimal face, and line~\eqref{eqn:mp:hat_z_bar_delta_good} that 
	\begin{align} 
	\bbt {\bx} + \bbw\left(\bar{\by}^{0} + \bby^F\left( \bzeta + \bar{\bzeta} - \hat{\bzeta} \right) \right)&\ge \bh(\bzeta + \bar{\bzeta}- \hat{\bzeta} ),&& \forall \bzeta \in B(\hat{\bzeta}, \bar{\epsilon}) \cap \Xi \notag \\
		\iff \bbt {\bx} + \bbw\left( \left( \bar{\by}^{0} + \bby^F \bar{\bzeta} - \bby^F \hat{\bzeta} \right) +  \bby^F \bzeta  \right)&\ge \bh(\bzeta) + \bbh(\bar{\bzeta} - \hat{\bzeta}),&& \forall \bzeta \in B(\hat{\bzeta}, \bar{\epsilon}) \cap \Xi, \label{eqn:mp:hat_z_bar_shifted_delta_good}
	\end{align}
	where the second line follows from the definition of $\bh(\bzeta) \triangleq \bh^0 + \bbh \bzeta$. Combining lines~\eqref{eqn:mp:hat_z_bar_shifted_delta_good} and \eqref{eqn:mp:rayvector}, there exists a vector $\tilde{\by}^0 \in \R^r$ such that 
	\begin{align*}
	\left( \bbt {\bx} + \bbw\left( \left( \bar{\by}^{0} + \bby^F \bar{\bzeta} - \bby^F \hat{\bzeta} \right) +  \bby^F \bzeta  \right)\right)  + \bbw \tilde{\by}^0 &\ge\left(  \bh(\bzeta) + \bbh(\bar{\bzeta} - \hat{\bzeta})\right) - \bbh(\bar{\bzeta} - \hat{\bzeta}), \forall \bzeta \in B(\hat{\bzeta}, \bar{\epsilon}) \cap \Xi, 
	\end{align*}
	or equivalently, 
	\begin{align*}
	\bbt {\bx} + \bbw\left( \underbrace{\left( \bar{\by}^{0} + \tilde{\by}^0 + \bby^F \bar{\bzeta} - \bby^F \hat{\bzeta} \right)}_{\hat{\by}^0} +  \bby^F \bzeta  \right)  &\ge  \bh(\bzeta) , \quad \forall \bzeta \in B(\hat{\bzeta}, \bar{\epsilon}) \cap \Xi. 
	\end{align*}
	We have thus proven that the inequality $Q{}^{\left\{\bby^F\right\}}_{{\epsilon}}(\bx,\hat{\bzeta})<\infty$ holds. Since the realization $\hat{\bzeta} \in F$ was chosen arbitrarily, and since the above reasoning holds for any minimal face of $\Xi$, our proof of condition~[\ref{lem:mp:feas:II}]\ref{eq:lem:mp:item_feas_2} is complete. 
	\Halmos
\end{proof}

\begin{customlemma}{1C}\label{lem:mp:C_feas_on_min_face_C_feas}
	[\ref{lem:mp:feas:II}] implies [\ref{lem:mp:feas:III}].
\end{customlemma}
\begin{proof}{Proof.}  
	Assume that condition~[\ref{lem:mp:feas:II}] holds, and let us define 
	\begin{align}\label{eq:C_def} \mathcal{C}\triangleq \left \{ \bby^F : F \text{ is a minimal face of } \Xi \right \}, 
	\end{align}
	where the matrices $\bby^F \in \R^{r \times d}$ are those given by condition~[\ref{lem:mp:feas:II}]\ref{eq:lem:mp:item_feas_2}.  Furthermore, let the first-stage decision $\bx \in \R^n$ and $\bar{\epsilon} > 0$ be those given by condition~[\ref{lem:mp:feas:II}]\ref{eq:lem:mp:item_feas_2}. Since any polyhedron $\Xi \subseteq \R^d$ has a finite number of minimal faces, it follows that $\mathcal{C} \subseteq \R^{r \times d}$ is a finite set. 

For each dimension $k\in\{\dimlower(\Xi),\ldots,\dim(\Xi)\}$, let us define 
\begin{align}
\epsilon_k &\triangleq  (4\theta^{\Xi}_\Xi)^{\underline{\text{dim}}(\Xi)-k} \min\{\Delta_\Xi(\Xi), \bar{\epsilon}\}. \label{line:mp:defn:delta_k}
\end{align}
It was shown in Lemma~\ref{lem:mp:radius} that any radius of a polyhedron $\Delta_\Xi(\Xi)$ is strictly positive, and shown in Lemma~\ref{lem:mp:angle} that the angle of polyhedron $\theta^\Xi_\Xi$ is contained in the interval $[1,\infty)$. 
Hence, it follows that each $\epsilon_k$ is strictly positive. Moreover, we easily observe that $\epsilon_k$ is decreasing in $k$.  Therefore, defining $\epsilon \triangleq \epsilon_{\dim(\Xi)}$, condition~[\ref{lem:mp:feas:III}] is a direct result of 
\begin{align*}Q^{\mathcal{C}}_{{\epsilon}}(\bx,\hat{\bzeta})\leq Q^{\mathcal{C}}_{\epsilon_k}(\bx,\hat{\bzeta})<\infty\quad\forall \hat{\bzeta}\in \relint(F)\end{align*}
for each dimension $k\in\{\dimlower(\Xi),\ldots,\dim(\Xi)\}$ and $F \in \Face(\Xi)$ with $\dim(F) = k$, where the last inequality $Q^{\mathcal{C}}_{\epsilon_k}(\bx,\hat{\bzeta})<\infty$  is a result of the following claim, which we now prove.

\paragraph{\underline{Claim:}}
Let the first-stage decision $\bx\in \Real^n$ and radius $\bar{\epsilon}>0$ satisfy condition~[\ref{lem:mp:feas:II}], and let
$\mathcal{C}\subseteq\Real^{r\times d}$ be defined as in line~\eqref{eq:C_def}.
Then for each dimension $k\in\{\dimlower(\Xi),\ldots,\dim(\Xi)\}$, for each $F \in \Face(\Xi)$ with $\dim(F) = k$, and for all realizations $\hat{\bzeta} \in \relint(F)$, there exists a vector $\by^0\in\Real^r$ and matrix $\bby\in \mathcal{C}$ which satisfy 
\begin{align}\label{eq:fixed_Y_induction}
\bbt \bx + \bbw (\by^0 + \bby \bzeta) \ge \bh(\bzeta),\; \forall \bzeta \in B(\hat{\bzeta}, \epsilon_k) \cap \Xi.
\end{align} 
\paragraph{\underline{Proof of Claim:}}
We prove the claim by induction on $k$.
	\begin{itemize}
		\item \textbf{Base case}: Suppose  that $k = \dimlower(\Xi)$. We recall that the faces $F \in \Face(\Xi)$ with dimension $\dim(F) = \dimlower(\Xi)$  are exactly the minimal faces of $\Xi$.  Thus, the claim follows immediately from condition~[\ref{lem:mp:feas:II}]\ref{eq:lem:mp:item_feas_2} and from the definition of $\mathcal{C} \subseteq \R^{r \times d}$ in line~\eqref{eq:C_def}.  
		
		\item \textbf{Induction step}: Fix any dimension $k \in \{\dimlower(\Xi)+1,\ldots,\dim(\Xi)\}$, and assume that the claim holds for all $F \in \Face(\Xi)$ with $\dim(F) \le k-1$. 
		
		Consider any $F \in \Face(F)$ with $\dim(F) = k$ and any realization $\hat{\bzeta} \in \relint(F)$. There are two cases to consider.
		\begin{itemize}
			\item (\underline{Case 1}) Suppose that 
			\begin{align}
			\Delta_\Xi(\hat{\bzeta}) \le {2} \epsilon_k. \label{line:mp:condition_one_for_final_proof}
			\end{align}
			For this case, we will prove a stronger version of the claim: namely, that there exist a vector $\by^0 \in \R^r$ and matrix $\bby \in \mathcal{C}$ such that 
			\begin{align}\label{eq:fixed_Y_induction_doubled}
			\bbt \bx + \bbw (\by^0 + \bby \bzeta) \ge \bh(\bzeta),\; \forall \bzeta \in B(\hat{\bzeta}, 2\epsilon_k) \cap \Xi.
			\end{align} 
			Indeed, since \eqref{line:mp:condition_one_for_final_proof} holds,  there exists a constraint $i\in[\tilde{m}]\setminus I_\Xi(\hat{\bzeta})$ such that
			\begin{align}
			\dist(\hat{\bzeta},\mathcal{H}_i)\leq 2\epsilon_k. \label{line:mp:dist_first_case_shim}
			\end{align}
			We will now show that 
			\begin{align}
			\mathcal{H}_i\cap F\neq \emptyset. \label{line:mp:hyperplane_face_empty}
			\end{align}
			Indeed, suppose for the sake of contradiction that \eqref{line:mp:hyperplane_face_empty} was false. Then,
			\begin{align}
			\Delta_\Xi(\Xi) &\ge \min \{ \Delta_\Xi(\Xi), \bar{\epsilon} \} \notag \\
			&>2\epsilon_k \notag \\ 
			&\geq \dist(\hat{\bzeta},\mathcal{H}_i) \notag \\ 
			&\geq \dist(F,\mathcal{H}_i) \nonumber\\
			&\geq \Delta_\Xi(\Xi).\nonumber
			\end{align}
			The second line 
			follows from \eqref{line:mp:defn:delta_k}, Lemma~\ref{lem:mp:angle}, and $k > \dimlower(\Xi)$.  
			The third line follows from 
			\eqref{line:mp:dist_first_case_shim}. 
			The fourth and fifth lines follow from the definition of distance between sets (\eqref{eq:dist_poly} in Appendix~\ref{appx:poly:general}), Definition~\ref{defn:radius_face}, and the supposition that $\mathcal{H}_i\cap F=\emptyset$. 
			Thus, we have a contradiction, which concludes the proof of \eqref{line:mp:hyperplane_face_empty}. 
			
			Since we have shown that $F' \triangleq F\cap \mathcal{H}_i$ is nonempty and $i\notin I_\Xi(\hat{\bzeta})$, $F'$ must be a proper face of $F$, and therefore $\ell \triangleq \dim(F') \le k-1$.
			Letting $\tilde{\bzeta}$ be the projection of $\hat{\bzeta}$ onto $F'$,
			\begin{align}
			\| \tilde{\bzeta}-\hat{\bzeta} \|=\dist(\hat{\bzeta},F\cap \mathcal{H}_i)\leq \theta_\Xi^\Xi\dist(\hat{\bzeta},\mathcal{H}_i)\leq 2\theta^\Xi_\Xi\epsilon_k, \label{line:mp:so_many_labels}
			\end{align}
			where the first inequality follows from Definition~\ref{defn:maximal_tangent}  and the second inequality follows from \eqref{line:mp:dist_first_case_shim}. Moreover, let $F''\in \Face(F)$ be the face which satisfies $\tilde{\bzeta}\in\relint(F'')$,  and define $\ell'\triangleq \dim(F'')\leq \ell$.
			Since $\ell'\leq \ell \le k-1$, it follows from the induction hypothesis that \eqref{eq:fixed_Y_induction} 
			is satisfied at $\tilde{\bzeta}$ for some vector $\by^0 \in \R^r$, matrix $\bby \in \mathcal{C}$, and radius $\epsilon_{\ell'}$. Moreover, consider any arbitrary 
			\begin{align}
			\bzeta \in B(\hat{\bzeta}, 2\epsilon_k) \cap \Xi. \label{line:mp:choose_a_special_zeta}
			\end{align} Then, 
			$$\|\bzeta - \tilde{\bzeta}\| \leq \| \bzeta - \hat{\bzeta} \| + \| \tilde{\bzeta} - \hat{\bzeta} \| \le 4 \theta_{\Xi}^\Xi\epsilon_{k} \le \epsilon_{\ell}\leq \epsilon_{\ell'}.$$
			Indeed, the first inequality follows from the triangle inequality. The second inequality follows from \eqref{line:mp:so_many_labels}, \eqref{line:mp:choose_a_special_zeta}, and Lemma~\ref{lem:mp:angle} (which says that $\theta_{\Xi}^\Xi\geq 1$). The final two inequalities follow because $\ell'\leq \ell \le k-1 $ and the definition of $\epsilon_k$ given in \eqref{line:mp:defn:delta_k}.  Thus, we have proven that 
			\begin{align*}
			B(\hat{\bzeta}, 2\epsilon_{k}) \cap \Xi \subseteq  B(\tilde{\bzeta}, \epsilon_{\ell'}) \cap \Xi, 
			\end{align*}
			and so \eqref{eq:fixed_Y_induction_doubled} must also be satisfied at $\hat{\bzeta}$ with $\by^0$ and $\bby$. Since \eqref{eq:fixed_Y_induction_doubled} is stronger than \eqref{eq:fixed_Y_induction}, we have concluded the proof of the induction step for the case of \eqref{line:mp:condition_one_for_final_proof}. 
			
			\vspace{1em} 
			
			\item (\underline{Case 2}) Suppose that 
			\begin{align}
			\Delta_\Xi(\hat{\bzeta}) > {2} \epsilon_k. \label{line:mp:condition_two_for_final_proof}
			\end{align}
			In this case, we use a proof technique which is visualized in Figure~\ref{fig:mp:norm_shrink}.
			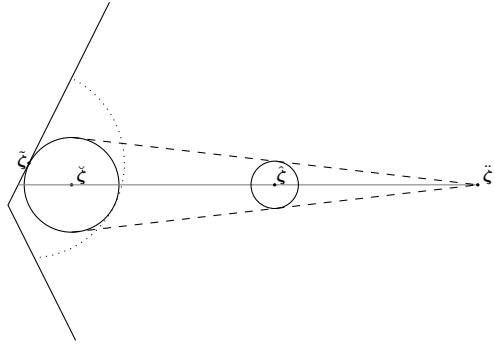
\begin{figure}[t]
				\centering
				\begin{tikzpicture}[scale=0.9]
\def\zax{0.25+0.69}
\def\zay{4.3}
\def\zbx{0.25+0.69+3}
\def\zby{4.3}
\def\zcx{0.25+0.69+6}
\def\zcy{4.3}
\def\zdx{0.31}
\def\zdy{4.62}
\def\ang{153}
\def\radius{0.7}
\def\radiusprevr{1.4}

\tikzmath{\textdistfromcenter=0.1pt;}
\filldraw [fill=white] (1.5,7)--(0,4)--(1,2);
\draw [] (\zax,\zay) circle (\radius);
\node [] at (\zax+0.15,\zay+0.15) {{\tiny  $\breve{\bzeta}$}};
\fill (\zax,\zay) circle (0.7pt);
%
%
\node [] at (\zdx-0.1,\zdy+0.05) {{\tiny  $\tilde{\bzeta}$}};
\fill (\zdx,\zdy) circle (0.7pt);
\begin{scope}
\clip  (1.5,7)--(0,4)--(1,2) -- (10,2) -- (10,6);
\draw [dotted] (\zdx+0.01,\zdy) circle (\radiusprevr);
\draw [color=gray] (0,\zay) -- (\zcx,\zcy);
\end{scope}
\draw [] (\zbx,\zby) circle (\radius/2);
\fill (\zbx,\zby) circle (0.7pt);
\node [] at (\zbx+0.1,\zby+0.15) {{\tiny  $\hat{\bzeta}$}};
%
\fill (\zcx,\zcy) circle (0.7pt);
\node [] at (\zcx+0.15,\zcy+0.15) {{\tiny  $\ddot{\bzeta}$}};

\draw [dashed] (\zax,\zay+\radius)--(\zcx,\zcy);
\draw [dashed] (\zax,\zay-\radius)--(\zcx,\zcy);
%
\end{tikzpicture}
\caption{Visualization of Case 2 - where $2\epsilon_k \le \Delta_\Xi(\hat{\bzeta})$. The ball around $\hat{\bzeta}$ has radius $\epsilon_k$,  and the ball around $\breve{\bzeta}$ has radius $2\epsilon_k$. The radius of the ball (dotted circle) around $\tilde{\bzeta}$ (which is found according to Case 1) is $4\theta_\Xi^\Xi \epsilon_k = \epsilon_{k-1}$. We observe that the ball around $\breve{\bzeta}$ is contained in dotted ball around $\tilde{\bzeta}$. $\hat{\bzeta}$ is equidistant from $\breve{\bzeta}$ and $\ddot{\bzeta}$, and the dashed line illustrate shrinking the ball around $\breve{\bzeta}$ towards $\ddot{\bzeta}$.}
				\label{fig:mp:norm_shrink}
			\end{figure}
			First, we choose any arbitrary $F' \in \Facet(F)$ and  realization $\bar{\bzeta}\in F'$. Indeed, since $F$ is not a minimal face, such a face $F'$ is guaranteed to exist. We then define 
			\begin{align}
			\gamma^+ \triangleq \min \left\{\gamma\geq 0:\Delta_\Xi(\hat{\bzeta}+\gamma(\bar{\bzeta}-\hat{\bzeta}))\leq {2\epsilon_k} \right\},  \label{line:mp:gamma_plus} \\
			\gamma^- \triangleq \min \left\{\gamma\geq 0:\Delta_\Xi(\hat{\bzeta}+\gamma(\hat{\bzeta}-\bar{\bzeta}))\leq {2\epsilon_k} \right\}.\label{line:mp:gamma_minus}
			\end{align}
			We readily observe that $\Delta_\Xi(\cdot)$ is a continuous function on $\relint(F)$ and this function decreases to zero as we get closer to the facets of $F$. Therefore, it follows from \eqref{line:mp:condition_two_for_final_proof} and the intermediate value theorem that \eqref{line:mp:gamma_plus} attains its optimum at $\gamma^+ > 0$, and \eqref{line:mp:gamma_minus} either attains its optimum at $\gamma^->0$ or is infeasible (in which case $\gamma^- = \infty$). We define
			\begin{align*}
			\breve{\bzeta} \triangleq \begin{cases}
			\hat{\bzeta}+\gamma^+(\bar{\bzeta}-\hat{\bzeta}),&\text{if } \gamma^- = \infty,\\
			\argmin \left \{ \|\hat{\bzeta}-\bzeta \|: \; \bzeta\in\{\hat{\bzeta}+\gamma^+(\bar{\bzeta}-\hat{\bzeta}),\hat{\bzeta}+\gamma^-(\hat{\bzeta}-\bar{\bzeta})\} \right \},&\text{otherwise}.
			\end{cases}
			\end{align*} 
			In both cases, since $\gamma^+, \gamma^- >0$, it follows that $\breve{\bzeta} \neq \hat{\bzeta}$ and 
			\begin{align}
			\Delta_{\Xi}(\breve{\bzeta})= {2\epsilon_k}. \label{line:mp:breve_zeta_satisfies}
			\end{align}
			We observe that $\breve{\bzeta}$ satisfies \eqref{line:mp:condition_one_for_final_proof} in Case 1. Therefore, it follows from \eqref{eq:fixed_Y_induction_doubled} that there exists a $\by^0 \in \R^r$ and $\bby \in \mathcal{C}$ such that 
			\begin{align}\label{eq:fixed_Y_induction_doubled_repeat}
			\bbt \bx + \bbw (\by^0 + \bby \bzeta) \ge \bh(\bzeta),\; \forall \bzeta \in B(\breve{\bzeta}, 2\epsilon_k) \cap \Xi.
			\end{align} 
			%
			%
			%
			%
			Now, let us define
			\begin{align*}
			\ddot{\bzeta} \triangleq 2\hat{\bzeta}-\breve{\bzeta},
			\end{align*}
			and note that $\hat{\bzeta}=(\breve{\bzeta}+\ddot{\bzeta})/2$ (see Figure~\ref{fig:mp:norm_shrink}). In both cases of $\gamma^-$, it follows from the definition of $\breve{\bzeta}$ that 
			$\ddot{\bzeta} \in \Xi$. 
			Moreover, it follows from $\ddot{\bzeta} \in \Xi$ and condition [\ref{lem:mp:feas:II}]\ref{eq:lem:mp:item_feas} that there exists $\ddot{\by} \in \R^r$ which satisfies
			\begin{align}\label{eq:claim_ddot_y}
			\bbt \bx + \bbw \ddot{\by} \ge \bh(\ddot{\bzeta}).
			\end{align} 
			Therefore, we can take a combination of \eqref{eq:fixed_Y_induction_doubled_repeat} and \eqref{eq:claim_ddot_y}, obtaining
			\begin{align*} 
			\bbt \bx + \bbw \left(\frac{1}{2}(\by^0+\ddot{\by}) + \frac{1}{2}\bby \bzeta\right) \ge \frac{1}{2}(\bh(\bzeta)+\bh(\ddot{\bzeta})),\; \forall \bzeta \in B(\breve{\bzeta}, 2\epsilon_{k}) \cap \Xi,
			\end{align*} 
			which is equivalent to
			\begin{align*} 
			\bbt \bx + \bbw \left(\frac{1}{2}({\by}^0+\ddot{\by}-\bby\ddot{\bzeta}) + \bby \left(\frac{1}{2}\bzeta+\frac{1}{2}\ddot{\bzeta}\right)\right) \ge \bh\left(\frac{1}{2}\bzeta+\frac{1}{2}\ddot{\bzeta}\right),\; \forall \bzeta \in B(\breve{\bzeta}, 2{\delta}_{k}) \cap \Xi.
			\end{align*}
			Defining $$\hat{\by} \triangleq \frac{1}{2}({\by}^0+\ddot{\by}-\bby\ddot{\bzeta}),$$ 
			we have that
			\begin{align}\label{eq:Claim_hat_y1} 
			\bbt \bx + \bbw (\hat{\by} + \bby \bzeta) \ge \bh(\bzeta),\; \forall \bzeta \in \left\{\frac{1}{2}\bzeta+\frac{1}{2}\ddot{\bzeta}:\bzeta\in B(\breve{\bzeta},2\epsilon_k)\cap \Xi\right\}.
			\end{align} 
			We conclude the proof by showing that 
			\begin{align}\label{eq:hatz_shrink_inclusion}
			B\left(\hat{\bzeta},\epsilon_k\right)\cap \Xi\subseteq \left\{\frac{1}{2}\bzeta+\frac{1}{2}\ddot{\bzeta}:\bzeta\in B(\breve{\bzeta},2\epsilon_k)\cap \Xi\right\},\end{align}
			which together with \eqref{eq:Claim_hat_y1} proves the claim. Indeed,  
			\begin{align}
			B\left(\hat{\bzeta},\epsilon_k\right)\cap\Xi
			&\subseteq 
			\left\{\frac{1}{2}\bzeta+\frac{1}{2}\hat{\bzeta}:\bzeta\in B(\hat{\bzeta},2\epsilon_k)\cap \Xi\right\}
			\label{line:mp:application_of_cone_shrinkage}\\
			&=
			\left\{\frac{1}{2}\bzeta+\frac{1}{2}\hat{\bzeta}:\bzeta\in B(\breve{\bzeta},2\epsilon_k)\cap \Xi+\hat{\bzeta}-\breve{\bzeta}\right\}
			\label{line:mp:not_sure_why_but_shim_explaining}\\
			&=\left\{\frac{1}{2}\bzeta+\frac{1}{2}\hat{\bzeta}:\bzeta\in B(\breve{\bzeta},2\epsilon_k)\cap \Xi\right\}+\frac{1}{2}(\hat{\bzeta}-\breve{\bzeta})\nonumber\\
			&=\left\{\frac{1}{2}\bzeta+\frac{1}{2}\ddot{\bzeta}:\bzeta\in B(\breve{\bzeta},2\epsilon_k)\cap \Xi\right\}+\frac{1}{2}(\hat{\bzeta}-\ddot{\bzeta})+\frac{1}{2}(\hat{\bzeta}-\breve{\bzeta})\nonumber\\
			&=\left\{\frac{1}{2}\bzeta+\frac{1}{2}\ddot{\bzeta}:\bzeta\in B(\breve{\bzeta},2\epsilon_k)\cap \Xi\right\}.
			\label{line:mp:kaboom}
			\end{align}
			Line \eqref{line:mp:application_of_cone_shrinkage} follows from \eqref{line:mp:condition_two_for_final_proof} and Lemma~\ref{lemma:coneshrinkege}. Line~\eqref{line:mp:not_sure_why_but_shim_explaining}  follows because $\hat{\bzeta},\breve{\bzeta}\in \relint(F)$ by construction, $2\epsilon_k\leq \min\{\Delta_{\Xi}(\breve{\bzeta}),\Delta_{\Xi}(\hat{\bzeta})\}$ by \eqref{line:mp:breve_zeta_satisfies} and \eqref{line:mp:condition_two_for_final_proof}, and Lemma~\ref{lem:mp:shifting_minimal_face_balls}.  
			Line \eqref{line:mp:kaboom} follows from the definition of $\ddot{\bzeta}$. This concludes the proof of Case 2. 
		\end{itemize}
	\end{itemize}
	\Halmos \end{proof}

\section{Proofs for Lemma~\ref{lem:mp:feas} from Section~\ref{sec:mp:lemmas}} \label{appx:shrinkage}
In this appendix, we employ the definitions and results from Appendix~\ref{appx:polyhedral} to prove the following result (Theorem~\ref{thm:mp:ubar_delta}), which is utilized in the proof of Lemma~\ref{lem:mp:feas} in Section~\ref{sec:mp:lemmas}.
\begin{theorem}\label{thm:mp:ubar_delta}
	Let $T \subseteq \Xi$ be a nonempty polyhedron. Given $\bar{\epsilon} > 0$ and $\lambda \in (0,1)$, define
	$${\epsilon} \triangleq \left(\frac{\lambda}{2 \theta_{\Xi}^T} \right)^{\textnormal{dim}(T) - \underline{\textnormal{dim}}(T)} \lambda \min \{ \bar{\epsilon}, \Delta_\Xi(T) \}.$$
	Then for all $\hat{\bzeta} \in T$, there exists a $\bar{\bzeta} \in T$ such that
	\begin{align*}
	B(\hat{\bzeta},{\epsilon})\cap\Xi\subseteq \left\{\lambda \bzeta + (1-\lambda) \bar{\bzeta}: \bzeta\in  B(\bar{\bzeta},\bar{\epsilon})\cap \Xi\right\}.
	\end{align*}
\end{theorem}

To motivate the above theorem, let us recall the structure of the proof of Lemma~\ref{lem:mp:feas}. In essence, the proof of Lemma~\ref{lem:mp:feas} consists of showing  that a first-stage decision $\lambda \bx + (1-\lambda) \bx^*$ has, for each realization $\hat{\bzeta} \in T \triangleq \Xi^*$, a linear decision rule which is feasible for all realizations in an uncertainty set centered at $\hat{\bzeta}$. To establish this, the proof of Lemma~\ref{lem:mp:feas} shows that, for each realization $\hat{\bzeta} \in T$, we can construct a linear decision rule which is feasible for all realizations in the set
\begin{align*}
\left\{\lambda \bzeta + (1-\lambda) \hat{\bzeta}: \bzeta\in  B(\hat{\bzeta},\bar{\epsilon})\cap\Xi\right\},
\end{align*}
where $\bar{\epsilon}> 0$ is a radius which is associated with the first-stage decision $\bx \in \R^n$. However, as Figure~\ref{fig:mp:shrinkage_examples} demonstrates, there generally does not exist a radius $\epsilon \in (0,\bar{\epsilon})$ which makes the following inclusion hold for all $\hat{\bzeta} \in T$:
\begin{align}\label{eq:shrinege_inclusion}
B(\hat{\bzeta}, \epsilon) \cap \Xi \;  \subseteq \; \left\{\lambda \bzeta + (1-\lambda) \hat{\bzeta}: \bzeta\in  B(\hat{\bzeta},\bar{\epsilon})\cap\Xi\right\}.
\end{align}
Therefore, Theorem \ref{thm:mp:ubar_delta} ensures that \eqref{eq:shrinege_inclusion} holds with some ${\epsilon} \in (0,\bar{\epsilon})$ and for all realizations $\hat{\bzeta} \in T$ if, on the right-hand side of \eqref{eq:shrinege_inclusion}, $\hat{\bzeta}$ is replaced by a nearby realization $\bar{\bzeta} \in T$.

\begin{figure}[t]
	\centering
	\input{figures/IssuesShrinkage_new_tikzfig_varying_lambda_shimrit2.tex}
	\label{fig:mp:shrinkage_examples}
\end{figure}


In  principle, \eqref{eq:shrinege_inclusion} would be satisfied for all realizations $\hat{\bzeta} \in T$  by some $\epsilon > 0$ if, on the right-hand side of \eqref{eq:shrinege_inclusion}, we replaced $\Xi$ with the larger set $\tilde{\Xi} \triangleq \Xi + B(\bzero,{\epsilon})$.  However, as Example~\ref{ex:support} from Section~\ref{sec:mp:prelim} illustrates,  there exist problems in which condition~[\ref{ass:feas}] would no longer be satisfied if the set $\Xi$ was replaced with $\Xi + B(\bzero,{\epsilon})$ for any $\epsilon > 0$.

In view of the above motivation, we now present the proof of Theorem~\ref{thm:mp:ubar_delta}:
\begin{proof}{Proof of Theorem~\ref{thm:mp:ubar_delta}.}
	Our proof of Theorem~\ref{thm:mp:ubar_delta} follows immediately from applying the following claim for $F=T$: 
	
	\paragraph{\underline{Claim}: } For each dimension $k \in \{\dimlower(T),\ldots,\dim(T)\}$,  define 
	\begin{align*}
	\alpha_k &\triangleq \left(\frac{2\theta_{\Xi}^T}{\lambda}\right)^{\dim(T)-k}.
	\end{align*}
	Then, for each $F \in \Face(T)$ with $\dim(F) = k$ and each realization $\hat{\bzeta} \in F$, there exists a $\bar{\bzeta} \in F$ such that 
	\begin{equation}\label{eq:shrinkege_rate}B\left(\hat{\bzeta},\alpha_k\epsilon\right)\cap \Xi\subseteq \left\{\lambda \bzeta + (1-\lambda) \bar{\bzeta}: \bzeta\in  B(\bar{\bzeta},\bar{\epsilon})\cap\Xi\right\}.\end{equation}
	
	\paragraph{\underline{Proof of Claim}: } We prove the claim by induction on $k$. 
	
	\begin{itemize}
		\item {\bf Base case:} Suppose  that $k = \dimlower(\Xi)$. We recall that the faces $F \in \Face(T)$ with dimension $\dim(F) = \dimlower(T)$  are exactly the minimal faces of $T$. 
		
		Consider any $F \in \Face(T)$ with $\dim(F) = \dimlower(T)$ and any constraint $i \in [\tilde{m}]$ of the polyhedron $\Xi$. We recall that the polyhedron $T$ is a subset of the polyhedron $\Xi$. 
		Hence, since $F$ is a minimal face of $T$, it must be the case that
		\begin{align*}
		F \cap \mathcal{H}_i \in \{\emptyset,F\}.
		\end{align*}
		Therefore, for any realization $\hat{\bzeta} \in F$, it follows that
		\begin{align}
		\hat{\bzeta} \notin \mathcal{H}_i \iff F \cap \mathcal{H}_i = \emptyset. \label{line:mp:face_hyperplane_empty_or_not}
		\end{align}  
		Furthermore, since $F$ is a polyhedral set contained in $\Xi$,
		\begin{align}
		\Delta_\Xi(F) &= \min \{\rho_1(F), \rho_2(F) \} \notag \\ 
		&\leq \rho_1(F)  \notag \\
		&= \min\limits_{i\in[\tilde{m}]:\mathcal{H}_i\cap F=\emptyset} \dist(F,\mathcal{H}_i) \notag \\ 
		&=  \min\limits_{i\in[\tilde{m}]:\mathcal{H}_i\cap F=\emptyset} \inf_{\bzeta \in F} \dist(\bzeta, \mathcal{H}_i)  \notag \\ 
		&\le  \min\limits_{i\in[\tilde{m}]:\mathcal{H}_i\cap F=\emptyset} \dist(\hat{\bzeta}, \mathcal{H}_i) \notag \\
		&= \min\limits_{i\in[\tilde{m}]:\hat{\bzeta}\notin \mathcal{H}_i}\dist(\hat{\bzeta},\mathcal{H}_i) \notag \\ 
		&= \Delta_\Xi(\hat{\bzeta}). \label{line:mp:applying_defn_radius_point}
		\end{align}
		Indeed, the first three lines follow from Definition~\ref{defn:radius_face}, the fourth line follows from  the definition of the distance between sets (see \eqref{eq:dist_poly} in Appendix~\ref{appx:poly:general}), the fifth line follows from $\hat{\bzeta} \in F$, the sixth line follows from \eqref{line:mp:face_hyperplane_empty_or_not}, and line~\eqref{line:mp:applying_defn_radius_point} follows from Definition~\ref{defn:radius_point}. Therefore,
		\begin{equation}\label{eq:prop_inequalities1}
		\Delta_\Xi(\hat{\bzeta})\geq \Delta_\Xi(F)\geq \Delta_\Xi(T)\geq \min\{\Delta_\Xi(T),\bar{\epsilon}\} =  \alpha_{\underline{\text{dim}}(T)}\frac{{\epsilon}}{\lambda},
		\end{equation}
		where the first inequality follows from \eqref{line:mp:applying_defn_radius_point}, the second inequality follows readily from Definition~\ref{defn:radius_face}, and the final equality follows from the definitions of $\alpha_{\underline{\text{dim}}(T)}$ and $\epsilon$. 
		Combining  line~\eqref{eq:prop_inequalities1} with Lemma~\ref{lemma:coneshrinkege}, we have shown that
		\begin{align*}B\left(\hat{\bzeta},\alpha_{\underline{\text{dim}}(T)}{\epsilon}\right)\cap \Xi&= \left\{\lambda \bzeta + (1-\lambda) \hat{\bzeta}: \bzeta\in  B(\hat{\bzeta},\alpha_{\underline{\text{dim}}(T)}\lambda^{-1}{\epsilon})\cap\Xi\right\}\\
		&\subseteq \left\{\lambda \bzeta + (1-\lambda) \hat{\bzeta}: \bzeta\in  B(\hat{\bzeta},\bar{\epsilon})\cap\Xi\right\}
		,\end{align*}
		where the last inclusion follows from $\alpha_{\underline{\text{dim}}(T)}\lambda^{-1}{\epsilon} \le \bar{\epsilon}$. 
		Thus, we conclude that the claim holds for every minimal face of $T$. 
		
		\item {\bf Induction step:} 
		Fix any dimension $k \in \{\dimlower(T),\ldots,\dim(T)\}$, and assume that the claim holds for all $F \in \Face(T)$ with $\dim(F) \le k$.
		We now prove the claim for any face of dimension $k+1$.  Indeed, let $F\in\Face(T)$ have dimension $k+1$ and choose any $\hat{\bzeta}\in F$. 
		
		If $\hat{\bzeta}\notin \relint(F)$, then there exists a $F'\in\Facet(F)$ such that $\hat{\bzeta}\in F'$, and since $\dim(F')=k$, the claim follows from the induction hypothesis and $\alpha_{k+1}<\alpha_k$. 
		
		We therefore focus on the case where $\hat{\bzeta}\in \relint(F)$.
		Indeed, if $\Delta_\Xi(\hat{\bzeta})\geq \alpha_{k+1}\lambda^{-1}\epsilon$, then Lemma~\ref{lemma:coneshrinkege} implies that
		\begin{align*}B\left(\hat{\bzeta},\alpha_{k+1}{\epsilon}\right)\cap \Xi&= \left\{\lambda \bzeta + (1-\lambda) \hat{\bzeta}: \bzeta\in  B(\hat{\bzeta},\alpha_{k+1}\lambda^{-1}{\epsilon})\cap\Xi\right\}\\
		&\subseteq \left\{\lambda \bzeta + (1-\lambda) \hat{\bzeta}: \bzeta\in  B(\hat{\bzeta},\bar{\epsilon})\cap\Xi\right\}
		,\end{align*}
		where the last inclusion follows from $\alpha_{k+1}\lambda^{-1}\epsilon\leq \bar{\epsilon}$.
		Otherwise, $\Delta_\Xi(\hat{\bzeta})< \alpha_{k+1}\lambda^{-1}\epsilon$, and it follows from Definition~\ref{defn:radius_point} that there exists a constraint $i\in [\tilde{m}] \setminus I_\Xi(\hat{\bzeta})$ such that 
		\begin{equation}\label{eq:dist_H_z}
		\dist(\hat{\bzeta},\mathcal{H}_i)=\Delta_\Xi(\hat{\bzeta})<\alpha_{k+1}\frac{\epsilon}{\lambda}. 
		\end{equation}
		
		We first show that
		\begin{align}
		\mathcal{H}_i\cap F \neq \emptyset. \label{line:mp:H_F_nonempty}
		\end{align}
		Suppose for the sake of contradiction that \eqref{line:mp:H_F_nonempty} was false. 
		Then, 
		\begin{align}
		\alpha_{k+1}\frac{\epsilon}{\lambda} &=\left(\frac{\lambda}{2{\theta}_{\Xi}^T}\right)^{k+1-\underline{\text{dim}}(T)}\min\{\Delta_\Xi(T),\bar{\epsilon}\} \notag \\ 
		&< \min\{\Delta_\Xi(T),\bar{\epsilon}\}\notag \\ 
		&\le \Delta_\Xi(T) \notag \\
		&\le \Delta_\Xi(F)\notag \\ 
		&\leq \dist(F,\mathcal{H}_i) \notag \\
		&\leq \dist(\hat{\bzeta},\mathcal{H}_i). \label{line:mp:upper_bounding_distance_hi}
		\end{align}
		The first line follows from the definitions of $\epsilon$ and $\alpha_{k+1}$, the second line follows from Lemma~\ref{lem:mp:angle}, the fourth and fifth lines follows from Definition~\ref{defn:radius_face} and the supposition that $\mathcal{H}_i\cap F = \emptyset$, and line~\eqref{line:mp:upper_bounding_distance_hi} follows from the definition of distance between sets and the fact that $\hat{\bzeta} \in F$. However, \eqref{line:mp:upper_bounding_distance_hi} is a contradiction of \eqref{eq:dist_H_z}, and thus we have shown that \eqref{line:mp:H_F_nonempty} holds.
		
		Since $\hat{\bzeta}\in F$, $i \notin I_\Xi(\hat{\bzeta})$, and $\mathcal{H}_i\cap F \neq \emptyset$, 
		the set $F'=\mathcal{H}_i\cap F$ is a facet $F$.
		We define $\breve{\bzeta}$ to be the projection of $\hat{\bzeta}$ onto $F'$. Then,
		\begin{equation}\label{eq:dist_z_Pz} \| \hat{\bzeta}-\breve{\bzeta} \| =\dist(\hat{\bzeta},\mathcal{H}_i\cap F)\leq\theta_i(F) \dist(\hat{\bzeta},\mathcal{H}_i) \le \alpha_{k+1}\frac{\epsilon{\theta}_{\Xi}^T}{\lambda},\end{equation}
		where the equality follows from the definition of projection, the first inequality follows from Definition~\ref{defn:maximal_tangent}, and the second inequality follows from \eqref{eq:dist_H_z} as well as  Definition~\ref{defn:maximal_tangent}. 
		Since $F'\in \Facet(F)$, we obtain that $\dim(F')=k$. Thus, applying the induction hypothesis to $\breve{\bzeta}\in F'$, there exists $\bar{\bzeta}\in F'$ such that
		\begin{equation*}
		B\left(\breve{\bzeta},\alpha_k\epsilon\right)\cap \Xi\subseteq
		\left\{\lambda \bzeta + (1-\lambda) \bar{\bzeta}: \bzeta\in  B(\bar{\bzeta},\bar{\epsilon})\cap\Xi\right\}.
		 \end{equation*}
		Moreover, for any 
		\begin{align}
		\bzeta\in B\left(\hat{\bzeta},\alpha_{k+1}\epsilon\right), \label{line:mp:choice_of_xi_in_ball}
		\end{align}
		we observe that
		\begin{align*}\norm{\bzeta-\breve{\bzeta}}&\leq \norm{\bzeta-\hat{\bzeta}}+\norm{\hat{\bzeta}-\breve{\bzeta}}\leq \alpha_{k+1}\epsilon+\alpha_{k+1}\frac{\epsilon {\theta}_{\Xi}^T}{\lambda}\leq \alpha_{k+1}\frac{2{\theta_{\Xi}^T}}{\lambda}\epsilon=\alpha_{k}\epsilon,
		\end{align*}
		where the first inequality follows from triangle inequality, the second inequality follows from \eqref{line:mp:choice_of_xi_in_ball} and \eqref{eq:dist_z_Pz}, and the third inequality follows from $\lambda \le 1$ and Lemma~\ref{lem:mp:angle}. 
		Therefore, we have shown that
		\begin{equation*}
		B\left(\hat{\bzeta},\alpha_{k+1}\epsilon\right)\subseteq B\left(\breve{\bzeta},\alpha_k\epsilon\right) \subseteq 
		\left\{\lambda \bzeta + (1-\lambda) \bar{\bzeta}: \bzeta\in  B(\bar{\bzeta},\bar{\epsilon})\cap\Xi\right\},
		\end{equation*} 
		which concludes the proof. 
	\end{itemize}
	\halmos
\end{proof}

\section{Performance Guarantees of Two-Stage Sample Robust Optimization}\label{appx:finite_sample}
In this appendix, we provide a review of probabilistic guarantees for two-stage distributionally robust optimization with the type-$\infty$ Wasserstein ambiguity set (Problem~\eqref{prob:mp:sro}), and demonstrate a specific two-stage problem (Example~\ref{ex:simple}) in which the first-stage decisions obtained from Problem~\eqref{prob:mp:sro} can provably outperform those obtained from alternative data-driven approaches.

In the context of single- and two-stage problems, a primary motivation in the literature for using Wasserstein-based distributionally robust optimization 
is finite-sample probabilistic guarantees. At the center of these probabilistic guarantees for the optimization problems are measure concentration results for the {empirical probability distribution}. 
When the underlying probability distribution ${\Prb}$  satisfies a certain light-tail assumption, \cite{fournier2015rate} show that the empirical probability distribution  $\widehat{\Prb}_N$ concentrates around the true probability distribution under the type-$p$ Wasserstein distance for any $p \in [1,\infty)$. Similar measure concentration results have been established for the type-$\infty$ Wasserstein distance under different probabilistic assumptions \citep{trillos2014rate,liu2019rate}. One such result is the following: 

\begin{theorem}[{\citealp[Theorem 1.1]{trillos2014rate}}]\label{thm:measure_concentration}
	Assume that the probability distribution of $\bxi \in \R^d$ has a density function $\rho:\bar{\Xi}\rightarrow [ 0,\infty)$, where $\bar{\Xi} \subseteq \Xi \subseteq \R^d$ is an open, connected, and bounded set with a Lipschitz boundary, and there exists a constant $\lambda\ge1$ such that $1 / \lambda \leq \rho(\bzeta)\leq \lambda$ for all $\bzeta \in \bar{\Xi}.$ 
	Then for every fixed $\alpha>2$, 
	\begin{align*}
	\Prb^\infty\left(\mathsf{d}_\infty (\widehat{\Prb}_N, \Prb )  > C \begin{cases}
	\frac{\log(N)^{3/4}}{N^{1/2}},& \textnormal{if } d = 1,\\[3pt]
	\frac{\log(N)^{1/d}}{N^{1/d}},&\textnormal{if } d \ge 2\end{cases} \right) = \mathcal{O}(N^{-\frac{\alpha}{2}}),
	\end{align*}
	where $C$ is a constant which depends only on $\alpha$, $\bar{\Xi}$, and $\lambda$, and $\mathsf{d}_\infty(\cdot,\cdot)$ is the type-$\infty$ Wasserstein distance. 
\end{theorem}

Following identical arguments as \citet[Theorem 3.5]{esfahani2018data}, the above measure concentration result directly implies the following finite-sample guarantee for Problem~\eqref{prob:mp:sro}:

\begin{corollary}[Finite-sample guarantee] \label{thm:finite_sample}
	Let the conditions of Theorem~\ref{thm:measure_concentration} hold. 
	Then for any $k>1$ and $\alpha>1$, there exists $\kappa>0$ and $c>0$ such that setting $\epsilon_N= \kappa N^{-\frac{1}{k\max\{d,2\}}}$ implies the following guarantee for all $N\in\N$:
	$$\Prb^\infty\left(\hat{V}^{\textnormal{SRO}}_N(\bx)\ge V^*(\bx),\;\; \forall \bx \in \R^n \right)\ge 1- {cN^{-\frac{\alpha}{2}}}.$$ 
\end{corollary}

The above finite-sample guarantee requires choosing the radius $\epsilon_N$ based on the constant $\kappa$, and the probabilistic guarantee depends on the constant $c$. 
In general, the values of these constants 
are dependent on properties of the underlying probability distribution, which are unknown in practice. Nonetheless, Corollary~\ref{thm:finite_sample} provides an explicit rate for decreasing the robustness parameter $\epsilon_N$ as more data is obtained, which can provide insight when choosing this parameter in practice. 

Moreover, Theorem~\ref{thm:measure_concentration} implies that, even for an arbitrary choice of $\kappa$, if the appropriate rate for $\epsilon_N$ is used, Problem~\eqref{prob:mp:sro} will still provide an upper bound approximation of the stochastic problem for all sufficiently large $N \in \N$ almost surely: 

\begin{corollary}\label{thm:smooth_result}
	Let the conditions of Theorem~\ref{thm:measure_concentration} hold. 
	If $\epsilon_N= \kappa N^{-\frac{1}{k\max\{d,2\}}}$ for any fixed $\kappa > 0$ and $k>1$, then there exists a random variable $\bar{N}$ with $\Exp[\bar{N}] < \infty$ such that $V^{\textnormal{SRO}}_N(\bx) \ge V^*(\bx)$ for all $N \ge \bar{N}$ and for all $\bx \in \R^n$. 
\end{corollary}
\begin{proof}{Proof.}
	Consider any fixed $\kappa > 0$ and $k > 1$, and assume that we have chosen the robustness parameter as $\epsilon_N= \kappa N^{-\frac{1}{k\max\{d,2\}}}$ for each $N \in \N$.  Define the following random variables for each $N \in \N$: 
	\begin{align*}
	Y_N&\triangleq \mathbb{I}\left \{\text{there exists } \bx\in \R^n \text{ such that } \hat{V}^{\textnormal{SRO}}_N(\bx)< V^*(\bx) \right\}.
	\end{align*}
	Finally, define the random variable: 
	\begin{align*}
	\bar{N} &\triangleq 1 +  \max_{N \in \N} \left \{ N: Y_N  = 1 \right \}. 
	\end{align*}
	Fixing $\alpha = 6$, it follows readily from Theorem~\ref{thm:measure_concentration} that there exist deterministic constants $c>0$, $C>0$, $\tilde{N}\in\N$ such that the following holds for all $N\geq \tilde{N}$: 
	\begin{align}\label{eq:A_N_prob}
	\Prb^{\infty}(Y_N = 1)\leq \Prb^\infty(\mathsf{d}_\infty (\widehat{\Prb}_N, \Prb )> \epsilon_N)\leq \Prb^\infty\left(\mathsf{d}_\infty (\widehat{\Prb}_N, \Prb )  > C \begin{cases}
	\frac{\log(N)^{3/4}}{N^{1/2}},& \textnormal{if } d = 1,\\[3pt]
	\frac{\log(N)^{1/d}}{N^{1/d}},&\textnormal{if } d \ge 2\end{cases} \right) \le cN^{-3}.
	\end{align}
	Therefore, 
	\begin{align*}
	\Exp\left[ \bar{N} \right]=\sum_{N=0}^\infty \Prb^\infty \left(\bar{N}> N\right)=\sum_{N=1}^\infty \Prb^\infty \left(\text{there exists } k\geq N \text{ such that }Y_k = 1 \right) &\le \sum_{N=1}^\infty \sum_{k=N}^\infty \Prb^\infty(Y_k = 1) < \infty. 
	\end{align*}
	Indeed, the first equality holds since the random variable $\bar{N}$ is non-negative, the second equality follows from the definition of $\bar{N}$, the first inequality follows from the union bound, and the final inequality follows directly from \eqref{eq:A_N_prob} . 
	%
	%
	\Halmos
\end{proof}

The above result has important implications regarding \emph{solution quality}. Indeed,  Corollary~\ref{thm:smooth_result} implies that, for all sufficiently large datasets, any first-stage decision which is feasible for Problem~\eqref{prob:mp:sro} will be feasible for the underlying stochastic problem. As we will demonstrate shortly in Example~\ref{ex:simple}, such a result does not hold in general for SAA. Moreover, in contrast to two-stage distributionally robust optimization with the type-$p$ Wasserstein ambiguity set for $p \in [1,\infty)$, this upper bound is not conservative  for problems without relative complete recourse. The following Example~\ref{ex:simple} highlights these desirable features of Problem~\eqref{prob:mp:sro}. 
\begin{example} \label{ex:simple}
	Consider the two-stage optimization problem
	\begin{equation} \label{prob:example_simple}
	\begin{aligned}
	&\underset{x \in \R}{\text{minimize}} &&x +  \Exp[Q(x,\xi)], \quad \quad Q(x,\xi) \triangleq \min_y \{ 0: x \ge \xi \},
	\end{aligned}
	\end{equation}
	where $\xi$ is uniformly distributed between $0$ and $1$. In essence, the above example is a single-stage linear optimization problem, where the second-stage cost equals $0$ if $x \ge \xi$ and $\infty$ otherwise. The optimal cost and first-stage decision are $v^* = x^* = 1.$ 
	We assume that the true distribution of the random variable is unknown, and our only knowledge comes from historical samples $\hat{\xi}{}^1,\ldots,\hat{\xi}{}^N$ of $\xi$ and knowledge that the support of the random variable is contained in $\Xi = [0,2]$. 
	
	Applying Problem~\eqref{prob:mp:saa} to this example results in the following formulation
	\begin{equation*}
	\begin{aligned}
	&\underset{x \in \R}{\textnormal{minimize}}&& x \\
	&\textnormal{subject to}&& x \ge \hat{\xi}{}^i \quad \forall i \in [N],
	\end{aligned}
	\end{equation*}
	with optimal cost and first-stage decisions given by 
	$\hat{v}_N^{\textnormal{SAA}} = \hat{x}_N^{\textnormal{SAA}} = \max_{i \in [N]} \hat{\xi}{}^i.$ 
	It is readily observed that the optimal cost and optimal first-stage decisions converge to those of Problem~\eqref{prob:example_simple} as $N \to \infty$. However, we observe that $\hat{x}_{N+1}^{\textnormal{SAA}} > \hat{x}_N^{\textnormal{SAA}}$ if and only if  $Q(\hat{x}_N^{\textnormal{SAA}}, \hat{\xi}{}^{N+1}) = \infty,$ and thus 
	$Q(\hat{x}{}^{\textnormal{SAA}}_N, \hat{\xi}{}^{N+1}) = \infty$ for infinitely many $N \in \N$ almost surely. 
	
	To obtain first-stage decisions with better out-of-sample feasibility, an alternative approach is to restrict the first-stage decisions to those for which the second-stage problem is feasible ($Q(\bx,\bzeta) < \infty$) for every realization in $\Xi$. Examples include two-stage distributionally robust optimization using the $p$-Wasserstein ambiguity set for $p \in [1,\infty)$ and positive radius; see \cite{bertsimas2018multistage}.
	This approach results in
	\begin{equation*}
	\begin{aligned}
	&\underset{x \in \R}{\textnormal{minimize}}&& x \\
	&\textnormal{subject to}&& x \ge \zeta \quad  \forall \zeta \in \Xi,
	\end{aligned}
	\end{equation*}
	which produces an optimal cost and first-stage decisions of $\hat{v}_N^{\textnormal{Feas}} = \hat{x}_N^{\textnormal{Feas}} = 2.$ In contrast to SAA, this approach provides a guarantee that the resulting first-stage decisions are always feasible for Problem~\eqref{prob:example_simple}. However, such a strong guarantee comes with a downside of poor average performance. 
	
	Sample robust optimization offers a tradeoff between the above approaches. In this example, Problem~\eqref{prob:mp:sro} takes the form 
	\begin{equation*}
	\begin{aligned}
	&\underset{x \in \R}{\textnormal{minimize}}&& x \\
	&\textnormal{subject to}&& x \ge \zeta \quad \forall \zeta \in \cup_{i=1}^N \mathcal{U}^i_N,
	\end{aligned}
	\end{equation*}
	and thus the optimal cost and first-stage decisions are $\hat{v}_N^{\textnormal{SRO}} = \hat{x}_N^{\textnormal{SRO}} = \min \{\max_{i \in [N]} \hat{\xi}{}^{i} + \epsilon_N, 2 \}.$ As prescribed by Corollary~\ref{thm:smooth_result}, we choose the radius to be $\epsilon_N = N^{-\frac{1}{2}}$. Similar to SAA, Theorem~\ref{thm:mp:conv} and our choice of $\epsilon_N$ guarantee that sample robust optimization finds first-stage decisions with good average performance which are asymptotically optimal as $N \to \infty$. In contrast, Corollary~\ref{thm:smooth_result} indicates that sample robust optimization offers much stronger feasibility guarantees than SAA, as $\hat{x}_N^{\textnormal{SRO}}$ will be feasible for Problem~\eqref{prob:example_simple} for all sufficiently large $N \in \N$. Thus, in this example, sample robust optimization has asymptotically optimal average performance while simultaneously offering attractive feasibility guarantees. \halmos
\end{example}

\end{APPENDICES}

\end{document}